\documentclass[a4paper]{amsart}
\usepackage[utf8]{inputenc}
\usepackage{tikz-cd}
\usepackage{amsmath}
\usepackage{verbatim}
\usepackage{amssymb}
\usepackage{amsthm}
\usepackage{hyperref}

\usepackage[bb=dsserif]{mathalpha}

\calclayout

\newtheorem*{theorem*}{Theorem}

\usepackage{ulem}

\usetikzlibrary{calc}
\usetikzlibrary{decorations.pathmorphing}
\tikzset{curve/.style={settings={#1},to path={(\tikztostart)
    .. controls ($(\tikztostart)!\pv{pos}!(\tikztotarget)!\pv{height}!270:(\tikztotarget)$)
    and ($(\tikztostart)!1-\pv{pos}!(\tikztotarget)!\pv{height}!270:(\tikztotarget)$)
    .. (\tikztotarget)\tikztonodes}},
    settings/.code={\tikzset{quiver/.cd,#1}
        \def\pv##1{\pgfkeysvalueof{/tikz/quiver/##1}}},
    quiver/.cd,pos/.initial=0.35,height/.initial=0}

\tikzset{tail reversed/.code={\pgfsetarrowsstart{tikzcd to}}}
\tikzset{2tail/.code={\pgfsetarrowsstart{Implies[reversed]}}}
\tikzset{2tail reversed/.code={\pgfsetarrowsstart{Implies}}}
\tikzset{no body/.style={/tikz/dash pattern=on 0 off 1mm}}

\DeclareMathOperator{\Top}{top}
\DeclareMathOperator{\im}{im}

\DeclareMathOperator{\field}{k}
\DeclareMathOperator{\op}{op}
\DeclareMathOperator{\rad}{rad}

\DeclareMathOperator{\dual}{D}
\DeclareMathOperator{\tw}{tw}
\DeclareMathOperator{\modu}{mod}
\DeclareMathOperator{\Ext}{Ext}

\DeclareMathOperator{\End}{End}
\DeclareMathOperator{\Hom}{Hom}
\DeclareMathOperator{\add}{add}

\DeclareMathOperator{\Sim}{Sim}

\DeclareMathOperator{\id}{id}

\DeclareMathOperator{\dualbar}{\mathbb{B}D}

\DeclareMathOperator{\Dual}{D}

\DeclareMathOperator{\tr}{tr}
\DeclareMathOperator{\twmod}{twmod}

\DeclareMathOperator{\complexes}{C}

\DeclareMathOperator{\chara}{char}

\DeclareMathOperator{\Binary}{Binary}

\DeclareMathOperator{\filt}{F}
\newtheorem{theorem}{Theorem}[section]
\newtheorem{definition}[theorem]{Definition}
\newtheorem{example}[theorem]{Example}
\newtheorem{corollary}[theorem]{Corollary}
\newtheorem{lemma}[theorem]{Lemma}
\newtheorem{remark}[theorem]{Remark}
\newtheorem{proposition}[theorem]{Proposition}
\title{Uniqueness up to Inner Automorphism of Regular Exact Borel Subalgebras}
\author{Anna Rodriguez Rasmussen}
\DeclareMathOperator{\trunc}{trunc}
\DeclareMathOperator{\spann}{span}
\DeclareMathOperator{\Ainfunit}{A-\infty_{\mathbb{L}}}

\begin{document}
\maketitle
\begin{abstract}
    In \cite{KKO}, Külshammer, König and Ovsienko proved that for any quasi-hereditary algebra $(A,\leq_A)$ there exists a Morita equivalent quasi-hereditary algebra $(R, \leq_R)$ containing a basic exact Borel subalgebra $B$. The Borel subalgebra $B$ constructed in \cite{KKO} is in fact a regular exact Borel subalgebra as defined in \cite{BKK}.
Later, Conde \cite{Conde} showed that given a quasi-hereditary algebra $(R,\leq_R)$ with a basic regular exact Borel subalgebra $B$ and a Morita equivalent quasi-hereditary algebra $(R',\leq_{R'})$ with a basic regular exact Borel subalgebra $B'$, the algebras $R$ and $R'$ are isomorphic, and Külshammer and Miemietz \cite{Miemietz} showed that there is even an isomorphism $\varphi:R\rightarrow R'$ such that $\varphi(B)=B'$.\\
In this article, we show that if $R=R'$, then $\varphi$ can be chosen to be an inner automorphism. Moreover, instead of just proving this for regular exact Borel subalgebras of quasi-hereditary algebras, we generalize this to an appropriate class of subalgebras of arbitrary finite-dimensional algebras.
As an application, we show that if $(A, \leq_A)$ is a finite-dimensional algebra and $G$ is a finite group acting on $A$ via automorphisms, then under some natural compatibility conditions, there is a  Morita equivalent quasi-hereditary algebra $(R, \leq_R)$ with a basic regular exact Borel subalgebra $B$ such that $g(B)=B$ for every $g\in G$.
\end{abstract}
\section{Introduction}
 Let $A$ be a finite-dimensional algebra over an algebraically closed field $\field$. The Wedderbun-Malcev theorem \cite{Wedderburn, Maltsev} states that there is a maximal semisimple subalgebra $L$ of $A$, which is unique up to conjugation.
Similarly, if $G$ is a connected linear algebraic group over $\field$ then $G$ has a Borel subgroup $B$, and by the Borel-Mozorov theorem  \cite{Borel, Morozov} this subgroup is unique up to conjugation. Correspondingly, if $\field$ has characteristic zero, any finite-dimensional semi-simple Lie algebra over $\field$ has a Borel subalgebra, which is again unique up to inner automorphism \cite[Ch. VI, Theorem 5]{serre}.\\
An important representation-theoretic object associated to such a Lie algebra $\mathfrak{g}$ is category $\mathcal{O}$, the highest weight category consisting of finitely generated $\mathfrak{g}$-modules which admit additional finiteness conditions on the action of its Borel and Cartan subalgebras. Category $\mathcal{O}$ decomposes into blocks, and each of its blocks is equivalent to the category of finite-dimensional modules over a quasi-hereditary algebra. Both the notion of a highest weight category and the notion of a quasi-hereditary algebra are due to Cline, Parshall and Scott \cite{CPS}.
In \cite{Koenig}, König introduced an analogon of Borel subalgebras of finite-dimensional Lie algebras for  quasi-hereditary algebras, so called exact Borel subalgebras.\\
However, in contrast to the setting of finite-dimensional Lie algebras, both existence and uniqueness fail in general \cite[Example 2.3]{Koenig}, even without the additional assumption of regularity.
Despite this, Külshammer, König and Ovsienko were able to show in \cite{KKO} that for any quasi-hereditary algebra $(A, \leq_A)$ there is a Morita equivalent quasi-hereditary algebra $(R, \leq_R)$ with a regular exact Borel subalgebra $(B, \leq_B)$. They were even able to give a procedure to construct $(R, \leq_R)$ and $(B, \leq_B)$ as well as the embedding $B\rightarrow R$, and to show that under this construction $B$ is basic, i.e. every simple $B$-module is one-dimensional.
The definition of a Borel subalgebra as well as the existence result were subsequently generalized by Goto \cite{Goto} to standardly stratified algebras and by Bautista--Perez--Salmerón \cite{BPS} to finite-dimensional algebras with strict homological systems. However, in the former case, the corresponding theorem yields not a regular exact Borel subalgebra, but a homological one.\\
Later, Külshammer and Miemietz established in \cite{Miemietz} that the pair  $(B, R)$ is unique up to isomorphism, that is, if $(R', \leq_{R'})$ is another quasi-hereditary algebra Morita equivalent to $(A, \leq_A)$ with basic regular exact Borel subalgebra $B'\subseteq R'$, then there is an isomorphism $\varphi:R\rightarrow R'$ such that $\varphi(B)=B'$.
Independently, a result obtained by Conde \cite{Conde} implies the uniqueness up to isomorphism of $R$ without requiring as much additional theory.
Nevertheless, it remained an open question whether, if $R=R'$, $\varphi$ could be chosen as an inner automorphism, that is, if $B\subseteq R$ was unique up to conjugation.
The main aim of this article is to assert this by proving the following theorem:
\begin{theorem*}(Theorem \ref{theorem_conjugation})\\
    Let $A$ be a finite-dimensional algebra and let $B, B'$ be two basic regular exact subalgebras of $A$ with simple modules $\{L_1^B, \dots , L_n^B\}=\Sim(B)$ and $\{L_1^{B'}, \dots , L_n^{B'}\}=\Sim(B')$ such that for every $1\leq i\leq n$ we have
\begin{align*}
	\Delta_i^A:=A\otimes_B L_i^B\cong A\otimes_{B'}L_i^{B'}.
\end{align*} Then there is an invertible element $a\in A$ such that $aBa^{-1}=B'$.
\end{theorem*}
Here, a regular exact subalgebra of an arbitrary finite-dimensional algebra is a natural generalization of a regular exact Borel subalgebra of a quasi-hereditary algebra, proposed by Külshammer in \cite{surveyjulian} under the name ``bound quiver''. This term encompasses in particular regular exact Borel subalgebras of left/right standardly stratified algebras as well as regular exact Borel subalgebras of finite-dimensional algebras with homological systems. It should be remarked, however, that the existence result by Goto \cite{Goto} only yields homological exact Borel subalgebras, for which we do not obtain a uniqueness result.
Additionally, note that in the case where the modules $\Delta_i$ are the projective modules of $A$, this theorem recovers the uniqueness statement of the Wedderburn--Malcev theorem.\\
The proof given here is inspired by the proof of the existence result in \cite{KKO}. It makes extensive use of A-infinity algebras as defined by Stasheff \cite{Stasheff} and of the category of so-called twisted modules over an A-infinity algebra.
Twisted modules (or, more generally complexes) over a dg-algebra/category go back to a definition by Bondal and Kapranov \cite{BondalKapranov2}, and were generalized by  Lefèvre-Hasegawa to twisted modules/complexes over an A-infinity algebra/category \cite{LefHas}. As explained in \cite{BondalKapranov2} and \cite{LefHas}, the homology $H^0(\tw(\mathcal{A}))$ in degree zero of the A-infinity category of twisted complexes over an A-infinity algebra $\mathcal{A}$ can be interpreted as the full subcategory of the derived category generated by the regular module under sums, summands, shifts and triangles, and the homology $H^0(\twmod(\mathcal{A}))$ of the category of twisted modules as the full subcategory of the derived category generated by the regular module under sums, summands and extensions.\\
In the reconstruction theorem of Keller \cite[7.7]{Keller}, the A-infinity category of twisted modules is used in order to reconstruct a basic finite-dimensional algebra $B$ from the A-infinity algebra $\Ext_B^*(L^B, L^B)$, where $L^B$ is the semisimple part of $B$. This is somewhat analogous to the reconstruction of a Koszul algebra $B$ from its Yoneda algebra $\Ext_B^*(L^B, L^B)$ in classical Koszul duality, see \cite{Koszul}, whence it is sometimes referred to as A-infinity Koszul duality. \\
Similarly to this, in \cite{KKO}, the A-infinity subalgebra 
\begin{align*}
    L\oplus \Ext_A^{>0}(\Delta^A, \Delta^A)
\end{align*}
of the A-infinity algebra $ \Ext_A^*(\Delta^A, \Delta^A)$ is used to construct a regular exact Borel subalgebra $B$ of $A$.\\
Thus, given two basic regular exact subalgebras $B, B'$ with simples $\Sim(B)=\{L_1^B, \dots, L_n^B\}$ and  $\Sim(B')=\{L_1^{B'}, \dots, L_n^{B'}\}$ such that \begin{align*}
    \Delta^A_i:=A\otimes_B L_i^B\cong A\otimes_{B'}L_i^{B'},
\end{align*}
the strategy of our proof is to construct an $A$-infinity isomorphism 
\begin{align*}
	h:\Ext^*_B(L^B,L^B)\rightarrow \Ext^*_{B'}(L^{B'}, L^{B'})
\end{align*}
which makes the diagram 
\begin{equation}\label{diagram_2}
  \begin{tikzcd}[ampersand replacement=\&]
	{\Ext_B^*(L^B, L^B)} \\
	\&\& {\Ext^*_A(\Delta^A,\Delta^A)} \\
	{\Ext_{B'}^*(L^{B'}, L^{B'})}
	\arrow["f", from=1-1, to=2-3]
	\arrow["{f'}"', from=3-1, to=2-3]
	\arrow["h"', from=1-1, to=3-1]
\end{tikzcd}
\end{equation} commute, where $f$ and $f'$ are the $A$-infinity homomorphisms induced by the induction functors; and then to use a version of Keller's reconstruction theorem to see that this induces an equivalence $H:\modu B\rightarrow \modu B'$ and a diagram \label{diagramX}
\begin{equation}\label{diagram_final}
\begin{tikzcd}[ampersand replacement=\&]
	{\modu B} \\
	\&\& {\modu A} \\
	{\modu B'}
	\arrow["{A\otimes_B -}", from=1-1, to=2-3]
	\arrow["{A\otimes_{B'}-}"', from=3-1, to=2-3]
	\arrow["F"', from=1-1, to=3-1]
\end{tikzcd}
\end{equation}
which commutes up to natural isomorphism.
By considering the endomorphism algebras of the projective generators $B$, $B'$ and $A\otimes_B B\cong A\cong A\otimes_{B'}B'$, this natural isomorphism then gives rise to an $a\in A^\times$ such that $aBa^{-1}=B'$.
In the course of this proof, we establish commutativity up to natural isomorphism of the diagram 
\begin{equation}\label{diagram_1}
	\begin{tikzcd}[ampersand replacement=\&]
		{\modu B} \&\& {F(\Delta)} \\
		{H^0(\twmod(\Ext^*_B(L_B, L_B))} \&\& {H^0(\twmod(\Ext^*_A(\Delta_A, \Delta_A))}
		\arrow["{A\otimes_B-}", from=1-1, to=1-3]
		\arrow[from=2-3, to=1-3]
		\arrow[from=2-1, to=1-1]
		\arrow["{H^0(\twmod(f))}", from=2-1, to=2-3]
	\end{tikzcd}
	\end{equation}
In the final two sections, we give two additional applications of this commutative diagram, both to quasi-hereditary algebras.\\
In the first application, we deduce a stronger version of our main theorem in the quasi-hereditary case under the additional assumption that $A$ is basic. More precisely, we show that in this setting it suffices for one of the algebras $B$, $B'$ to be regular:
\begin{theorem*}(Theorem \ref{thm_strong})\\
    Let $(A, \leq_A)$ be a basic finite-dimensional  quasi-hereditary algebra, $B$ be a regular exact Borel  subalgebra of $A$ and $B'$ be an exact Borel subalgebra of $A$. Then there is an $a\in A^{\times}$ such that $aBa^{-1}=B'$.
\end{theorem*}
In particular, $B'$ is regular exact, which is a result that was proven independently in a different way by Conde and König, in a preprint that is yet to appear. Note also that by \cite[Theorem 5.2]{Conde} a basic quasi-hereditary algebra $(A, \leq_A)$ admits a regular exact Borel subalgebra if and only if the radicals of the standard modules admit costandard filtrations.\\
The second is an application to skew group algebras of quasi-hereditary algebras. If $A$ is a finite-dimensional algebra and $G$ is a group acting on $A$ via automorphisms, then one can define the skew group algebra
\begin{align*}
    A*G:=A\otimes \field G
\end{align*}
as a $\field$-vector space together with the multiplication 
\begin{align*}
    (a\otimes g)\cdot (a'\otimes g'):=a g(a')\otimes gg'.
\end{align*}
Skew group algebras are a natural construction in representation theory, and have been studied extensively e.g. in \cite{ReitenRiedtmann}, \cite{demonet}, \cite{lemeur2}.
In previous work \cite{mypaper}, we have considered the setting where
$(A,\leq_A)$ is a quasi-hereditary algebra and $G$ is a finite group such that $\chara(k)$ does not divide $|G|$, acting on $A$ via automorphisms. Under a natural compatibility condition of $\leq_A$ with this action, see \cite[Definition 3.1]{mypaper}, we have established that there is an induced partial order $\leq_{A*G}$ on the simple modules of the skew group algebra $A*G$ such that $(A*G,\leq_{A*G})$ is quasi-hereditary. Moreover, if $B$ is a regular exact Borel subalgebra of $(A, \leq_A)$ such $g(B)=B$ for all $g\in G$, then $B*G$ is a regular exact Borel subalgebra of $(A*G, \leq_{A*G})$.\\
Since  \cite{KKO} provides an existence result for regular exact Borel subalgebras, this naturally raises the question whether, in the above setting, this can be extended to an existence result of regular exact Borel subalgebras invariant as a set under the group action. We answer this question in the affirmative by proving the following theorem:
\begin{theorem*}(Theorem \ref{thm_skew})\\
    Let $(A,\leq_A)$ be a quasi-hereditary algebra and $G$ be a finite group such that $\chara(k)$ does not divide $|G|$, acting on $A$ via automorphisms such that $\leq_A$ is compatible with this action as in \cite{mypaper}. Then there is a Morita equivalent quasi-hereditary algebra $(R, \leq_R)$ with  a $G$-action such that $R$ has a regular exact Borel subalgebra $B$ with $g(B)=B$ for all $g\in B$ and such that the Morita equivalence
    \begin{align*}
        F:\modu A\rightarrow\modu R
    \end{align*}
    is $G$-equivariant.
\end{theorem*}
The structure of the article is as follows:
In Section \ref{section_notation}, we introduce our notation, and in Section \ref{section_prelim} we give an introduction to quasi-hereditary algebras and their exact Borel subalgebras.
In Section \ref{section_dg} we recall twisted modules for dg-algebras, as well as the functor $\twmod$ assigning to a dg algebra the dg-category of twisted modules. We also recall an equivalence due to Bondal and Kapranov \cite{BondalKapranov} between $H^0(\twmod_L\End_B^*(P^\cdot(L^B)))$ and $\modu B$, where $P^\cdot(L^B)$ is a projective resolution of $B$, and establish a commutative diagram 
\begin{equation}\label{diagram_1.1}
  \begin{tikzcd}[ampersand replacement=\&]
	{H^0(\twmod(\End_B(P^\cdot(L^B))} \&\& {H^0(\twmod(\End_A(A\otimes_BP^\cdot(L^B))} \\
	{\modu B} \&\& {F(\Delta^A)}
	\arrow[from=1-3, to=2-3]
	\arrow[from=1-1, to=2-1]
	\arrow["{A\otimes_B-}", from=2-1, to=2-3]
	\arrow["{H^0(\twmod(g))}", from=1-1, to=1-3]
\end{tikzcd}
\end{equation}
\\
where $g:\End_B(P^\cdot(L^B))\rightarrow \End_A(A\otimes P^\cdot(L^B)), \varphi\mapsto \id_A\otimes \varphi$.\\
In Section \ref{section_ainfty}, we recall (strictly unital) A-infinity algebras and categories. Following \cite{LefHas}, we give the definition twisted modules over an A-infinity algebra and the extension of the functor $\twmod$ to the A-infinity setting, both due to Lefèvre-Hasegawa \cite{LefHas}. Moreover, we recall from \cite[Lemma 3.25]{Seidel} that if $f$ is a quasi-isomorphism of $A$-infinity algebras, then $H^0(\twmod(f))$ is an equivalence. Thus we obtain a commutative diagram as in Diagram \eqref{diagram_1}, where the vertical arrows are equivalences.\\
In Section \ref{section_trunc}, we discuss truncation for augmented A-infinity algebras and use this to construct the A-infinity isomorphism in Diagram \eqref{diagram_2}.
In Section \ref{section_proofmain}, we synthesize the previous sections to prove the main result, Theorem \ref{theorem_conjugation}.\\
In Section \ref{section_strong}, we give an alternative proof of a slightly stronger version, Theorem \ref{thm_strong}, of our result in the case that $A$ is basic. This theorem also follows from a result announced by Conde and König in conjunction with our main theorem. Additionally, a stronger version of this theorem was claimed in \cite{yuehui_zhang}; however, in the proof given therein, some statements are problematic, see \ref{example_yuehui}.\\
In  Section \ref{section_skew}, we prove our application, Theorem \ref{thm_skew}, concerning the existence of $G$-invariant regular exact Borel subalgebras for quasi-hereditary algebras with a group action.
\section{Notation}\label{section_notation}
Throughout, let $\field$ be an algebraically closed field, and let $L\cong \field^n$. All algebras are supposed to be finite-dimensional algebras over $\field$ and all modules are supposed to be finite-dimensional left modules. If $A$ is an algebra, we denote by $\Sim(A)$ a set of representatives of the isomorphism classes of the simple $A$-modules.\\ For any $A$-module $M$ we denote by $P(M)$ a projective cover of $M$. For any collection $M=(M_i)_{i\in I}$ of indecomposable $A$-modules we denote by $\filt(M)$ the full subcategory of $\modu A$ consisting of modules which admit a filtration by $M_i$, $i\in I$; that is, $N\in \filt(M)$ if and only if there is a chain
\begin{align*}
    (0)=N_0\subseteq N_1\subseteq \dots \subseteq N_{m-1}\subseteq N_m=N
\end{align*}
of $A$-submodules of $N$ such that for all $1\leq j\leq m$ there is $i=i_j\in I$ such that $N_j/N_{j-1}\cong M_i$.\\
For $X\in \modu L\otimes L^{\op}$ we denote by $\Dual X$ the $L$-$L$-bimodule given by the $\field$-dual $\Dual X:=\Hom_{\field}(X, \field)$.\\
We denote by $\modu^{\mathbb{Z}}L$ the category of degree-wise finite-dimensional $\mathbb{Z}$-graded $L$-modules and by  $\modu^{\mathbb{Z}}L\otimes L^{\op}$ the category of degree-wise finite-dimensional $\mathbb{Z}$-graded $L$-$L$-bimodules. For  $X=\bigoplus_{n\in \mathbb{Z}}X_n\in \modu^{\mathbb{Z}}L$ or $X=\bigoplus_{n\in \mathbb{Z}}X_n\in\modu^{\mathbb{Z}}L\otimes L^{\op}$ and $x\in X_n$ we denote by $|x|:=n$ the degree of $x$. Moreover, we denote by $sX$ the (positive degree) shift of $X$, i.e. $X=sX$ as a non-graded $L$-module resp. $L$-$L$-bimodule and for $x\in X$ homogeneous, we have
\begin{align*}
    |sx|=|x|+1.
\end{align*}
Finally, for $X=\bigoplus_{n\in \mathbb{Z}}X_n\in  \modu^{\mathbb{Z}}L\otimes L^{\op}$, we equip the $L$-$L$ bimodule $\Dual X$ with a  $\mathbb{Z}$-grading via
\begin{align*}
    (\Dual X)_n:=\Dual X_{-n}.
\end{align*}
\section{Preliminaries}\label{section_prelim}

In this section, we give an overview of our main motivation, the study of quasi-hereditary algebras and their exact Borel subalgebras, as well as of the generalized setting in which Theorem \ref{theorem_conjugation} is given.
For an introduction to quasi-hereditary algebras see for example \cite{DlabRingel}.\\
The data of a quasi-hereditary algebra consists of a finite-dimensional algebra $A$ together with a partial order  $\leq_A$ on $\Sim(A)$.
Given such a partial order, one can define the so-called standard modules $\Delta(L)$:
\begin{definition}\label{definition_standardmodules}
    Let $A$ be a finite-dimensional algebra and $\leq_A$ be a partial order on $\Sim(A)$. Then for every $L\in \Sim(A)$ we define
    \begin{align*}
        \Delta(L):=P(L)/\sum_{L'\nleq_A L, f\in \Hom_A(P_{L'}, P_L)}\im(f).
    \end{align*}
    $\Delta(L)$ is called the standard module associated to $L$ with respect to the partial order $\leq_A$.
    We denote by $\Delta:=(\Delta(L))_{L\in \Sim(A)}$ the collection of standard modules.
\end{definition}
In general, different partial orders may give rise to the same set of standard modules. In this case, they are called equivalent. Moreover, a Morita equivalence of  pairs $(A, \leq_A)$, $(R,\leq_R)$ is defined as an equivalence
\begin{align*}
    F:\modu A\rightarrow \modu R
\end{align*}
such that for every $L\in \Sim(A)$
\begin{align*}
    F(\Delta(L))\cong \Delta(F(L)).
\end{align*}
A pair $(A, \leq_A)$ is called quasi-hereditary if the following hold:
\begin{enumerate}
    \item $\End_A(\Delta(L))\cong \field$ for all $L\in \Sim(A)$.
    \item $A\in \filt(\Delta)$.
\end{enumerate}
Prominent examples include Schur algebras, algebras underlying blocks of Bernstein-Gelfand-Gelfand category $\mathcal{O}$ and algebras of global dimension less than or equal to two.
A particularly easy example is the case where $A=\field Q/I$ for some quiver $Q$ without oriented cycles, and the partial order of $A$ is given by 
\begin{align*}
    L\leq_A L' \Leftrightarrow \textup{ there is a path }p:e_L\rightarrow e_{L'}\textup{ in }Q.
\end{align*}
In this case $\Delta_L\cong L$ for every $L\in \Sim(A)$, so that $\filt(\Delta)=\modu A$ and $\End_A(\Delta(L))\cong \field$ by Schur's lemma.
In general, any quasi-hereditary algebra $(A, \leq_A)$ with simple standard modules is Morita equivalent to a pair $(\field Q/I, \leq)$
as described above. Such quasi-hereditary algebras are called directed.\\
If $(A, \leq_A)$ is a quasi-hereditary algebra, then the Dlab-Ringel Reconstruction Theorem  shows that $(A,\leq_A)$ is up to Morita equivalence already determined by the category $\filt(\Delta)$:
\begin{theorem}\cite[Theorem 2]{DlabRingel}
    Let  $(A, \leq_A)$ and  $(R, \leq_R)$ be quasi-hereditary algebras. Suppose there is an equivalence
    \begin{align*}
        F:\filt(\Delta^A)\rightarrow \filt(\Delta^R).
    \end{align*}
    Then $F$ extends to an equivalence
    \begin{align*}
        \overline{F}:\modu A\rightarrow \modu R.
    \end{align*}
\end{theorem}
Hence the subcategory $\filt(\Delta)$ is of particular interest in studying $(A,\leq_A)$.
In \cite{Koenig}, König defined the concept of an exact Borel subalgebra of a quasi-hereditary algebra:
\begin{definition}
    Let $(A,\leq_A)$ be a finite-dimensional algebra. A subalgebra $B$ is called an exact Borel subalgebra if
    \begin{enumerate}
        \item The induction functor  \begin{align*}
        A\otimes_B -:\modu B\rightarrow \modu A
    \end{align*}
    is exact.
    \item There is a bijection $\phi:\Sim(B)\rightarrow\Sim(A)$ such that for all $L\in \Sim(B)$ we have
    \begin{align*}
        A\otimes_B L\cong \Delta(\phi(L)).
    \end{align*}
    \item $(B, \leq_B)$ is directed with respect to the partial order
    \begin{align*}
        L\leq_B L':\Leftrightarrow \phi(L)\leq_A\phi(L').
    \end{align*}
    \end{enumerate}
\end{definition}
One often considers exact Borel subgalgebras with additional properties, see \cite[Definition 3.4]{BKK} and \cite[p. 405]{Koenig}:
\begin{definition}\label{definition_regular}
    \begin{enumerate}
        \item An exact Borel subalgebra $B$ of a quasi-hereditary algebra $(A, \leq_A)$ is called normal if the embedding $\iota:B\rightarrow A$ has a splitting as a right $B$-module homomorphism whose kernel is a right ideal in $A$.
        \item An exact Borel subalgebra $B$ of a quasi-hereditary algebra $(A, \leq_A)$ is called regular if it is normal and for every $n\geq 1$ the maps
        \begin{align*}
            \Ext^n_B(L^B, L^B)\rightarrow   \Ext^n_A(\Delta^A, \Delta^A), [f]\mapsto [\id_A\otimes_B f]
        \end{align*}
        are isomorphisms, where $L^B=\bigoplus_{L_i^B\in \Sim(B)}L_i^B$ and $\Delta^A=A\otimes_B L^B\cong \bigoplus_{L_i^A\in \Sim(A)}\Delta(L_i^A)$.
        \item An exact Borel subalgebra $B$ of a quasi-hereditary algebra $(A, \leq_A)$ is called strong if it contains a maximal semisimple subalgebra of $A$.
    \end{enumerate}
\end{definition} 
In particular, regular is a useful condition in order to transfer homological data from $\modu B$ to $\filt(\Delta)$ and the other way around.
Additionally, regular exact Borel subalgebras are of particular interest since the exact Borel subalgebras constructed in \cite{KKO} are regular.\\
While quasi-hereditary algebras and their regular exact Borel subalgebras are the main motivation for this article, it was observed in \cite{surveyjulian} that there is a natural generalization of this to arbitrary finite-dimensional algebras, and our results are formulated in this setting. Note that this definition also encompasses regular exact Borel subalgebras of standardly stratified algebras and of finite-dimensional algebras with homological systems as defined by Goto \cite{Goto} and Bautista--Perez--Salmerón \cite{BPS}, respectively.
\begin{definition}\label{definition_boundquiver}
    Let $A$ be a finite-dimensional algebra and $(M_i)_{i\in I}$ a collection of $A$-modules. Then a bound quiver for $\filt(M)$ is a subalgebra $B$ of $A$ such that
\begin{enumerate}
        \item The induction functor  
        \begin{align*}
        A\otimes_B -:\modu B\rightarrow \modu A
    \end{align*}
    is exact.
    \item There is a bijection $\phi:\Sim(B)\rightarrow I$ such that for all $L\in \Sim(B)$ we have
    \begin{align*}
        A\otimes_B L\cong M_{\phi(L)}
    \end{align*}
    \item The ring extension $A:B$ is regular in the sense of Definition \ref{definition_regular}, replacing $\Delta_i^A$ by $M_{\phi(i)}$.
    \end{enumerate}
    We call a bound quiver $B\subseteq A$ for some  $\filt(M)$ a regular exact subalgebra.
\end{definition}
Thus, if $(A, \leq_A)$ is a quasi-hereditary algebra with regular exact Borel subalgebra $B$, then $B$ is a bound quiver for $\filt(\Delta^A)$.
Note that in Definition \ref{definition_boundquiver} there is no criterion corresponding to the directedness of $B$. However, its was shown in \cite[p. 18]{surveyjulian} that any basic regular exact subalgebra $B$ of a quasi-hereditary algebra $A$ such that $A\otimes_B B/\rad(B)\cong\bigoplus_{L\in \Sim(A)}\Delta(L)^A$ is a regular exact Borel subalgebra.
As was already observed by König in \cite[Example 2.2]{Koenig}, this does not hold if one doesn't require regularity. However, in the definition of a strong exact Borel subalgebra \cite[p. 405]{Koenig}, König also notes that, if $A$ is basic, then again the directedness follows from the other two criteria. This relates to our result \ref{thm_strong}.
We expect the following lemma to be well-known, but could not find a reference, so that we provide a proof:
\begin{lemma}\label{lemma_radical}
    Let  $(A, \leq_A)$ be a  finite-dimensional quasi-hereditary algebra and  $B$ be a strong regular exact Borel subalgebra of $A$. Then $A\rad(B)\subseteq \rad(A)$.
\end{lemma}
\begin{proof}
    Let $e_1, \dots, e_n$ be a set of principle indecomposable orthogonal idempotents in $B$. Then since $B$ is strong, they are also a set of principle indecomposable orthogonal idempotents in $A$.
    Let $x\in e_j\rad(B)e_i$ for some $1\leq i, j\leq n$. Then since $B$ is directed, $i<j$. In particular $i\neq j$, so that $e_j A e_i\subseteq \rad(A)$. Hence $x\in  e_j\rad(B)e_i\subseteq e_j A e_i\subseteq \rad(A)$.
\end{proof}
Using this, we can conclude the following result:
\begin{lemma}\label{lemma_strong_of_directed}
    Let $(A, \leq_A)$ be a basic directed quasi-hereditary algebra, and $B\subseteq A$ be an exact Borel subalgebra. Then $B=A$.
\end{lemma}
\begin{proof}
    Note that since $A$ is basic, $B$ is a strong exact Borel subalgebra, so that  $\rad(B)\subseteq \rad(A)$. Moreover, by assumption 
    \begin{align*}
       A/\rad(A)\cong \bigoplus_{L\in \Sim(A)}L\cong \bigoplus_{L\in \Sim(A)}\Delta^A(L)\\
       \cong  \bigoplus_{L\in \Sim(B)}A\otimes_B L\cong A\otimes_B (B/\rad(B))\cong A/A\rad(B).
    \end{align*}
     Thus $\rad(A)=A\rad(B)$. Let $L_B$ be a maximal semisimple subalgebra of $B$. Then by assumption it is a maximal semisimple subalgebra of $A$, so that $A=L_B\oplus \rad(A)$ and thus
     \begin{align*}
         \rad(A)=A\rad(B)=(L_B\oplus \rad(A))\rad(B)=L_B\rad(B)+\rad(A)\rad(B)=\rad(B)+\rad(A)\rad(B),
     \end{align*}
     so that by Nakayama's lemma, applied to the right $B$-module $\rad(A)$, we obtain $\rad(B)=\rad(A)$.
     Hence
    \[
         A=L_B\oplus \rad(A)=L_B\oplus \rad(B)=B.\hfill\qedhere
    \]
\end{proof}
\section{Twisted Modules over dg-algebras}\label{section_dg}
In this section, we recall the definition of twisted modules over a dg-algebra due to Bondal and Kapranov \cite{BondalKapranov} and establish the commutativity of the diagram
\begin{equation}\label{diagram_1.11}
  \begin{tikzcd}[ampersand replacement=\&]
	{\modu B} \&\& {F(\Delta^A)} \\
	{H^0(\twmod_L(\End_B^*(P^\cdot(L^B))} \&\& {H^0(\twmod_L(\End_B^*(P^\cdot(L^B))}
	\arrow[from=2-1, to=1-1]
	\arrow[from=2-3, to=1-3]
	\arrow["{A\otimes_B-}", from=1-1, to=1-3]
	\arrow["{H^0(\twmod_L(g))}", from=2-1, to=2-3]
\end{tikzcd}
\end{equation}
For more on dg-algebras, categories and twisted modules, see also \cite{Drinfeld}.
Let us first recall the definition of a dg-algebra:
\begin{definition}
    A dg-algebra  $\mathcal{A}=(A, d)$ is a graded associative, not necessarily unital, algebra $A=\bigoplus_{n\in \mathbb{Z}} A_n$ together with a $\field$-linear map
    $d:A\rightarrow A$ which is homogeneous of degree $1$ such that for all $a\in A_n, b\in A$
    \begin{align*}
        d(ab)=d(a)b+(-1)^n d(b).
    \end{align*}
    A dg-algebra homomorphism 
    \begin{align*}
        f:\mathcal{A}=(A, d)\rightarrow \mathcal{A}'=(A', d')
    \end{align*}
    is an algebra homomorphism $f:A\rightarrow A'$ homogeneous of degree zero such that $d\circ f=f\circ d$.
    Composition is the composition of maps.\\
	We call a dg-algebra $\mathcal{A}$ unital over $L$ if there is a dg-algebra homomorphism $i_L:L\rightarrow \mathcal{A}$, where the differential on $L$ is trivial. We call a homomorphism 
 $$f:(\mathcal{A}, i_L)\rightarrow (\mathcal{A}', i_L')$$ of dg-algebras, which are unital over $L$, unital over $L$, if 
	$f\circ i_L=i'_L$.\\
 Let $\mathcal{A}=(A, d, \iota_L)$ be a dg-algebra unital over $L$. Then a dg-module over $\mathcal{A}$ is a pair $(M, d_M)$ consisting of a a graded $A$-module such that $\iota_L(1_L)\cdot m=m$ for all $m\in M$ and a homorphism $d_M:M\rightarrow M$ of $L$-modules, which is homogeneous of degree $1$, such that for all $m\in M$ and all homogeneous elements $a\in A_n$ of degree $n$ we have
 \begin{align*}
     d_M(am)=d(a)m+(-1)^n a d_M(m).
 \end{align*}
\end{definition}
\begin{remark}\label{remark_long}
	\begin{enumerate}
        \item Let $R$ be a unital associative algebra. We can view $R$ as a dg-algebra unital over $\field$ which is concentrated in degree zero. Let $X=((X_n)_n, d)$ be dg-module over $R$. Then $\End^*_R(X)=\bigoplus_{n\in \mathbb{Z}}\End^n_R(X)$ obtains the structure of a dg-algebra with multiplication given by composition and differential given by 
        \begin{align*}
            d_{\End^*_R(X)}(f):=d\circ f-(-1)^{n}f\circ d
        \end{align*}
        for every $f\in \End^n_R(X)$. Moreover, if $T:\modu R\rightarrow \modu R'$ is a functor, then $T$ induces a dg- algebra homomorphism
        \begin{align*}
            T_{\End^*_R(X)}:\End^*_R(X)\rightarrow \End^*_{R'}(T(X))
        \end{align*}
        where $T(X)=((T(X_n))_n, T(d))$.
		\item Suppose $M_1, \dots, M_n$ are modules in $\modu R$ for some finite-dimensional algebra $R$. For $1\leq i\leq n$ let $P^\cdot(M_i)$ be a projective resolution of $M_i$, let $M:=\bigoplus_{i=1^n}M_i$ and let $P^\cdot(M)=\bigoplus_{i=1}^n P^\cdot(M_i)$. Denote by $\varepsilon_i$ the $i$-th unit vector in $L=\field^n$. Then $$i_L: L\rightarrow \End_R^*(P^\cdot(M)), \varepsilon_i\mapsto \id_{P^\cdot(M_i)}$$ makes $\End_R^*(P^\cdot(M))$ unital over $L$. Moreover, if $T:\modu R\rightarrow \modu R'$ is a functor, then 
        \begin{align*}
            T_{\End^*_R(P^\cdot(M))}\circ i_L:L\rightarrow \End^*_{R'}(T(P^\cdot(M)))
        \end{align*}
        makes $ \End^*_{R'}(T(P^\cdot(M)))$ unital over $L$ in such a way that $T_{\End^*_R(P^\cdot(M))}$ is unital over $L$.
		\item If $\mathcal{A}$ is unital over $L$ then it canonically obtains the structure of an $L$-$L$-bimodule via restriction.
		\item If $X$ is a left $R$-module and $Y$ a left $R'$-module, then $\Hom_{\field}(X, Y)$ has the canonical structure of an $R'$-$R$-bimodule. In particular, if $R=R'=L$, then $\Hom_{\field}(X, Y)$ has in this way the structure of an $L$-$L$-bimodule. Since $L$ is commutative, there are other ways to associate a bimodule structure, but we always consider this one.
\end{enumerate}
\end{remark}
\begin{lemma}
    Let $\mathcal{A}$ be a dg-algebra. Then its homology $H^*(\mathcal{A})$ (with respect to its differential $d$) obtains the structure of an associative (not necessarily unital) algebra via
    \begin{align*}
        (a+\im(d))\cdot (b+\im(d))=ab+\im(d).
    \end{align*}
    Moreover, $H^*$ defines a functor from dg-algebras to (not necessarily unital) associative graded algebras.
    Similarly, $H^0$ defines a functor from dg-algebras to (not necessarily unital) associative algebras.\\
\end{lemma}
\begin{definition}
    A dg-algebra homomorphism $f:\mathcal{A}\rightarrow\mathcal{A}'$ is called a quasi-isomorphism if $H^*(f):H^*(\mathcal{A})\rightarrow H^*(\mathcal{A}')$ is an isomorphism of associative algebras.
\end{definition}
In order to define the category of twisted modules, we need to define the additive enlargement of a dg-algebra $\mathcal{A}$. This essentially turns our dg-algebra into a dg-category, in the same way that the projective modules $\add A$ form a category for an associative algebra $A$, where $L$ takes the place of a maximal semisimple subalgebra of $A$.
First, we recall the definition of a dg-category:
\begin{definition}
    A dg-category $\mathcal{A}$ over $\field$ is a collection of objects and for all objects $X, Y\in \mathcal{A}$, a graded $\field$-vector space $\mathcal{A}(X,Y)=(\mathcal{A}(X,Y)_n)_{n\in \mathbb{Z}}$ together with a differential \begin{align*}
        d_{\mathcal{A}(X,Y)}:\mathcal{A}(X,Y)\rightarrow \mathcal{A}(X,Y)
    \end{align*}
    of degree one,
    as well as for all objects $X, Y, Z$ a composition map
    \begin{align*}
        \circ:\mathcal{A}(Y, Z)\otimes_{\field}\mathcal{A}(X,Y)\rightarrow \mathcal{A}(X, Z)
    \end{align*}
    of degree zero,
    such that $\circ $ is associative and for all $f\in \mathcal{A}(Y,Z)_n$, $g\in \mathcal{A}(X,Y)$
    \begin{align*}
        d_{\mathcal{A}(X, Z)}(f\circ g)=d_{\mathcal{A}(Y, Z)}(f)\circ g+(-1)^n f\circ d_{\mathcal{A}(X,Y)}(g).
    \end{align*}
    A dg-category is called unital if for every object $X$ there is $\id_X\in \mathcal{A}(X, X)_0$ such that $ d_{\mathcal{A}(X,X)}(\id_X)=0$ and for all objects $Y$ and all $f\in \mathcal{A}(Y,X)$, $g\in \mathcal{A}(X,Y)$ we have $\id_X\circ f=f$ and $g\circ \id_X=g$.
\end{definition}
\begin{definition}
      Let $\mathcal{A}$ be a unital dg-category. Then its homology $H^*(\mathcal{A})$ is the ($\mathbb{Z}$-graded) category with objects being the objects of $\mathcal{A}$, morphism spaces 
      \begin{align*}
          H^*(\mathcal{A})(X,Y)=H^*(\mathcal{A}(X,Y))
      \end{align*}
      and composition induced by the composition in $\mathcal{A}$.
      A unital dg-functor $F:\mathcal{A}\rightarrow \mathcal{A}'$ gives rise to a functor 
      \begin{align*}
          H^*(F):H^*(\mathcal{A})\rightarrow H^*(\mathcal{A}'),\\ 
                X\mapsto F(X), f+\im(d)\mapsto F(f)+\im(d).
      \end{align*}
      In this way, $H^*$ becomes a functor from unital dg-categories to ($\mathbb{Z}$-graded) categories. 
      Similarly, $H^0$  becomes a functor from unital dg-categories to categories. 
\end{definition}
\begin{definition}
    A unital dg-functor $F$ is called a quasi-equivalence if $H^*(F)$ is an equivalence of categories.
\end{definition}
\begin{definition}
	Let $\mathcal{A}=(\mathcal{A}, d^{\mathcal{A}}, i_L)$ be a dg-algebra unital over $L$. Then $\add(\mathcal{A})$ is the unital dg-category whose objects are $L$-modules, whose homomorphisms are given by
	\begin{align*}
		\add(\mathcal{A})(X,Y)=\mathcal{A}\otimes_{L\otimes L^{\op}}\Hom_{\field}(X,Y),
	\end{align*}
	with grading induced by the grading of $\mathcal{A}$, whose composition is given by 
	\begin{align*}
		(a_1\otimes g_1)\circ (a_2\otimes g_2)=a_1a_2\otimes (g_1\circ g_2)
	\end{align*}
	and whose differential is given by
	\begin{align*}
		d^{\add(\mathcal{A})}(a\otimes g)=d^{\mathcal{A}}(a)\otimes g.
	\end{align*}
 Its units are given by
 \begin{align*}
     \id_X=1_L\otimes \id_X\in \mathcal{A}\otimes_{L\otimes L^{\op}} \End_{\field}(X).
 \end{align*}
\end{definition}

\begin{definition}
	If $f:\mathcal{A}\rightarrow \mathcal{A}'$ is a dg-algebra homomorphism unital over $L$, then there is an induced unital dg-functor
	\begin{align*}
		\add(f):\add(\mathcal{A})\rightarrow \add(\mathcal{A}'), X\mapsto X, a\otimes g\mapsto f(a)\otimes g.
	\end{align*}
	This gives rise to a functor $\add$ from dg-algebras unital over $L$ to unital dg-categories.
\end{definition}
\begin{definition}
    Let $Q$ be an $L$-$L$-bimodule and $X$ be a left $L$-module. Then an element
    $w\in Q\otimes_{L\otimes L^{\op}}\End_{\field}(X)$ is called triangular, if there is a presentation $w=\sum_{i=1}^m q_i\otimes f_i$ and an $N\in \mathbb{N}$ such that for all $i_1, \dots, i_N\in \{1, \dots m\}$ we have
    \begin{align*}
        f_{i_N}\circ\dots\circ f_{i_1}=0.
    \end{align*}
\end{definition}
\begin{lemma}
    Let $R$ be an algebra, $(M_j)_{j=1}^n$ be a set of $R$-modules. Let $M:=\bigoplus_{i=1}^n M_i$ and $\mathcal{A}:=\End_R(P^{\cdot}(M))$. Moreover, let $X$ be an $L$-module and $w_X\in (\mathcal{A}\otimes_{L}\End_{\field} (X))_1$. Then $w_X$ is triangular if and only if there is a sequence of submodules
    \begin{align*}
        (0)=X_0\subseteq X_1\subseteq\dots \subseteq X_N=X
    \end{align*}
    such that $w_X(P^\cdot(M)\otimes_L X_i)\subseteq P^\cdot(M)\otimes_L X_{i-1}$
    for $1\leq i\leq N$. 
\end{lemma}
\begin{proof}
    Suppose $w_X$ is triangular.  Let $w=\sum_{i=1}^m g_i\otimes f_i$ be a presentation of $w$ and $N\in \mathbb{N}$ such that for all $i_1, \dots, i_N\in \{1, \dots m\}$ we have
    \begin{align*}
        f_{i_N}\circ\dots\circ f_{i_1}=0.
    \end{align*}
    For $0\leq i\leq N$ let 
    \begin{align*}
        X_{i}:=\sum_{l=N-i}^N\sum_{1\leq i_1, \dots, i_{l}\leq m}\im(f_{i_{l}}\circ \dots\circ f_{i_1})\subseteq X.
    \end{align*}
    Then by definition $X_i\subseteq X_{i+1}$. Moreover, $X_0=(0)$, $X_N=X$ and for $1\leq i\leq N$, $1\leq j\leq m$ we have $f_j(X_i)\subseteq X_{i-1}$.
    Hence $w_X(P^{\cdot}(M)\otimes_L X_i)\subseteq P^{\cdot}(M)\otimes_L X_{i+1}$.\\
    On the other hand, suppose there is a 
    sequence of submodules
    \begin{align*}
        (0)=X_0\subseteq X_1\subseteq\dots \subseteq X_N=X
    \end{align*}
    such that $w_X(P^\cdot(M)\otimes_L X_i)\subseteq P^\cdot(M)\otimes_L X_{i-1}$
    for $1\leq i\leq N$. 
    Since $L$ is semisimple, we can find a complement $Y_i$ for $X_{i-1}$ in $X_i$ for every $1\leq i\leq N$.
    Then we have $X_i=\bigoplus_{j=1}^i Y_i$ for all $1\leq i\leq N$. In particular,  $X=\bigoplus_{j=1}^N Y_i$.
    Denote by $\pi_{Y_i}: X\rightarrow Y_i$ the canonical projection and by  $\iota_{Y_i}:Y_i\rightarrow X$ the canonical embedding.
    For $1\leq i\leq j$ let $w_{ij}:=(\id_{P^\cdot(M)}\otimes_L\pi_{Y_j})\circ w_X\circ (\id_{P^\cdot(M)}\otimes_L\iota_{Y_i})$.
    Then, since  
    \begin{align*}
       w_X(P^{\cdot}(M)\otimes_L Y_i)\subseteq w_X(P^\cdot(M)\otimes_L X_i)\subseteq P^\cdot(M)\otimes_L X_{i-1}=\bigoplus_{j=1}^{i-1} P^\cdot(M)\otimes_L Y_j,
    \end{align*}
     we have $w_{ij}=0$ if $i\leq j$.\\
     Thus
     \begin{align*}
         w_X=\sum_{i,j=1}^N w_{ij}=\sum_{i=1}^N\sum_{j=1}^{i-1} w_{ij}.
     \end{align*}
     Let us write $w_X=\sum_{r=1}^m a_r\otimes f_r$ for $a_r\in \End_R(P^\cdot(M))$ and $f_r\in \Hom_{\field}(X,X)$. Then
     \begin{align*}
         w_X=\sum_{i=1}^N\sum_{j=1}^{i-1} w_{ij}
         =\sum_{i=1}^N\sum_{j=1}^{i-1}(\id_{P^\cdot(M)}\otimes_L\pi_{Y_j})\circ w_X\circ (\id_{P^\cdot(M)}\otimes_L\iota_{Y_i})\\
          =\sum_{i=1}^N\sum_{j=1}^{i-1}\sum_{r=1}^m (\id_{P^\cdot(M)}\otimes_L\pi_{Y_j})\circ(a_r\otimes f_r)\circ (\id_{P^\cdot(M)}\otimes_L\iota_{Y_i})
          =\sum_{i=1}^N\sum_{j=1}^{i-1}\sum_{r=1}^m a_r\otimes (\pi_{Y_j}\circ f_r\circ \iota_{Y_i}).
     \end{align*}
    Moreover, for all $K\in \mathbb{N}$, $1\leq r_1, \dots, r_K\leq m$, $1\leq i_1, \dots , i_K, j_{1}\dots j_K\leq N$, the composition
    \begin{align*}
         (\pi_{Y{j_K}}\circ f_{r_K}\circ \iota_{Y_{i_K}})\circ (\pi_{Y_{j_{K-1}}}\circ f_{r_{K-1}}\circ \iota_{Y_{i_{K-1}}}) \dots \circ (\pi_{Y{j_1}}\circ f_{r_1}\circ \iota_{Y_{i_1}})
    \end{align*}
     can be non-zero only if $j_K<i_K=j_{K-1}<i_{K-1}=j_{K-2}<\dots <i_2=j_1<i_1$.
    In particular, this is possible only if $j_K\leq i_1-K\leq N-K$. Hence, setting $K=N$, we obtain that 
    \begin{align*}
         (\pi_{Y{j_N}}\circ f_{r_N}\iota_{Y_{i_N}})\circ (\pi_{Y{j_{N-1}}}\circ f_{r_{N-1}}\iota_{Y_{i_{N-1}}}) \dots \circ (\pi_{Y{j_1}}\circ f_{r_1}\iota_{Y_{i_1}})=0
    \end{align*}
    for all $1\leq r_1, \dots, r_N\leq m$, $1\leq i_1, \dots , i_N, j_{1}\dots j_N\leq N$.
\end{proof}
We can now give the definition of twisted modules over a dg-algebra. 
\begin{definition}\cite[Definition 1]{BondalKapranov2}
	Let $\mathcal{A}=(\mathcal{A}, d_{\mathcal{A}}, i_L)$ be a dg-algebra unital over $L$. Then a twisted module over $\mathcal{A}$ is a pair consisting of an $L$-module $X$ together with an element $w_X\in \mathcal{A}_1\otimes_{L\otimes L^{\op}} \End_{\field}(X)=\add(\mathcal{A})(X,X)_1$ such that
    $w_X$ is triangular and
    fulfills the Maurer--Cartan equation 
	\begin{align*}
		d^{\add(\mathcal{A})}(w_X)+w_X\circ w_X=0.
	\end{align*}
	The category $\twmod_L(\mathcal{A})$ is the unital dg-category whose objects are twisted modules over $\mathcal{A}$, whose homomorphisms are given by 
	\begin{align*}
		\twmod(\mathcal{A})((X, w_X), (Y, w_Y))=\add(\mathcal{A})(X,Y)
	\end{align*}
	with the same grading,
	and whose differential is given by 
	\begin{align*}
		d^{\twmod(\mathcal{A})((X, w_X), (Y, w_Y))}(f)=d^{\add(\mathcal{A})}(f)+w_Y\circ f-(-1)^{|f|}f\circ w_X.
	\end{align*}
 Its units are the same as in $\add(\mathcal{A})$.
\end{definition}
The following can be found in \cite[Section 7.2., p. 166]{LefHas} for twisted complexes over A-infinity algebras, which generalize twisted modules over dg-algebras.
\begin{definition}
	Let $f:\mathcal{A}\rightarrow \mathcal{A}'$ be a dg-algebra homomorphism strictly unital over $L$. Then there is an induced unital dg-functor
	\begin{align*}
		\twmod(f):\twmod(\mathcal{A})&\rightarrow \twmod(\mathcal{A}'),\\
  (X, w_X)&\mapsto (X, \add(f)(w_X)), g\mapsto \add(f)(g).
	\end{align*}
	This gives rise to a functor $\twmod$ from dg-algebras unital over $L$ to unital dg-categories.
\end{definition}
The following equivalence $\filt(M)\cong H^0(\twmod_L(\End(P^\cdot(M))))$ for $M=\bigoplus_{i=1}^nM_i$ and $L\cong \field^n$ can be obtained from \cite[Section 4, Theorem 1]{BondalKapranov} specializing $E_i:=P^\cdot(M_i)$ in $\mathcal{E}:=\complexes(A)$, 
and then restricting this to the full subcategories generated by $(L_1, 0),\dots ,(L_n, 0)$ resp. $E_1, \dots , E_n$ via cones (but not shifts).\\
We give a proof, in order to illustrate that in this special case, this statement follows directly from the horseshoe lemma. However, readers familiar with this material might want to skip it.
\begin{theorem}\label{theorem_equivalence_End_F(M)}
	Let $R$ be an algebra, $(M_j)_{j=1}^n$ be a set of $R$-modules. Then there is an equivalence 
	\begin{align*}
		C_M: H^0(\twmod_L(\End(P^\cdot(M))))\rightarrow \filt(M)
	\end{align*} given by 
	\begin{align*}
		(X, w)\mapsto H^0(P^\cdot(M)\otimes_L X, d_{P^{\cdot}(M)}\otimes\id_X+w),
		[f]\mapsto H^0(f).
	\end{align*}
	where we identify $\End_A^*(P^\cdot(M))\otimes_{L\otimes L^{\op}}\Hom_{\field}(X, Y)$ with $\Hom_A^*(P^{\cdot}(M)\otimes_L X, P^{\cdot}(M)\otimes_L Y)$ via
\begin{align*}
	(f\otimes g)(a\otimes x)= f(a)\otimes g(x).
\end{align*}
	Note that the Koszul sign rule does not give rise to a sign here, since $X$ is ungraded and thus $g$ is viewed as being of degree zero.
\end{theorem}
\begin{proof}
    First, let us show that $C_M$ is well-defined.\\
	Let $(X, w)$ be an object in $H^0(\twmod(\End_A^*(P^\cdot(M))))$. Then $X$ is an $L$-module and 
 \begin{align*}
     w\in \End_A^1(P^{\cdot}(M))\otimes_{L\otimes L^{\op}} \End_L(X)\cong \End_A^1(P^{\cdot}(M)\otimes_L X).
 \end{align*} Since $w$ fulfills the Maurer--Cartan equation, 
	\begin{align*}
		0&=(d_{\End_A^*(P^{\cdot}(M))}\otimes\id_{\End_L(X)})(w)+w\circ w\\
		&=(d_{P^{\cdot}(M)}\otimes\id_{X})\circ w- (-1) w\circ (d_{P^{\cdot}(M)}\otimes\id_{X})+w\circ w\\
		&=(d_{P^{\cdot}(M)}\otimes \id_X)\circ (d_{P^{\cdot}(M)}\otimes\id_X)+(d_{P^{\cdot}(M)}\otimes\id_X)\circ w+w\circ (d_{P^{\cdot}(M)}\otimes\id_X)+w\circ w\\
		&=(d_{P^{\cdot}(M)}\otimes\id_X+w)\circ (d_{P^{\cdot}(M)}\otimes\id_X+w).
	\end{align*}
	Thus, $d_{P^{\cdot}(M)}\otimes\id_X+w$ is a differential on $P^\cdot(M)\otimes_L X$. 
 Moreover, by triangularity of $w$ there is a sequence
	\begin{align*}
		(0)=X_0\subseteq X_1\subseteq \dots \subseteq X_n=X
	\end{align*}
	of $L$-submodules of $X$ such that $w(P^\cdot(M)\otimes X_i)\subseteq P^{\cdot}(M)\otimes X_{i-1}$ for every $1\leq i\leq n$. In particular,
	$$(d_{P^{\cdot}(M)}\otimes\id_X+w)(P^\cdot(M)\otimes X_i)\subseteq P^{\cdot}(M)\otimes X_{i-1},$$ so that as a complex of $A$-modules, $(P^\cdot(M)\otimes X, d_{P^{\cdot}(M)}\otimes\id_X+w)$ has a series of subcomplexes
	\begin{align}\label{chainofsubcomplexes}
		(0)&=(P^{\cdot}(M)\otimes X_{0}, d_{P^{\cdot}(M)}\otimes\id_{X_0}+w_{|P^{\cdot}(M)\otimes X_{0}})\subseteq \dots \\
		& \subseteq (P^{\cdot}(M)\otimes X_{n}, d_{P^{\cdot}(M)}\otimes\id_{X_n}+w_{|P^{\cdot}(M)\otimes X_{n}})=(P^\cdot(M)\otimes X, d_{P^{\cdot}(M)}\otimes\id_X+w),
	\end{align}
	such that on every factor complex $P^{\cdot}(M)\otimes (X_i/X_{i-1})$ the differential is given by $d_{P^{\cdot}(M)}\otimes \id_{X_i/X_{i-1}}$.
  Exactness of $P^\cdot(M)$ thus implies that $$H^k(P^{\cdot}(M)\otimes (X_j/X_{j-1}), d_{P^{\cdot}(M)}\otimes \id_{X_j/X_{j-1}})=0$$ for every $1\leq j\leq n$, $k\neq 0$. Moreover, $(P^\cdot(M)\otimes X_i, d_{P^{\cdot}(M)}\otimes\id_{X_i}+w_{|P^{\cdot}(M)\otimes X_{i}}))$ has a filtration by the complexes $(P^{\cdot}(M)\otimes (X_j/X_{j-1}), d_{P^{\cdot}(M)}\otimes \id_{X_j/X_{j-1}})$, $1\leq j\leq n$.
 Since short exact sequences of complexes give rise to long exact sequences in homology, this implies that 
   $$H^k(P^\cdot(M)\otimes X_i,  d_{P^{\cdot}(M)}\otimes\id_{X_i}+w_{|P^{\cdot}(M)\otimes X_{i}})=0$$ for every $1\leq i\leq n$, $k\neq 0$. Thus $(P^\cdot(M)\otimes X_i, d_{P^{\cdot}(M)}\otimes\id_{X_i}+w_{|P^{\cdot}(M)\otimes X_{i}})$ is a projective resolution of its degree zero homology $H^0(P^\cdot(M)\otimes X_i,  d_{P^{\cdot}(M)}\otimes\id_{X_i}+w_{|P^{\cdot}(M)\otimes X_{i}})$ for every $1\leq i\leq n$.\\
   In particular, applying $H^0$ to the chain of subcomplexes  \eqref{chainofsubcomplexes}, we obtain a chain of $A$-submodules
	\begin{align*}
		(0)&=H^0(P^{\cdot}(M)\otimes X_{0}, d_{P^{\cdot}(M)}\otimes\id_{X_0}+w_{|P^{\cdot}(M)\otimes X_{0}})
		 \subseteq \dots \\ & \subseteq H^0(P^{\cdot}(M)\otimes X_{n}, d_{P^{\cdot}(M)}\otimes\id_{X_n}+w_{|P^{\cdot}(M)\otimes X_{n}})=H^0(P^\cdot(M)\otimes X, d_{P^{\cdot}(M)}\otimes\id_X+w),
	\end{align*}
 of $H^0((P^\cdot(M)\otimes X, d_{P^{\cdot}(M)}\otimes\id_X+w))$,
	such that the factor modules are isomorphic to $$H^0(P^{\cdot}(M)\otimes (X_i/X_{i-1}), d_{P^{\cdot}(M)}\otimes \id_{X_i/X_{i-1}})=M\otimes_L (X_i/X_{i-1}),$$ i.e. to a direct sum of $M_j$. Thus $$H^0((P^\cdot(M)\otimes X, d_{P^{\cdot}(M)}\otimes\id_X+w))\in \filt(M).$$ This implies that $C_M$ is well-defined.
  Since $\twmod(\End^*_A(P^\cdot(M))))((X, w), (X', w'))$ is the complex  
\begin{align*}
	\twmod(\End^*_A(P^\cdot(M)))((X, w), (X', w'))&=\End^*_A(P^\cdot(M))\otimes_{L\otimes L^{\op}} \Hom_{\field}(X, X')\\
	&\cong \Hom_A^*(P^\cdot(M)\otimes_L X, P^\cdot(M)\otimes_L X')
\end{align*} 
	with differential 
	\begin{align*}
		f\mapsto &d_{\End_A^*(P^\cdot(M))}\otimes \id_{\Hom_{\field} (X, X')}(f)+f\circ w-(-1)^{|f|}w'\circ f\\
		&=(d_{P^\cdot(M)}\otimes\id_{X'}+w')\circ f-(-1)^{|f|}f\circ (d_{P^\cdot(M)}\otimes\id_{X}+w),
	\end{align*}
	its homology $H^0(\twmod(\End^*_A(P^\cdot(M))))((X, w), (X', w')))$ is exactly $$\Hom_A(H^0(P^\cdot(M)\otimes_L X, d_{P^\cdot(M)}\otimes\id_{X}+w), H^0(P^\cdot(M)\otimes_L X', d_{P^\cdot(M)}\otimes\id_{X'}+w')),$$ so that $C_M$ is fully faithful.\\
It remains to be shown that it is dense.
Clearly, $M_j$ is in the essential image of $C_M$ as it is isomorphic to $C_M((L_i, 0))$. Hence it suffices to show that the essential image is closed under extensions. So suppose $(X', w')$ and $(X'', w'')$ are objects in $H^0(\twmod(\End_A^*(P^\cdot(M))))$ and we have an exact sequence
\begin{align*}
	(0)\rightarrow H^0((P^\cdot(M)\otimes_L X', d_{P^{\cdot}(M)}\otimes\id_{X'}+w'))\rightarrow Y\rightarrow H^0((P^\cdot(M)\otimes_L X'', d_{P^{\cdot}(M)}\otimes\id_{X''}+w'')\rightarrow (0).
\end{align*}
Since $(P^\cdot(M)\otimes_L X', d_{P^{\cdot}(M)}\otimes\id_{X'}+w')$ and $(P^\cdot(M)\otimes_L X'', d_{P^{\cdot}(M)}\otimes\id_{X''}+w'')$ are projective resolutions of $H^0((P^\cdot(M)\otimes_L X', d_{P^{\cdot}(M)}\otimes\id_{X'}+w'))$ and $H^0((P^\cdot(M)\otimes_L X'', d_{P^{\cdot}(M)}\otimes\id_{X''}+w'')$ respectively, the horseshoe lemma implies that there is a projective resolution of $Y$ given by $P^\cdot(M)\otimes_L X'\oplus P^\cdot(M)\otimes_L X''$ with differential
\begin{align*}
	\begin{pmatrix}
		d_{P^{\cdot}(M)}\otimes\id_{X'}+w'& w'''\\
		0 & d_{P^{\cdot}(M)}\otimes\id_{X''}+w''
	\end{pmatrix}
	=\begin{pmatrix}
		d_{P^{\cdot}(M)}\otimes\id_{X'}& 0\\
		0 & d_{P^{\cdot}(M)}\otimes\id_{X''}
	\end{pmatrix}
	+ \begin{pmatrix}
		w'& w'''\\
		0 & w''
	\end{pmatrix}
\end{align*}
for some $w'''\in \Hom_R^1(P^\cdot(M)\otimes_L X'',P^\cdot(M)\otimes_L X')=\End_R^1(P^\cdot(M))\otimes_{L\otimes L^{\op}}\Hom_{\field}(X'', X')$.
As a complex of $R$-modules, this is isomorphic to $P^\cdot(M)\otimes_L (X'\oplus X'')$ with differential $d_{P^\cdot(M)}\otimes\id_{X'\oplus X''}+w$ where $w=\begin{pmatrix}
	w'& w'''\\
	0 & w''
\end{pmatrix}$.
Since $d_{P^\cdot(M)}\otimes\id_{X'\oplus X''}+w$ is a differential, $w\in \End_R^1(P^\cdot(M)\otimes_L(X'\oplus X''))=\End_R^1(P^\cdot(M))\otimes_{L\otimes L^{\op}}\End_{\field}(X'\oplus X'')$ fulfills the Maurer--Cartan equation. Finally, if
\begin{align*}
	(0)=X'_0\subseteq \dots \subseteq X'_n=X'
\end{align*}
and 
\begin{align*}
	(0)=X''_0\subseteq \dots \subseteq X''_m=X''_m
\end{align*}
are filtrations of $X'$ and $X''$ as $L$-modules such that
\begin{align*}
	w'(P^\cdot(M)\otimes_L X'_{i})\subseteq P^\cdot(M)\otimes_L X'_{i-1}
\end{align*}
and
\begin{align*}
	w''(P^\cdot(M)\otimes_L X''_{i})\subseteq P^\cdot(M)\otimes_L X''_{i-1}
\end{align*}
then 
\begin{align*}
	(0)=X_0:=X'_0&\subseteq \dots \subseteq X_n:=X'_n=X'=X'\oplus X''_0\\
	&\subseteq X_{n+1}:=X'\oplus X''_1\subseteq \dots \subseteq X_{n+m}:=X'\oplus X''_m=X'\oplus X''
\end{align*}
is a filtration of $X:=X'\oplus X''$ as an $L$-module, and since 
\begin{align*}
\begin{pmatrix}
	0& w'''\\
	0 & 0
\end{pmatrix}(P^\cdot(M)\otimes X')=(0)
\end{align*}
and 
\begin{align*}
	\begin{pmatrix}
		0& w'''\\
		0 & 0
	\end{pmatrix}(P^\cdot(M)\otimes X'')\subseteq P^\cdot(M)\otimes X',
	\end{align*}
	we have
	\begin{align*}
		w(P^\cdot(M)\otimes_L X_i)=\begin{pmatrix}
			w'& w'''\\
			0 & w''
		\end{pmatrix}(P^\cdot(M)\otimes_L X'_i)=w'(P^\cdot(M)\otimes_L X'_i)\subseteq P^\cdot(M)\otimes_L X'_{i-1}
	\end{align*}
	for $0\leq i\leq n$ and
	\begin{align*}
		w(P^\cdot(M)\otimes_L X_{n+i})=\begin{pmatrix}
			w'& w'''\\
			0 & w''
		\end{pmatrix}(P^\cdot(M)\otimes_L (X'\oplus X''_i))\\\subseteq w'(P^\cdot(M)\otimes_L X')+ \begin{pmatrix}
			0& w'''\\
			0 & 0
		\end{pmatrix}(P^\cdot(M)\otimes X'')+w''(P^\cdot(M)\otimes X''_i)\\
		\subseteq P^\cdot(M)\otimes_L X'+P^\cdot(M)\otimes_L X'+P^\cdot(M)\otimes_L X''_{i-1}\\=P^\cdot(M)\otimes_L (X'+X''_{i-1})=P^\cdot(M)\otimes X_{n+i-1}
		\end{align*}
		for $1\leq i\leq m$. Hence
	$w(P^\cdot(M)\otimes_L X_i)\subseteq P^\cdot(M)\otimes_L X_{i-1}$, so that $w$ is triangular.
	Therefore, $(X, w)\in H^0(\twmod(\End_A^*(P^\cdot(M))))$ and $Y\cong C_M(X, w)$.
\end{proof}
 The following lemma shows that the equivalence above commutes with an exact functor $T$ that preserves projectives, a statement which we could not find in the literature, although we expect this to be generally known.
 \begin{lemma}
    Let $R, R'$ be algebras and let  $(M_j)_{j=1}^n$ be a set of $R$-modules.
     Suppose $T:\modu R\rightarrow \modu R'$ is an exact functor, preserving projectives. Then we obtain a diagram
\[\begin{tikzcd}[ampersand replacement=\&]
	{\filt(M)} \&\&\& {\filt(T(M))} \\
	{H^0(\twmod(\End_R^*(P^\cdot(M))))} \&\&\& {H^0(\twmod(\End_{R'}^*(T(P^\cdot(M)))))}
	\arrow["{C_M}", from=2-1, to=1-1]
	\arrow["{C_{T(M)}}", from=2-4, to=1-4]
	\arrow["T", from=1-1, to=1-4]
	\arrow["{H^0(\twmod(T_{\End_R^*(P^\cdot(M))}))}"', from=2-1, to=2-4]
\end{tikzcd}\]
that commutes up to natural isomorphism $\alpha: T\circ C_M\rightarrow C_{T(M)}\circ T$ given by the canonical isomorphisms
\begin{align*}
	\alpha_X:T(H^0(P^\cdot (M)\otimes X, d_{P^{\cdot}(M)}\otimes\id_X+w))&\rightarrow 
	H^0(T(P^\cdot(M)\otimes X), T(d_{P^{\cdot}(M)}\otimes\id_X+w)).
\end{align*}
 \end{lemma}
 \begin{proof}
     Let 
     \begin{align*}
         f&\in H^0(\twmod_L(\End_R^*(P^\cdot(M)))((X,w_X), (Y, w_Y))\\
         &=H^0(\End_R^*(P^\cdot(M)\otimes_{L\otimes L^{\op}} \End_{\field}(X, Y))).
     \end{align*}
     Write $f=\left[  \sum_{j=1}^m h_j\otimes f_j \right]$ for $h_j\in \End_R^*(P^\cdot(M))$ and $f_j\in \End_{\field}(X,Y)$.
     Then by definition $C_M(f)$ is just the image of $f$ under the canonical isomorphism
     \begin{align*}
         H^0(\End_R^*(P^\cdot(M))\otimes_{L\otimes L^{\op}} \End_{\field}(X, Y))\rightarrow 
        H^0(\Hom_R^*(P^\cdot(M)\otimes_L X, &P^\cdot(M)\otimes_L Y)\\
       & \rightarrow \Hom_R(H^0(P^\cdot(M)\otimes_L X,  P^\cdot(M)\otimes_L Y)).
     \end{align*}
     More concretely,
     \begin{align*}
         C_M(f)=\left[  \sum_{j=1}^m h_j\otimes f_j \right]:H^0(P^\cdot(M)\otimes_L X)&\rightarrow H^0(P^\cdot(M)\otimes_L Y),\\
         \left[\sum_{i=1}^s p_i\otimes x_i\right]&\mapsto  \left[\sum_{j=1}^m\sum_{i=1}^s h_j(p_i)\otimes f_j(x_i)\right]
         =\left[\left(\sum_{j=1}^m h_j\otimes f_j\right)\left(\sum_{i=1}^s p_i\otimes x_i\right)\right].\\
    \end{align*} 
    Hence, $T(C_M(f))=T\left(\left[\sum_{j=1}^m h_j\otimes f_j\right]\right)$.
    On the other hand,
    \begin{align*}
        H^0(\twmod_L(T_{\End_R^*(P^\cdot(M))}))(f)&=  H^0(\twmod_L(T_{\End_R^*(P^\cdot(M))})\left(\left[  \sum_{j=1}^m h_j\otimes f_j \right]\right)\\
        &=\left[  \sum_{j=1}^m \twmod_L(T_{\End_R^*(P^\cdot(M))})( h_j\otimes f_j ) \right]\\
        &=\left[ \sum_{j=1}^m T( h_j)\otimes f_j  \right].
    \end{align*}
    Moreover, since $T$ is additive and maps $M_i$ to $T(M_i)$, we have
    \begin{align*}
         \sum_{j=1}^m T( h_j)\otimes f_j =  \sum_{j=1}^m T( h_j\otimes f_j)=  T\left(\sum_{j=1}^m  h_j\otimes f_j\right)
    \end{align*}
    and thus
    \begin{align*}
         C_M(H^0( \twmod_L(T_{\End_R^*(P^\cdot(M))}))(f))=\left[T\left(\sum_{j=1}^m  h_j\otimes f_j\right)\right].
    \end{align*}
    Therefore, we have a commutative diagram
\[\begin{tikzcd}[ampersand replacement=\&]
	{T(H^0(P^\cdot(M)\otimes_L X))} \& {H^0(T(P^\cdot(M)\otimes_L X))} \\
	{T(H^0(P^\cdot(M)\otimes_L Y))} \& {H^0(T(P^\cdot(M)\otimes_L Y))}
	\arrow["{C_{T(M)}\circ H^0(\twmod_L(T_{\End_R^*(P^\cdot(M))}))(f)}", from=1-2, to=2-2]
	\arrow["{T\circ C_M(f)}", from=1-1, to=2-1]
	\arrow["{\alpha_Y}", from=2-1, to=2-2]
	\arrow["{\alpha_X}", from=1-1, to=1-2]
\end{tikzcd}\]
 \end{proof}
\section{Twisted Modules over A-infinity Algebras}\label{section_ainfty}
In this section, we give the generalization of the construction of twisted modules from the previous section to A-infinity algebras, and recall that quasi-isomorphisms of A-infinity algebras give rise to equivalences of $H^0(\twmod)$. Thus we obtain a commutative diagram where the vertical arrows are equivalences. This section is largely based on \cite{LefHas}, and we follow their sign conventions.
\subsection{A-infinity algebras}
\begin{definition}
  An A-infinity algebra is a graded vector space $\mathcal{A}=\bigoplus_{n\in \mathbb{Z}}\mathcal{A}_n$ together with $\field$-linear maps
  \begin{align*}
    m_n:\mathcal{A}^{\otimes n}\rightarrow \mathcal{A}
  \end{align*}
  of degree $2-n$ for every $n\in \mathbb{N}$ such that
  \begin{align*}
    \sum_{r+s+t=n}(-1)^{rs+t}m_{r+t+1}(1^{\otimes r}\otimes m_s\otimes 1^{\otimes t})=0.
  \end{align*}
  An A-infinity homomorphism $f:\mathcal{A}\rightarrow\mathcal{A}'$ is a family $(f_n)_n$ of $\field$-linear maps
  \begin{align*}
    f_n:\mathcal{A}^{\otimes_{\field} n}\rightarrow \mathcal{A}'
  \end{align*}
  of degree $1-n$ such that
	\begin{align*}
		\sum_{r+s+t=n}(-1)^{rs+t}f_{r+t+1}(1^{\otimes r}\otimes m_s\otimes 1^{\otimes t})=\sum_{j_1+\dots+j_k=n}(-1)^{\sum_{l=1}^k (1-j_l) \sum_{l'=1}^l j_{l'}}m_k(f_{j_1}\otimes \dots \otimes f_{j_k}).
	\end{align*} An  A-infinity homomorphism $f$ is called strict if $f_n=0$ for $n\neq 1$.
 Suppose $g=(g_n)_n:\mathcal{A}\rightarrow\mathcal{A}'$, $f=(f_n)_n:\mathcal{A}'\rightarrow\mathcal{A}''$ are A-infinity homomorphisms. Then their composition $f\circ g$ is defined as
 \begin{align*}
     (f\circ g)_n:=\sum_{j_1+\dots+j_k=n}(-1)^{\sum_{l=1}^k (1-j_l) \sum_{l'=1}^l j_{l'}}f_k(g_{j_1}\otimes \dots \otimes g_{j_k}).
 \end{align*}
 This is again an A-infinity homomorphism, and composition is associative, so that we obtain a category; the category of A-infinity algebras.
\end{definition}
\begin{remark}
  Note that a dg-algebra $\mathcal{A}$ with differential $d$ and multiplication $m$ is an A-infinity algebra via $m_1=d$, $m_2=m$, $m_n=0$ for $n>2$, and that dg-algebra homomorphisms are strict A-infinity homomorphisms. Moreover, composition of two strict A-infinity homomorphisms $f=f_1$, $g=g_1$ is the strict A-infinity homomorphism $f\circ g=f_1\circ g_1$. Hence the category of dg-algebras and dg-algebra homomorphisms is a (non-full) subcategory of the category of A-infinity algebras.
\end{remark}
\begin{lemma}\label{lemma_ainftyhominvertible}(see for example \cite[Proposition 11]{markl})
    An A-infinity homomorphism $f=(f_n)_n:\mathcal{A}\rightarrow \mathcal{A}'$ is invertible if and only if $f_1$ is invertible as a $\field$-linear map.
\end{lemma}
\begin{definition}
    Let $\mathcal{A}$ be an A-infinity algebra. Then its homology $H^*(\mathcal{A})$ (with respect to the differential $m_1^{\mathcal{A}}$) obtains the structure of a graded associative (not necessarily unital) algebra via 
    \begin{align*}
        (a+\im(m_1))\cdot (b+\im(m_1))=m_2(a,b)+\im(m_1).
    \end{align*}
    Moreover, $H^*$ defines a functor from A-infinity algebras to (not necessarily unital) associative graded algebras.\\
    An A-infinity homomorphism $f=(f_n)_n$ is called an A-infinity quasi-isomorphism if $H^*(f_1)$ is an isomorphism of graded vector spaces.
    Similarly, $H^0$ defines a functor from A-infinity algebras to (not necessarily unital) associative algebras.\\
\end{definition}
The following theorem is known as Kadeishvili's theorem, or alternatively, as the homological perturbation lemma, and is an important statement in the study of A-infinity algebras.
\begin{theorem}\cite[Theorem 1]{kadeishvili}(see also \cite[Theorem 3.4]{merkulov})\label{thm_kadeishvili}
    Let $\mathcal{A}$ be an A-infinity algebra. Then  $H^*(\mathcal{A})$ obtains the structure of an A-infinity algebra such that there is a quasi-isomorphism $f=(f_n)_n:H^*(\mathcal{A})\rightarrow \mathcal{A}$
    with $H^*(f_1)=\id_{H^*(\mathcal{A})}$. Moreover, this structure is unique up to A-infinity isomorphism.
\end{theorem}
\begin{definition}
  An A-infinity algebra strictly unital over $L$ is an A-infinity algebra $\mathcal{A}$ together with a strict $A$-infinity homomorphism $L\rightarrow \mathcal{A}$ such that for all $x_1,\dots x_n\in \mathcal{A}$ with $x_j\in L$ for some $j$, we have
  \begin{align*}
    m_n(x_1\otimes\dots\otimes x_n)=0
  \end{align*} 
  unless $n=2$, and 
  \begin{align*}
    m_2(x\otimes 1_L)=m_2(1_L\otimes x)=x
  \end{align*}
  for $x\in \mathcal{A}$, where we identify $L$ with its image in $\mathcal{A}$.\\
	An A-infinity homomorphism $f:\mathcal{A}\rightarrow\mathcal{A}'$ between two A-infinity algebras which are strictly unital over $L$ is called strictly unital over $L$ if the diagram 
\[\begin{tikzcd}[ampersand replacement=\&]
	{\mathcal{A}} \&\& {\mathcal{A}'} \\
	\& L
	\arrow["f", from=1-1, to=1-3]
	\arrow[from=2-2, to=1-1]
	\arrow[from=2-2, to=1-3]
\end{tikzcd}\]
commutes.
\end{definition}
\begin{remark}\label{remark_L-L-bimods}
Note that if $\mathcal{A}$ is strictly unital over $L$, then for all $y\in L$, $x_1, \dots, x_n \in \mathcal{A}$ it follows from the A-infinity equation that
\begin{align*}
	m_{n} (x_1\otimes \dots \otimes m_2(x_i\otimes y)\otimes x_{i+1}\otimes\dots\otimes x_n )=m_{n} (x_1\otimes \dots \otimes x_i\otimes m_2(y\otimes x_{i+1})\otimes\dots\otimes x_n ).
\end{align*} 
Thus, A-infinity algebras strictly unital over $L$  are in one-to-one correspondence with strictly unital A-infinity algebras in the monoidal category of $L$-$L$-bimodules.
Since this is a semisimple monoidal category, Since this is a semisimple monoidal category, \cite[Proposition 3.2.4.1]{LefHas} implies that Theorem \ref{thm_kadeishvili} holds analogously in the category of A-infinity algebras strictly unital over $L$, and moreover, by \cite[Lemma 3.2.4.5]{LefHas}, the given quasi-equivalence has a quasi-inverse that is strictly unital over $L$.
Additionally, by Lemma \cite[Lemma 3.2.4.6]{LefHas}, Lemma \ref{lemma_ainftyhominvertible} also holds analogously in the category of A-infinity algebras strictly unital over $L$. \\ 
All A-infinity algebras we consider are strictly unital over $L$; note that the embedding $L\rightarrow \mathcal{A}$ is part of the data. We denote by $\Ainfunit$ the category of strictly unital A-infinity algebras over $L$, whose morphism are the A-infinity homomorphisms which are strictly unital over $L$.
\end{remark}
\begin{example}\label{example_Ext}
	Let $\mathcal{A}$ be a dg-algebra strictly unital over $L$. Then by Kadeishvili's theorem \ref{thm_kadeishvili}, \ref{remark_L-L-bimods}, $H^*(\mathcal{A})$ obtains the structure of an A-infinity algebra strictly unital over $L$ and there is a quasi-isomorphism $f=(f_n)_n:H^*(\mathcal{A})\rightarrow \mathcal{A}$ which is strictly unital over $L$ and has a strictly unital quasi-inverse, such that $H^*(f_1)=\id_{H^*(\mathcal{A})}$.\\
	In particular, consider an exact functor $T:\modu R\rightarrow \modu R'$ preserving projective modules and an $R$-module $M=\bigoplus_{i=1}^n M_i$. Let $P^\cdot(M)$ be a projective resolution of $M$. Then $T(P^\cdot(M))$ is a projective resolution of $T(M)$ and by \ref{remark_long}, $T$ induces a a dg-algebra homomorphism
	\begin{align*}
		\End_R^*(P^\cdot (M))\rightarrow \End_R^*(T(P^\cdot (M))), \varphi\mapsto T(\varphi)
	\end{align*}
	which is strictly unital over $L=k^n$. By the above, $\Ext_R^*(M, M)\cong H^*(\End_R^*(P^\cdot(M)))$ and $\Ext_{R'}^*(T(M), T(M))\cong H^*(\End_{R'}^*(T(P^\cdot (M)))$ thus obtain the structure of A-infinity algebras strictly unital over $L$, and we have $A$-infinity quasi-isomorphisms
	\begin{align*}
		f:\Ext_R^*(M, M)\rightarrow \End_R^*(P^\cdot (M))
	\end{align*}
	and
	\begin{align*}
		f':\Ext_{R'}^*(T(M), T(M))\rightarrow \End_{R'}^*(T(P^\cdot (M)))
	\end{align*}
	which are strictly unital over $L$ with strictly unital quasi-inverses.\\ Hence they give rise to an A-infinity homomorphism
	\begin{align*}
		g:\Ext_R^*(M, M)\rightarrow \Ext_{R'}^*(T(M), T(M))
	\end{align*}
	which is strictly unital over $L$ such that the diagram
\[\begin{tikzcd}[ampersand replacement=\&]
	{\End_R^*(P^\cdot(M))} \&\& {\End_{R'}^*(T(P^\cdot(M)))} \\
	{\Ext_R^*(M, M)} \&\& {\Ext^*_{R'}(T(M), T(M))}
	\arrow["{\varphi\mapsto T(\varphi)}", from=1-1, to=1-3]
	\arrow["{f'}"', from=2-3, to=1-3]
	\arrow["f", from=2-1, to=1-1]
	\arrow["g", from=2-1, to=2-3]
\end{tikzcd}\]
commutes. In particular, by the composition law for A-infinity homomorphisms, $$f'_1\circ g_1([\varphi])=T(f_1([\varphi])),$$ so that
\begin{align*}
    g_1([\varphi])=H^*(f_1')(g_1([\varphi]))= H^*(f_1'\circ g_1)([\varphi])=[T(\varphi)].
\end{align*}
\end{example}
\subsection{A-infinity categories}
\begin{definition}
    An A-infinity category $\mathcal{A}$  (over $\field$) is given by a class $\mathcal{A}$ of objects and for any two objects $X, Y\in \mathcal{A}$ a graded vector space $\mathcal{A}(X,Y)=\bigoplus_{n\in \mathbb{Z}}\mathcal{A}(X,Y)_n$ together with a collection of maps
    \begin{align*}
        (m_n^{\mathcal{A}})_{X_0, X_1, \dots , X_n}:\mathcal{A}(X_{n-1}, X_n)\otimes \mathcal{A}(X_{n-2}, X_{n-1})\otimes \dots \otimes \mathcal{A}(X_0, X_1)\rightarrow \mathcal{A}(X_0, X_n)
    \end{align*}
    for any objects $X_0,\dots , X_n$, which are homogeneous of degree $2-n$ such that the A-infinity equation
     \begin{align*}
    \sum_{r+s+t=n}(-1)^{rs+t}m_{r+t+1}(1^{\otimes r}\otimes m_s\otimes 1^{\otimes t})=0.
  \end{align*}
    is satisfied. Note that we usually omit the second subscript, and, if the A-infinity algebra in question  is clear, also the superscript of $(m_n^{\mathcal{A}})_{X_1, \dots , X_n}$.\\
    An A-infinity functor $F:\mathcal{A}\rightarrow \mathcal{A}'$ is a collection $F=(F_n)_{n\geq 0}$ of maps
    \begin{align*}
        F_0:\mathcal{A}\rightarrow \mathcal{A}'\\
        (F_n)_{X_0, \dots, X_n}: \mathcal{A}(X_{n-1}, X_n)\otimes \mathcal{A}(X_{n-2}, X_{n-1})\otimes \dots \otimes \mathcal{A}(X_0, X_1)\rightarrow \mathcal{A}'(F_0(X_0), F_0(X_n)) \textup{ for }n\geq 1,
    \end{align*}
    where $F_n$ is homogeneous of degree $1-n$ for $n\geq 1$ and such that for all $n\geq 1$
    \begin{align*}
        \sum_{r+s+t=n, s\geq 1}(-1)^{rs+t}F_{r+t+1}(1^{\otimes r}\otimes m_s\otimes 1^{\otimes t})=\sum_{j_1+\dots+ j_k=n, j_l\geq 1}(-1)^{\sum_{l=1}^k(1-j_l)\sum_{l'=1}^l j_{l'}}m_k(F_{j_1}\otimes \dots \otimes F_{j_k}).
    \end{align*}
    If $f, g$ are A-infinity functors then their composition is the A-infinity functor given by
    \begin{align*}
        (F\circ G)_0&=F_0\circ G_0\\
        (F\circ G)_n&=\sum_{j_1+\dots+j_k=n, j_l\geq 1}(-1)^{\sum_{l=1}^k(1-j_l)\sum_{l'=1}^l j_{l'}}F_k(G_{l_1}\otimes\dots\otimes G_{l_k}) \textup{ for }n\geq 1.
    \end{align*}
\end{definition}
\begin{remark}
    Let $\mathcal{A}$ be an A-infinity category. Then for every object $X\in \mathcal{A}$, $((m_n^{\mathcal{A}})_{X, \dots , X})_{n\in \mathbb{N}}$ endows $\mathcal{A}(X,X)$ with the structure of an A-infinity algebra.
\end{remark}
\begin{definition}
   An A-infinity category $\mathcal{A}$ is called strictly unital if for every object $X$ there is $\id_X\in \mathcal{A}(X,X)_0$ such that
   \begin{itemize}
    \item $m_2^{\mathcal{A}}(y\otimes \id_X)=y$ for all $y\in \mathcal{A}(X,Y)$,
    \item $m_2^{\mathcal{A}}(\id_X\otimes y)=y$ for all $y\in \mathcal{A}(Y,X)$,
    \item $m_n^{\mathcal{A}}(y_n\otimes \dots \otimes y_{k+1}\otimes \id_X\otimes y_{k-1}\otimes \dots \otimes y_1)=0$ for all $n\neq 2$, $y_i\in \mathcal{A}(X_{i-1}, X_i)$, where $X_k=X_{k-1}=X$.
   \end{itemize}
    An A-infinity functor $F:\mathcal{A}\rightarrow \mathcal{A}'$ between two strictly unital A-infinity categories is called strictly unital if
    \begin{itemize}
        \item $F_1(\id_X)=\id_{F_0(X)}$ for all $X\in \mathcal{A}$,
        \item $F_n(y_n\otimes \dots \otimes y_{k+1}\otimes \id_X\otimes ,y_{k-1}\otimes \dots \otimes y_1)=0$ for all $n>1$, $y_i\in \mathcal{A}(X_{i-1}, X_i)$, where $X_k=X_{k-1}=X$.
    \end{itemize}
    
\end{definition}
\begin{definition}\label{definition_homology_ainfty_cat}
    Let $\mathcal{A}$ be a strictly unital A-infinity category. Then $H^*(\mathcal{A})$ is the ($\mathbb{Z}$-graded) category whose objects are the objects of $\mathcal{A}$, whose morphisms are given by 
    \begin{align*}
        H^*(\mathcal{A})(X,Y)=H^*(\mathcal{A}(X, Y), (m_1^{\mathcal{A}})_{X,Y})
    \end{align*}
    and where composition of morphisms is induced by $m_2^{\mathcal{A}}$.
    A strictly unital A-infinity functor $F:\mathcal{A}\rightarrow \mathcal{A}'$ gives rise to a functor
    \begin{align*}
        H^*(F):H^*(\mathcal{A})\rightarrow H^*(\mathcal{A}'), X\mapsto f_0(X), \varphi+ \im(m_1)\mapsto f_1(\varphi)+\im(m_1). 
    \end{align*}
    of $\mathbb{Z}$-graded categories.
    In this way, $H^*$ becomes a functor from strictly unital A-infinity categories to $\mathbb{Z}$-graded categories.
    $F$ is called a quasi-equivalence if $H^*(F)$ is an equivalence of categories.\\
    Similarly, $H^0$ defines a functor from strictly unital A-infinity categories to  categories.
\end{definition}

\subsection{Twisted Modules}
\begin{definition}\cite[Definition 7.2.0.2, 6.1.2.2]{LefHas}
	Let $\mathcal{A}$ be an A-infinity algebra strictly unital over $L$. The A-infinity category $\add^{\mathbb{Z}}(\mathcal{A})$ is the A-infinity category whose objects are graded $L$-modules and whose homomorphisms are given by
	\begin{align*}
		\add^{\mathbb{Z}}(\mathcal{A})(X,Y)=\mathcal{A}\otimes_{L\otimes L^{\op}}\End_{\field}(X, Y)
	\end{align*}
	with the grading being the usual grading on the tensor product of graded spaces, and the A-infinity multiplications being
	\begin{align*}
		m_n^{\add^{\mathbb{Z}}(\mathcal{A})}((a_1\otimes g_1)\otimes \dots\otimes (a_n\otimes g_n)):=(-1)^{n \sum_{i=1}^n|g_i|+\sum_{i<j} |a_i||g_j|}m_n(a_1\otimes\dots\otimes a_n)\otimes (g_1\circ\dots\circ g_n).
	\end{align*}
	The category of twisted complexes over $\mathcal{A}$ is the A-infinity category whose objects are pairs $(X, w_X)$ of a graded $L$-module $X$ and an element $w_X\in (\add^{\mathbb{Z}}(\mathcal{A})(X,Y))_1$ such that
	\begin{enumerate}
		\item $w_X$ is triangular
		\item \begin{align*}
			\sum_{n=1}^\infty m_n^{\add^{\mathbb{Z}}(\mathcal{A})}(w_X^{\otimes n})=0,
		\end{align*}
    where the sum is finite since  $w_X$ is triangular,
	\end{enumerate}
	with A-infinity multiplications given by 
	\begin{align*}
		m_n^{\tw}(x_1\otimes \dots \otimes x_n):=\sum_{k=n}^\infty \sum_{j_0+\dots +j_n=k-n}(-1)^{\sum_{l=1}^n lj_l} m_k^{\add{\mathbb{Z}}(\mathcal{A})}(w_{X_0}^{\otimes j_0}\otimes x_1\otimes w_{X_1}^{\otimes j_1}\otimes \dots\otimes  w_{X_{n-1}}^{\otimes j_{n-1}}\otimes x_n \otimes w_{X_n}^{j_n}).
	\end{align*}
	 The category of twisted modules is the full subcategory of $\tw(\mathcal{A})$ given by the objects $(X, w_X)$ where $X$ is concentrated in degree zero.
\end{definition}
\begin{remark}\label{remark_A-infty-cat_end}
    Note that we have a strict A-infinity isomorphism
    \begin{align*}
        \mathcal{A}\rightarrow \add^{\mathbb{Z}}(\mathcal{A})(L, L)=\mathcal{A}\otimes_{L\otimes L^{\op}}\End_{\field}(L), a\mapsto a\otimes\id_L,
    \end{align*}
    where $L$ is viewed as being concentrated in degree zero.
\end{remark}
\begin{lemma}\label{lemma_twfunctor}\cite[Section 7.2, p.166]{LefHas}
	Let $f:\mathcal{A}\rightarrow\mathcal{A}'$ be an A-infinity homomorphism strictly unital over $L$. Then we obtain induced A-infinity functors
	\begin{align*}
		\add^{\mathbb{Z}}(f):\add^{\mathbb{Z}}(\mathcal{A})&\rightarrow \add^{\mathbb{Z}}(\mathcal{A}'),\\
		 X&\mapsto X,\\
		 \add^{\mathbb{Z}}(f)_n((a_1\otimes g_1)\otimes \dots \otimes (a_n\otimes g_n))&=(-1)^{(n-1)\sum_{i=1}^n|g_i|+\sum_{i< j}|a_i||g_j|}f_n(a_1\otimes \dots \otimes a_n)\otimes (g_1\circ \dots\circ g_n),
	\end{align*}
	as well as
	\begin{align*}
		\tw(f):\tw(\mathcal{A})&\rightarrow \tw(\mathcal{A}'),\\
		 (X, w_X)&\mapsto \left(X, \sum_{n=1}^\infty \add^{\mathbb{Z}}(f)_n(w_X)\right),\\
		 \tw(f)_n(x_1\otimes \dots \otimes x_n)&=\sum_{k=n}^\infty\sum_{j_0+\dots+j_n=k}(-1)^{\sum_{l=1}^n lj_l}\add^{\mathbb{Z}}(f)_k(w_X^{\otimes j_0}\otimes x_1\otimes w_X^{\otimes j_1}\dots \otimes w_X^{\otimes j_{n-1}}\otimes x_n\otimes  w_X^{\otimes j_{n}} ),
	\end{align*}
	and its restriction
	\begin{align*}
		\twmod(f):\twmod(\mathcal{A})\rightarrow \twmod(\mathcal{A}'), (X, w_X)\mapsto \left(X, \sum_n \add^{\mathbb{Z}}(f)_n(w_X)\right), \twmod(f)_n=\tw(f)_n.
	\end{align*}
	This gives rise to functors
	\begin{align*}
		\add^{\mathbb{Z}}, \tw, \twmod
	\end{align*}
	from A-infinity algebras strictly unital over $L$ to strictly unital A-infinity categories.\\
	 Note that the restriction of $\twmod$ as defined here to the subcategory of $\Ainfunit$ given by strictly unital dg-algebras and strictly unital dg-algebra homomorphisms is indeed the functor $\twmod$ from the previous section.
\end{lemma}
Of most importance to us is the following result, which can be found in \cite[Lemma 3.25]{Seidel}:

\begin{proposition}\label{proposition_twmod_equiv}
	Let $f:\mathcal{A}\rightarrow \mathcal{A}'$ be an A-infinity quasi-isomorphism. Then $\add^{\mathbb{Z}}(f), \tw(f)$ and $\twmod(f)$ are quasi-equivalences.
\end{proposition}
\section{Truncation of A-infinity algebras}\label{section_trunc}
In this section, we consider truncation and augmentation of A-infinity algebras, and show a universal property of the augmentation. This is useful later, since for a basic finite-dimensional algebra $B$ with semisimple part $L^B$,  $\Ext_B^*(L^B, L^B)$ is augmented over $L^B$.
\begin{definition}
    Denote by $\Ainfunit^+$ the full subcategory of $\Ainfunit$ given by the A-infinity algebras $\mathcal{A}$ such that $\mathcal{A}_n=(0)$ for $n<0$ and $\mathcal{A}_0=L$. Such A-infinity algebras are sometimes referred to as strictly coconnected, see for example \cite{coconnected}.
    Let 
    \begin{align*}
        I_L:\Ainfunit^+\rightarrow \Ainfunit
    \end{align*}
    be the canonical embedding.
\end{definition}
\begin{proposition}\label{proposition_restriction}
	There is a well-defined functor
	\begin{align*}
		\trunc_L: \Ainfunit\rightarrow \Ainfunit^+, \mathcal{A}\mapsto L\oplus \mathcal{A}^{>0}, f\mapsto f_{|L\oplus\mathcal{A}^{>0}}
	\end{align*}
\end{proposition}
\begin{proof}
	Suppose $\mathcal{A}$ is a strictly unital A-infinity algebra over $L$.
	Then for $n\in \mathbb{N}$, $x_1, \dots, x_n \in L\oplus \mathcal{A}^{>0}$ we have
	\begin{itemize}
		\item $|m_n(x_1\otimes\dots\otimes x_n)|\geq n+|m_n|=n+2-n=2$; for  $x_1, \dots , x_j\in \mathcal{A}^{>0}$,
		\item $m_n(x_1\otimes\dots\otimes x_n)=0$ for $n=1$ or $n>2$ and $x_j\in L$ for some $1\leq j\leq n$
		\item $|m_2(x_1\otimes x_2)|\geq 1$ for $\max\{|x_1|, |x_2|\}\geq 1$, and
		\item $m_2(x_1, x_2)\in L$ for $x_1, x_2\in L$.
	\end{itemize}
	Hence $L\oplus \mathcal{A}^{>0}$ is a well-defined strictly unital A-infinity subalgebra of $\mathcal{A}$.
	Suppose \begin{align*}
		f=(f_n)_n: \mathcal{A}\rightarrow \mathcal{B}
	\end{align*}
	is a strictly unital A-infinity homomorphism.
	Then for $n\in \mathbb{N}$, $x_1, \dots, x_n \in L\oplus \mathcal{A}^{>0}$ we have
	\begin{itemize}
		\item $|f_n(x_1\otimes\dots\otimes x_n)|\geq n+|f_n|=n+1-n=1$; for  $x_1, \dots , x_j\in \mathcal{A}^{>0}$,
		\item $f_n(x_1\otimes\dots\otimes x_n)=0$ for $n>2$ and $x_j\in L$ for some $1\leq j\leq n$,
		\item $f_1(x_1)=x_1$ for $x\in L$, and
		\item $|f_1(x_1)|=|x_1|\geq 1$ for $|x_1|\geq 1$.
	\end{itemize}
	Thus 
	\begin{align*}
		f_{|L\oplus\mathcal{A}^{>0}}:L\oplus \mathcal{A}^{>0}\rightarrow L\oplus \mathcal{B}^{>0}
	\end{align*}
	is a well-defined strictly unital A-infinity homomorphism over $L$.
\end{proof}
Moreover, we have the following:
\begin{proposition}\label{proposition_adjoint_truncation}
    $\trunc_L$ is right adjoint to the canonical inclusion  $I_L:\Ainfunit^+\rightarrow \Ainfunit$, with the counit of the adjunction given by the canonical inclusion
    \begin{align*}
        \varepsilon:I_L\circ \trunc_L\rightarrow \id_{\Ainfunit}, \\
         \varepsilon_\mathcal{A}:L\oplus \mathcal{A}^{>0}\rightarrow \mathcal{A},
         (\varepsilon_\mathcal{A})_1(a)=a, (\varepsilon_{\mathcal{A}})_n=0 \textup{ for }n>1.
    \end{align*}
    and the unit given by the identity
    \begin{align*}
        \eta:\id_{\Ainfunit^+}\rightarrow \trunc_L\circ I_L, \\
         \eta_\mathcal{B}=\id_{\mathcal{B}}:\mathcal{B}=L\oplus \mathcal{B}^{>0}\rightarrow L\oplus \mathcal{B}^{>0}.
    \end{align*}
\end{proposition}
\begin{proof}
Clearly, both $\varepsilon$ and $\eta$ are natural transformations. Moreover, for $\mathcal{A}\in \Ainfunit$
\begin{align*}
\trunc_L(\varepsilon_{\mathcal{A}})\circ \eta_{\trunc_L(\mathcal{A})}=\trunc_L(\id_{\mathcal{A}})\circ \id_{\trunc_L(\mathcal{A})}=\id_{\trunc_L(\mathcal{A})}
\end{align*}
and for any $\mathcal{B}=L\oplus \mathcal{B}^{>0}\in \Ainfunit^+$
\begin{equation*}
   \varepsilon_{I_L({\mathcal{B}})}\circ I_L(\eta_{\mathcal{B}}) = \varepsilon_{I_L({\mathcal{B}})}\circ I_L(\id_{\mathcal{B}})= \varepsilon_{I_L({\mathcal{B}})}=\id_{I_L({\mathcal{B}})}.\qedhere
\end{equation*}
\end{proof}
Reformulating this adjunction gives us the following property:
\begin{corollary}
    Let $\mathcal{A}\in \Ainfunit$ be an A-infinity algebra strictly unital over $L$. Then $\trunc_L(\mathcal{A})$ has the following universal property:\\
    For any $\mathcal{B}\in \Ainfunit^+$ and any A-infinity homomorphism $f:\mathcal{B}\rightarrow \mathcal{A}$ there is a unique A-infinity homomorphism $h:\mathcal{B}\rightarrow \trunc_L \mathcal{A}$
    such that $f=\varepsilon_{\mathcal{A}}\circ h$.
\end{corollary}
Since A-infinity homomorphisms are isomorphisms whenever they are isomorphisms in degree zero \ref{lemma_ainftyhominvertible}, we have the following:
\begin{proposition}\label{proposition_ainftyhoms}
  Suppose $f:\mathcal{B}\rightarrow \mathcal{A}, f':\mathcal{B}'\rightarrow \mathcal{A}$ are two strictly unital A-infinity homomorphisms over $L$ between strictly unital A-infinity algebras over $L$ concentrated in non-negative degree such that $f_1$ is an isomorphism in degree greater than zero and $ \mathcal{B}$ and $\mathcal{B}'$ have only $L$ in degree zero. Then there is a strictly unital A-infinity homomorphism $g:\mathcal{B}'\rightarrow \mathcal{B}$ over $L$ that makes the diagram 
\[\begin{tikzcd}[ampersand replacement=\&]
	{\mathcal{B}'} \&\& {\mathcal{B}} \\
	\& {\mathcal{A}}
	\arrow["f"', from=1-3, to=2-2]
	\arrow["{f'}", from=1-1, to=2-2]
	\arrow["g", from=1-1, to=1-3]
\end{tikzcd}\]
	commute. Moreover, if $f_1'$ is also an isomorphism in positive degree, then $g$ is an isomorphism.
\end{proposition}
\begin{proof}
    By the universal property of $\trunc_L(\mathcal{A})$, we have an A-infinity homomorphism $h:\mathcal{B}\rightarrow \trunc_L(\mathcal{A})$ such that $f=\varepsilon_{\mathcal{A}}\circ h$. In particular,
    $f_1=(\varepsilon_{\mathcal{A}})_1\circ h_1$. Since $f_1$ is by assumption an isomorphism in degree greater than zero, and $(\varepsilon_{\mathcal{A}})_1$ is by definition the canonical embedding
    $$L\oplus \mathcal{A}^{>0}\rightarrow \mathcal{A}$$ and thus an isomorphism in degree greater than zero, $h_1$ is an isomorphism in degree greater than zero. Moreover, since $h$ is strictly unital, $h_1(L)=L$. Thus
    \begin{align*}
        h_1:\mathcal{B}=L\oplus \mathcal{B}^{>0}\rightarrow \trunc_L(\mathcal{A})=L\oplus \mathcal{A}^{>0}
    \end{align*}
    is an isomorphism, so by  \ref{lemma_ainftyhominvertible}, $h$ is an A-infinity isomorphism.\\
    On the other hand, by the universal property of $\trunc_L(\mathcal{A})$, there is $h':\mathcal{B}'\rightarrow \trunc_L(\mathcal{A})$ such that $f'=\varepsilon_{\mathcal{A}}\circ h'$.
    Now for $g:=h^{-1}\circ h'$ we have
    \begin{align*}
        f\circ g=f\circ h^{-1}\circ h'=\varepsilon_{\mathcal{A}}\circ h\circ h^{-1}\circ h'=\varepsilon_{\mathcal{A}}\circ h'=f'.
    \end{align*}
    Moreover, if $f'$ is an isomorphism in degree greater than zero, then, as for $f$, we obtain that $h'$ is an isomorphism, so that $g=h^{-1}\circ h'$ is an isomorphism.
\end{proof}
\section{Proof of the main result}\label{section_proofmain}
In this section, we prove the main result on conjugation-uniqueness of basic regular exact Borel subalgebras, Theorem \ref{theorem_conjugation}.
\begin{theorem}\label{theorem_commutative_main}
Let $A$ be a finite-dimensional $\field$-algebra, and let $B, B'$ be two basic regular exact subalgebras of $A$ with simple modules $\{L_1^B, \dots , L_n^B\}=\Sim(B)$ and $\{L_1^{B'}, \dots , L_n^{B'}\}=\Sim(B')$ such that for every $1\leq i\leq n$ we have
\begin{align*}
	\Delta_i^A:=A\otimes_B L_i^B\cong A\otimes_{B'}L_i^{B'}.
\end{align*}
  Then there is an equivalence $G:\modu B\rightarrow \modu B'$ that makes the diagram 
\[\begin{tikzcd}[ampersand replacement=\&]
	{\modu B'} \&\& {\modu B} \\
	\& {\modu A}
	\arrow["{A\otimes_B -}", from=1-3, to=2-2]
	\arrow["{A\otimes_{B'}-}"', from=1-1, to=2-2]
	\arrow["G", from=1-1, to=1-3]
\end{tikzcd}\]
	commute up to natural isomorphism.
\end{theorem}
\begin{proof}
Let $L^B=\bigoplus_{i=1}^n L_i^B$,  $L^{B'}=\bigoplus_{i=1}^n L_i^{B'}$ and $\Delta^A=\bigoplus_{i=1}^n \Delta_i^A=\Delta$ and let $L=\field^n$.
	As in Example \ref{example_Ext} we have a diagram 
\[\begin{tikzcd}[ampersand replacement=\&]
	{\Ext_{B'}^*(L_{B'}, L_{B'})} \&\& {\Ext_B^*(L_B, L_B)} \\
	\& {\Ext_A^*(\Delta, \Delta)}
	\arrow["{f'}"', from=1-1, to=2-2]
	\arrow["f", from=1-3, to=2-2]
\end{tikzcd}\]
where $f, f'$ are strictly unital A-infinity homomorphisms over $L$ which are isomorphisms in degree greater than zero.\\
	Since ${\Ext_B^*(L_B, L_B)}$ and ${\Ext_{B'}^*(L_{B'}, L_{B'})}$ have only $L$ in degree zero, Proposition \ref{proposition_ainftyhoms} yields a strictly unital A-infinity isomorphism $g: \Ext^*_{B'}(L^{B'}, L^{B'})\rightarrow \Ext^*_B(L^B, L^B)$ over $L$ that makes the diagram 
\[\begin{tikzcd}[ampersand replacement=\&]
	{\Ext_{B'}^*(L_{B'}, L_{B'})} \&\& {\Ext_B^*(L_B, L_B)} \\
	\& {\Ext_A^*(\Delta, \Delta)}
	\arrow["{f'}"', from=1-1, to=2-2]
	\arrow["f", from=1-3, to=2-2]
	\arrow["g", from=1-1, to=1-3]
\end{tikzcd}\]
	commute.\\
	Hence we obtain a commutative diagram
\[\begin{tikzcd}[ampersand replacement=\&]
	{H^0(\twmod(\Ext_{B'}^*(L_{B'}, L_{B'})))} \&\& {H^0(\twmod(\Ext_B^*(L_B, L_B)))} \\
	\& {H^0(\twmod(\Ext_A^*(\Delta, \Delta)))}
	\arrow["{H^0(\twmod(g))}"'{pos=0.2}, from=1-1, to=2-2]
	\arrow["{H^0(\twmod(f))}"{pos=0.2}, from=1-3, to=2-2]
	\arrow["{H^0(\twmod(g))}", from=1-1, to=1-3]
\end{tikzcd}\]
Using the commutative diagram \eqref{diagram_1} which was established in Section \ref{section_dg} and Section \ref{section_ainfty} for $B$ and $B'$, as well as Proposition \ref{proposition_twmod_equiv}, we thus obtain the result.
\end{proof}
\begin{theorem}\label{theorem_conjugation} Let $A$ be a finite-dimensional $\field$-algebra, and let $B, B'$ be two basic regular exact subalgebras of $A$ with simple modules $\{L_1^B, \dots , L_n^B\}=\Sim(B)$ and $\{L_1^{B'}, \dots , L_n^{B'}\}=\Sim(B')$ such that for every $1\leq i\leq n$ we have
\begin{align*}
	\Delta_i^A:=A\otimes_B L_i^B\cong A\otimes_{B'}L_i^{B'}.
\end{align*} 
Then there is an $a\in A^{\times}$ such that $B'=a^{-1}Ba$.
\end{theorem}
\begin{proof}
By Theorem \ref{theorem_commutative_main}, there is an equivalence $G:\modu B'\rightarrow \modu B$
 that makes the diagram 
\[\begin{tikzcd}[ampersand replacement=\&]
	{\modu B'} \&\& {\modu B} \\
	\& {\modu A}
	\arrow["{A\otimes_B -}", from=1-3, to=2-2]
	\arrow["{A\otimes_{B'}-}"', from=1-1, to=2-2]
	\arrow["G", from=1-1, to=1-3]
\end{tikzcd}\]
	commute up to natural isomorphism.
	Since $G$ is an equivalence and $B$ and $B'$ are basic projective generators in $\modu B$, $\modu B'$ respectively, $G(B')\cong B$.
	We pick an isomorphism $\xi:G(B')\rightarrow B$ and define $G'$ via structure transport to be the unique functor that makes the family of maps $(\tau_X)_X$ given by $\tau_X=\id_X$ for all $X\neq B'$ and $\tau_{B'}=\xi$ a natural isomorphism between $G$ and $G'$.
 More concretely, we set
	\begin{align*}
		G':\modu B'\rightarrow \modu B,\\
		 G'(X):=G(X) \textup{ for all objects }X\neq B'\\
		G'(B')=B\\
		 G'(f)=G(f) \textup{ for all }f\in \Hom_{B'}(X, Y) \textup{ where }X, Y\neq B'\\
		G'(f)=\xi\circ G(f) \textup{ for all }f\in \Hom_{B'}(X, B') \textup{ where }X \neq B'\\
		G'(f)=G(f) \circ \xi^{-1} \textup{ for all }f\in \Hom_{B'}(B', Y) \textup{ where }Y \neq B'\\
		G'(f)=\xi\circ G(f)\circ \xi^{-1} \textup{ for all }f\in \Hom_{B'}(B', B').
	\end{align*}
	Then, since $G'$ is by definition naturally isomorphic to $G$, the diagram 
\[\begin{tikzcd}[ampersand replacement=\&]
	{\modu B'} \&\& {\modu B} \\
	\& {\modu A}
	\arrow["{G'}", from=1-1, to=1-3]
	\arrow["{A\otimes_B -}", from=1-3, to=2-2]
	\arrow["{A\otimes_{B'}-}"', from=1-1, to=2-2]
\end{tikzcd}\]
	commutes up to natural isomorphism, so that, replacing $G$ by $G'$ if necessary, we can assume that $G(B')=B$.
 Denote by $\alpha:(A\otimes_{B'}-)\rightarrow (A\otimes_{B}-)\circ G$ the natural isomorphism making the diagram commute.
Then $\alpha$  gives an isomorphism of left $A$-modules $\alpha_{B'}:A\otimes_{B'} B'\rightarrow A\otimes_{B}B$ such that the diagram
\[\begin{tikzcd}[ampersand replacement=\&]
	{\End_{B'}(B')} \& {\End_A(A\otimes_{B'}B')} \\
	{\End_B(B)} \& {\End_A(A\otimes_{B}B)}
	\arrow["{f\mapsto G(f)}"', from=1-1, to=2-1]
	\arrow["{\id_A\otimes-}", from=1-1, to=1-2]
	\arrow["{\id_A\otimes-}"', from=2-1, to=2-2]
	\arrow["{\rho_{\alpha_{B'}}}", from=1-2, to=2-2]
\end{tikzcd}\]
commutes, where $\rho_{\alpha_{B'}}$ denotes conjugation by $\alpha_{B'}$. 
Composing with the canonical isomorphisms $e:A\otimes_B B\rightarrow A$, and $e':A\otimes_{B'}B'\rightarrow A$ of left $A$-modules and using that we have algebra isomorphisms $B^{\op}\rightarrow \End_B(B), b\mapsto r_b$, where $r_b$ denotes right multiplication by $b$, we obtain an isomorphism $\beta=e\circ \alpha_B\circ (e')^{-1}:A\rightarrow A$ of left $A$-modules such that the diagram
\[\begin{tikzcd}[ampersand replacement=\&]
	{(B')^{\op}} \&\& {\End_{B'}(B')} \& {\End_A(A\otimes_{B'}B')} \& {\End_A(A)} \&\& {A^{\op}} \\
	{B^{\op}} \&\& {\End_B(B)} \& {\End_A(A\otimes_B B)} \& {\End_A(A)} \&\& {A^{\op}}
	\arrow["{r_a\mapsto a}", from=2-5, to=2-7]
	\arrow["{r_a\mapsto a}", from=1-5, to=1-7]
	\arrow["{\rho_e}", from=2-4, to=2-5]
	\arrow["{\rho_{e'}}", from=1-4, to=1-5]
	\arrow["{\rho_{\alpha_{B'}}}"', from=1-4, to=2-4]
	\arrow["{\rho_{\beta}}"', from=1-5, to=2-5]
	\arrow["{b\mapsto r_{b}}", from=2-1, to=2-3]
	\arrow["{b'\mapsto r_{b'}}", from=1-1, to=1-3]
	\arrow["{\id_A\otimes-}", from=1-3, to=1-4]
	\arrow["{\id_A\otimes-}", from=2-3, to=2-4]
	\arrow["{f\mapsto G(f)}"', from=1-3, to=2-3]
	\arrow["\varphi"', from=1-7, to=2-7]
	\arrow["{g'}"', from=1-1, to=2-1]
\end{tikzcd}\]
commutes, where $g':(B')^{\op}\rightarrow B^{\op}$ and $\varphi:A^{\op}\rightarrow A^{\op}$ are the isomorphisms induced by the first, respectively last, square. Note that the composition of the vertical arrows are the canonical inclusions, and that, since $\beta$ is an automorphism of $A$ as an $A$-module, $\beta=r_{a}$ for some unit $a\in A^{\times}$. Since  
$$\rho_{r_a}(r_x)=r_a^{-1}\circ r_x\circ r_a=r_{axa^{-1}},$$
the induced isomorphism $\varphi$ is given by 
\begin{align*}
    \varphi(x)=axa^{-1}.
\end{align*}
Thus we obtain a commutative diagram
\[\begin{tikzcd}[ampersand replacement=\&]
	{B'} \&\& B \\
	A \&\& A
	\arrow[from=1-3, to=2-3]
	\arrow["{g'}", from=1-1, to=1-3]
	\arrow[from=1-1, to=2-1]
	\arrow["{\rho_{a^{-1}}}", from=2-1, to=2-3]
\end{tikzcd}\]
where the vertical arrows are the canonical inclusions and $g'$ is an isomorphism. Thus $B'=a^{-1}Ba$.
\end{proof}
\section{Strong Regular Exact Borel Subalgebras}\label{section_strong}
In this section, we consider the case where  $(A, \leq_A)$ is a quasi-hereditary algebra with a basic strong regular exact Borel subalgebra $B$, or equivalently, since strongness implies that $B$ is basic if and only if $A$ is basic, where $(A, \leq_A)$ is a basic quasi-hereditary algebra with a regular exact Borel subalgebra $B$. In general, the basic representative does not admit a regular exact Borel subalgebra \cite[Example]{Koenig}. In \cite{Conde}, Conde gives a numerical as well as a homological criterion for when it does, and proves that if the basic representative admits a regular exact Borel subalgebra, then so does any representative, merely by matching the multiplicities of the projectives. These exact Borel subalgebras then contain a maximal semisimple subalgebra of the algebra and are thus strong.\\
First, we establish some basic properties of strong regular exact Borel subalgebras. We expect these to be generally known to the community, however, as we were not able to find references for these statements, we provide proofs.
\begin{lemma}\label{lemma_strong_top}
    Let $(A, \leq_A)$ be a finite-dimensional quasi-hereditary algebra and  $B$ be a basic regular exact Borel subalgebra of $A$. Let $\Sim(B)=\{L_1^B, \dots , L_n^B\}$ and $\Sim(A)=\{L_1^A, \dots , L_n^A\}$ such that 
    \begin{align*}
        \Delta_i^A:=A\otimes_B L_i^B\cong \Delta^A(L_i^A).
    \end{align*}
  Let $M\in \modu B$. Then we have 
  \begin{align*}
      [\Top(A\otimes_B M):L_i^A]\geq[\Top(M):L_i^B]
  \end{align*} for every $1\leq i\leq n$.
  Moreover, if $B$ is strong, then 
   \begin{align*}
      [\Top(A\otimes_B M):L_i^A]=[\Top(M):L_i^B]
  \end{align*} for every $1\leq i\leq n$.
\end{lemma}
\begin{proof}
  We have a projection 
  \begin{align*}
   \pi: M\rightarrow \Top(M),
  \end{align*}
  and since $A\otimes_B -$ is exact, this gives rise to a projection
  \begin{align*}
    \id_A\otimes \pi:A\otimes_B M\rightarrow A\otimes_B \Top(M),
  \end{align*}
  so that $\Top(A\otimes_B \Top(M))$ is a direct summand of $\Top(A\otimes_B M)$.
  Since $B$ is an exact Borel subalgebra of $A$, we have 
  \begin{align*}
    A\otimes_B \Top(M)\cong \bigoplus_{i=1}^n [\Top(M):L_i^B]A\otimes_B L_i^B\cong \bigoplus_{i=1}^n [\Top(M):L_i^B]\Delta^A_i.
  \end{align*}
  so that $\Top(A\otimes_B \Top(M))\cong \bigoplus_{i=1}^n[\Top(M):L_i^B]L^A_i$. Hence
  \begin{align*}
    [\Top(A\otimes_B M):L_i^A]\geq [\Top(M):L_i^B]
  \end{align*}
  for every $i\in \{1,\dots , n\}$.\\
  Assume that $B$ is strong. Then, by Lemma \ref{lemma_radical}, we have $\rad(B)\subseteq \rad(A)$. Thus
  \begin{align*}
    A\otimes_B \rad(M)\subseteq \rad(A\otimes_B M)
  \end{align*} 
  and hence, since $A\otimes_B -$ is exact, $\Top(A\otimes_B M)$ is a semisimple factor module of 
 \begin{align*}
  A\otimes_B \Top(M)\cong \bigoplus_{i=1}^n [\Top(M):L_i^B]\Delta_i^A.
 \end{align*}
Thus $\Top(A\otimes_B M)$ is a direct summand of
\begin{align*}
  \Top\left(\bigoplus_{i=1}^n [\Top(M):L_i^B]\Delta_i^A\right)\cong \bigoplus_{i=1}^n [\Top(M):L_i^B]L_i^A,
\end{align*}
so that \begin{align*}
  [\Top(A\otimes_B M):L_i^A]\leq [\Top(M):L_i^B]
\end{align*}
for every $i\in \{1,\dots , n\}$.
\end{proof}
\begin{remark}
     Let $(A, \leq_A)$ be a finite-dimensional quasi-hereditary algebra and  $B$ be a basic regular exact Borel subalgebra of $A$.  Let $\Sim(B)=\{L_1^B, \dots , L_n^B\}$ and $\Sim(A)=\{L_1^A, \dots , L_n^A\}$ such that 
    \begin{align*}
        \Delta_i^A:=A\otimes_B L_i^B\cong \Delta^A(L_i^A).
    \end{align*} Denote by $P_i^A$ the projective cover $L_i^A$ and by $P_i^B$ the projective cover of $L_i^B$. Moreover, let $n_i^A$ be the multiplicity of $P_i^A$ in $A$. Then we have
     \begin{align*}
         n_i^A=[\Top(A):L_i^A]=[\Top(A\otimes_B B):L_i^A]
     \end{align*}
     and
     \begin{align*}
         [\Top(B):L_i^B]=1
     \end{align*}
     since $B$ is basic. Hence the equality 
      \begin{align*}
         [\Top(A\otimes_B M):L_i^A]=[\Top(M):L_i^B]
     \end{align*} in \ref{lemma_strong_top} holds for every $M\in \modu B$ if and only if $A$ is basic.
\end{remark}
The following theorem is a stronger version of our main theorem, Theorem \ref{theorem_conjugation}, in the case where $A$ is basic. Crucially, it is not assumed that $B'$ is regular. This result also follows from a result announced by Conde and König in conjunction with our main theorem, Theorem \ref{theorem_conjugation}. Here, we propose an alternative proof relying on similar methods as those in our main theorem.\\
The claim of this theorem, and in fact, even a stronger version, was also made in \cite{yuehui_zhang}. However, there are problems in the proof given therein. For example, in the proof of \cite[Lemma 2.2]{yuehui_zhang}, it was claimed that for any exact Borel subalgebra $B$ of $A=\field Q/I$ containing all idempotents corresponding to vertices of $Q$,  and for any path  $\alpha\in A$ between vertices $i$ and $j$ such that $\alpha\notin B$, we have $ \alpha \otimes_B L_i^B\neq 0$, where $L_i^B$ is the simple module corresponding to the vertex $i$. Note here that our left-right conventions are reversed. The following example shows that this claim does not hold in general:
\begin{example}\label{example_yuehui}
Consider the quiver $Q$ given by
\[\begin{tikzcd}[ampersand replacement=\&]
	1 \& 3 \& 2
	\arrow["\alpha", from=1-1, to=1-2]
	\arrow["\beta", from=1-2, to=1-3]
\end{tikzcd}\]
and let $A=\field Q$. Then we have an exact Borel subalgebra $B=\field Q'$ for $Q'$ given by
\[\begin{tikzcd}[ampersand replacement=\&]
	1 \& 3 \& 2
	\arrow["\alpha", from=1-1, to=1-2]
\end{tikzcd}\]
In particular, $\beta\alpha\notin B$, but $\beta\alpha\otimes_B L_1=0$, where $L_1^B$ is the  simple module corresponding to the vertex $1$.
\end{example}
\begin{theorem}\label{thm_strong}
Let $(A, \leq_A)$ be a finite-dimensional quasi-hereditary $\field$-algebra, $B$ be a strong basic regular exact Borel subalgebra and $B'$ be an exact Borel subalgebra of $A$.
  There is an $a\in A^\times$ such that $aB'a^{-1}\cong B$.
\end{theorem}
\begin{proof}
  Since $B, B'$ are exact Borel subalgebras of $A$, we have a diagram 
\[\begin{tikzcd}[ampersand replacement=\&]
	{\Ext_{B'}^*(L^{B'}, L^{B'})} \&\& {\Ext_B^*(L^B, L^B)} \\
	\& {\Ext_A^*(\Delta^A, \Delta^A)}
	\arrow["{[f]\mapsto [\id_A\otimes f]}"{pos=0}, from=1-3, to=2-2]
	\arrow["{[f]\mapsto [\id_A\otimes f]}"'{pos=0}, from=1-1, to=2-2]
\end{tikzcd}\]
  which, by Proposition \ref{proposition_ainftyhoms} and regularity of $B$, we can complete to a commutative diagram 
\[\begin{tikzcd}[ampersand replacement=\&]
	{\Ext_{B'}^*(L^{B'}, L^{B'})} \&\& {\Ext_B^*(L^B, L^B)} \\
	\& {\Ext_A^*(\Delta^A, \Delta^A)}
	\arrow["{[f]\mapsto [\id_A\otimes f]}"{pos=0}, from=1-3, to=2-2]
	\arrow["{[f]\mapsto [\id_A\otimes f]}"'{pos=0}, from=1-1, to=2-2]
	\arrow["g", from=1-1, to=1-3]
\end{tikzcd}\]
 where $g$ is an A-infinity homomorphism strictly unital over $L$. As in the proof of Theorem \ref{theorem_commutative_main} we use Proposition \ref{proposition_ainftyhoms}, to obtain a diagram 
\[\begin{tikzcd}[ampersand replacement=\&]
	{\modu B'} \&\& {\modu B} \\
	\& {\filt(\Delta^A)}
	\arrow["{A\otimes_{B'}-}"', from=1-1, to=2-2]
	\arrow["{A\otimes_B-}", from=1-3, to=2-2]
	\arrow["G", from=1-1, to=1-3]
\end{tikzcd}\]
that commutes up to natural isomorphism.
Note that $A\otimes_B G(B')\cong A\otimes_{B'}B'\cong A$. In particular, $\Top(A\otimes_B G(B'))=L^A$. By Lemma \ref{lemma_strong_top}, this implies that $\Top(G(B'))\cong L^B$. Hence there is an exact sequence 
\[\begin{tikzcd}[ampersand replacement=\&]
	{(0)} \& K \& B \& {G(B')} \& {(0)}
	\arrow["\pi", from=1-3, to=1-4]
	\arrow[from=1-4, to=1-5]
	\arrow["\iota", from=1-2, to=1-3]
	\arrow[from=1-1, to=1-2]
\end{tikzcd}\]
which, since $A\otimes_B -$ is exact, gives rise to an exact sequence
\[\begin{tikzcd}[ampersand replacement=\&]
	{(0)} \& {A\otimes_B K} \& {A\otimes_B B} \& {A\otimes_B G(B')} \& {(0)}
	\arrow["{\id_A\otimes\pi}", from=1-3, to=1-4]
	\arrow[from=1-4, to=1-5]
	\arrow["{\id_A\otimes\iota}", from=1-2, to=1-3]
	\arrow[from=1-1, to=1-2]
\end{tikzcd}\]
Since $A\otimes_B B\cong A\otimes_B G(B')$, these have the same dimension, so that $\id_A\otimes \pi$ is an isomorphism. Hence $A\otimes_B K=(0)$. Moreover, normality, or alternatively, \cite[Lemma 3.7]{BKK}, implies that $B$ is a direct summand of $A$ as a right $B$-module. This yields that $K=(0)$.
Thus $\pi$ is an isomorphism. Now we can argue as in the proof of Theorem \ref{theorem_conjugation} to assume without loss of generality that $G(B')=B$. Let $\alpha: (A\otimes_B -)\circ G\rightarrow (A\otimes_{B'}-)$ be the natural isomorphism between the functors. Then again as in Theorem \ref{theorem_conjugation}, we obtain an $a\in A^{\times}$ making the diagram 
\[\begin{tikzcd}[ampersand replacement=\&]
	{(B')^{\op}} \&\& {\End_{B'}(B')} \& {\End_A(A\otimes_{B'}B')} \& {\End_A(A)} \&\& {A^{\op}} \\
	{B^{\op}} \&\& {\End_B(B)} \& {\End_A(A\otimes_B B)} \& {\End_A(A)} \&\& {A^{\op}}
	\arrow["{r_a\mapsto a}", from=2-5, to=2-7]
	\arrow["{r_a\mapsto a}", from=1-5, to=1-7]
	\arrow["{\rho_e}", from=2-4, to=2-5]
	\arrow["{\rho_{e'}}", from=1-4, to=1-5]
	\arrow["{\rho_{\alpha_{B'}}}"', from=1-4, to=2-4]
	\arrow["{\rho_{\beta}}"', from=1-5, to=2-5]
	\arrow["{b\mapsto r_{b}}", from=2-1, to=2-3]
	\arrow["{b'\mapsto r_{b'}}", from=1-1, to=1-3]
	\arrow["{\id_A\otimes-}", from=1-3, to=1-4]
	\arrow["{\id_A\otimes-}", from=2-3, to=2-4]
	\arrow["{f\mapsto G(f)}"', from=1-3, to=2-3]
	\arrow["{\varphi=\rho_{a^{-1}}}"', from=1-7, to=2-7]
	\arrow["{g'}"', from=1-1, to=2-1]
\end{tikzcd}\]
commute, where $\rho_{\alpha_{B'}}$ is conjugation by $\alpha_{B'}$, i.e. $\rho_{\alpha_{B'}}(f)=\alpha_{B'}^{-1}\circ f\circ \alpha_{B'}$, $\rho_{e}$ is conjugation by the canonical isomorphism
\begin{align*}
	e: A\otimes_B B\rightarrow A, a\otimes b\mapsto ab
\end{align*}
and 
$\rho_{e'}$ is conjugation by the canonical isomorphism
\begin{align*}
	e': A\otimes_{B'} B'\rightarrow A, a\otimes b'\mapsto ab',
\end{align*}
and $g':(B')^{\op}\rightarrow B^{\op}$ and $\varphi:A^{\op}\rightarrow A^{\op}$ are the homomorphisms induced by the first, respectively last, square.
 In particular, we have a commutative diagram of finite-dimensional algebras
\[\begin{tikzcd}[ampersand replacement=\&]
	{B'} \& A \\
	B \& A
	\arrow["{\rho_{a^{-1}}}", from=1-2, to=2-2]
	\arrow["{g'}", from=1-1, to=2-1]
	\arrow[from=2-1, to=2-2]
	\arrow[from=1-1, to=1-2]
\end{tikzcd}\]
where the horizontal arrows are the canonical inclusions. Thus
 $aB'a^{-1}\subseteq B$ and $g':B'\rightarrow B$ is given by $g'(b')=ab'a^{-1}$.
 In particular, by definition of $g'$ we have 
 \begin{align*}
     G(r_{b'})=r_{g'(b')}=r_{ab'a^{-1}}
 \end{align*}
 for all $b'\in B'$.
 We show that $aB'a^{-1}$ is a strong exact Borel subalgebra of $B$. 
First, let 
\[\begin{tikzcd}[ampersand replacement=\&]
	{(0)} \& {M'} \& M \& {M''} \& {(0)}
	\arrow[from=1-1, to=1-2]
	\arrow["{f'}", from=1-2, to=1-3]
	\arrow["{f''}", from=1-3, to=1-4]
	\arrow[from=1-4, to=1-5]
\end{tikzcd}\]
be an exact sequence in $\modu B'$. Then we obtain a commutative diagram
\[\begin{tikzcd}[ampersand replacement=\&, column sep =large]
	{(0)} \& {A\otimes_{B'} M'} \& {A\otimes_{B'} M} \& {A\otimes_{B'} M''} \& {(0)} \\
	{(0)} \& {A\otimes_B G(M')} \& {A\otimes_B G(M)} \& {A\otimes_B G(M'')} \& {(0)}
	\arrow[from=1-1, to=1-2]
	\arrow["{\id_A\otimes f'}", from=1-2, to=1-3]
	\arrow["{\id_A\otimes f''}", from=1-3, to=1-4]
	\arrow[from=1-4, to=1-5]
	\arrow[from=1-2, to=2-2]
	\arrow["{\id_A\otimes G(f'')}", from=2-3, to=2-4]
	\arrow["{\id_A\otimes G(f')}", from=2-2, to=2-3]
	\arrow[from=2-4, to=2-5]
	\arrow[from=2-1, to=2-2]
	\arrow[from=1-3, to=2-3]
	\arrow[from=1-4, to=2-4]
\end{tikzcd}\]
where the upper row is exact, since $B'$ is an exact Borel subalgebra of $A$, and the vertical arrows are the isomorphisms given by $\alpha$. Hence the lower row is exact. As before, since $B$ is normal resp. by \cite[Lemma 3.7]{BKK}, $A$ decomposes as a right $B$-module into a direct sum $A_B \cong B_B \oplus V$. Thus, as a sequence of $\field$-vector spaces, the lower row decomposes into a direct sum of two exact sequences 
\begin{equation}
\label{exactseq1}\begin{tikzcd}[ampersand replacement=\&]
	{(0)} \& {B\otimes_B G(M')} \&\& {B\otimes_B G(M')} \&\& {B\otimes_B G(M')} \& {(0)}
	\arrow["{\id_B\otimes G(f'')}", from=1-4, to=1-6]
	\arrow["{\id_B\otimes G(f')}", from=1-2, to=1-4]
	\arrow[from=1-6, to=1-7]
	\arrow[from=1-1, to=1-2]
\end{tikzcd}\end{equation}
and 
\begin{equation}
\begin{tikzcd}[ampersand replacement=\&]
	{(0)} \& {V\otimes_B G(M')} \&\& {V\otimes_B G(M')} \&\& {V\otimes_B G(M')} \& {(0).}
	\arrow["{\id_V\otimes G(f'')}", from=1-4, to=1-6]
	\arrow["{\id_V\otimes G(f')}", from=1-2, to=1-4]
	\arrow[from=1-6, to=1-7]
	\arrow[from=1-1, to=1-2]
\end{tikzcd}
\end{equation}
The upper sequence is, as  a sequence of $\field$-vector spaces, isomorphic to the sequence
\begin{equation}
\label{exactseq2}
\begin{tikzcd}[ampersand replacement=\&]
	{(0)} \& { G(M')} \&\& {G(M')} \&\& {G(M')} \& {(0).}
	\arrow["{ G(f'')}", from=1-4, to=1-6]
	\arrow["{G(f')}", from=1-2, to=1-4]
	\arrow[from=1-6, to=1-7]
	\arrow[from=1-1, to=1-2]
\end{tikzcd}
\end{equation}
Since the sequence \eqref{exactseq1} is exact, so is the sequence \eqref{exactseq2}.
Thus $G$ is an exact functor. In particular, the Eilenberg-Watts theorem yields a $B$-$B'$-bimodule $X$ such that $G\cong X\otimes_{B'}-$. 
Note that
$X\cong X\otimes_{B'}B'\cong G(B')=B$ as a $B$-$B'$-bimodule,  where the right module structure on $G(B')=B$ is given by
\begin{align*}
     b\cdot b':=G(r_{b'})(b)=bab'a^{-1}.
\end{align*}
Hence $X\cong Ba$ as a  $B$-$B'$-bimodule.
Consider the canonical equivalence
 \begin{align*}
	H:\modu B' \rightarrow \modu aB'a^{-1}, M\mapsto aM, f\mapsto (am\mapsto af(m)).
 \end{align*}
 Then the diagram 
\[\begin{tikzcd}[ampersand replacement=\&]
	{\modu B'} \&\& {\modu aB'a^{-1}} \\
	\& {\modu B}
	\arrow["{ Ba\otimes_{B'}-}"'{pos=0}, from=1-1, to=2-2]
	\arrow["{B\otimes_{aB'a^{-1}}-}"{pos=0.2}, from=1-3, to=2-2]
	\arrow["H", from=1-1, to=1-3]
\end{tikzcd}\]
commutes up to natural isomorphism given by
\begin{align*}
    \gamma:(B\otimes_{aB'a^{-1}}-)\circ H&\rightarrow Ba\otimes_{B'}-\\
    \gamma_M:B\otimes_{aB'a^{-1}}aM&\rightarrow Ba\otimes_{B'}M\\
            b\otimes am&\mapsto ba\otimes m.
\end{align*}
Hence 
\[\begin{tikzcd}[ampersand replacement=\&]
	{\modu B'} \&\& {\modu aB'a^{-1}} \\
	\& {\modu B}
	\arrow["G"'{pos=0}, from=1-1, to=2-2]
	\arrow["{B\otimes_{aB'a^{-1}}}"{pos=0.2}, from=1-3, to=2-2]
	\arrow["H", from=1-1, to=1-3]
\end{tikzcd}\]
commutes up to some natural equivalence $\gamma':(B\otimes_{aB'a^{-1}}-)\circ H\rightarrow G$.
Since $H$ is an equivalence and $G$ is exact, this implies that  $(B\otimes_{aB'a^{-1}}-)$ is exact.
Moreover, note that
\begin{align*}
  A\otimes_B G(L_i^{B'})\cong A\otimes_{B'}L_i^{B'}\cong \Delta_i^A\cong A\otimes_B L_i^B
\end{align*}
In particular, $\Top(A\otimes_B G(L_i^B))\cong \Top(\Delta_i^A)\cong L_i^A$ for every $1\leq i\leq n$, so that by Lemma \ref{lemma_strong_top} $\Top(G(L_i^{B'}))\cong L_i^B$. Thus, there are projections
\begin{align*}
  \pi_i: G(L_i^{B'})\rightarrow L_i^B.
\end{align*}
Again, these give rise to short exact sequences
\[\begin{tikzcd}[ampersand replacement=\&]
	{(0)} \& {A\otimes_B K_i} \& {A\otimes_B G(L_i^{B'})} \& {A\otimes_B L_i^B} \& {(0)}
	\arrow[from=1-1, to=1-2]
	\arrow["{\id_A\otimes\iota_i}", from=1-2, to=1-3]
	\arrow["{\id_A\otimes_B \pi_i}", from=1-3, to=1-4]
	\arrow[from=1-4, to=1-5]
\end{tikzcd}\]
Since $A\otimes_B G(L_i^{B'})\cong {A\otimes_B L_i^B}$, $\id_A\otimes_B \pi_i$ is an isomorphism for dimension reasons, we again conclude $A\otimes_B K_i=(0)$. As before, normality of $B$ in $A$ implies that $K_i=(0)$ and thus $\pi_i$ becomes an isomorphism.
For $1\leq i\leq n$ let $L_i^{aB'a^{-1}}:=H(L_i^{B'})$. Then $\{L_1^{aB'a^{-1}}, \dots ,L_n^{aB'a^{-1}}\}$ is a complete set of representatives of the isomorphism classes of the simple $aB'a^{-1}$-modules, and by the above
\begin{align*}
    B\otimes_{a'Ba^{-1}}L_i^{aB'a^{-1}}= B\otimes_{aB'a^{-1}}H(L_i^{B'})\cong G(L_i^{B'})\cong L_i^B.
\end{align*}
Thus $aB'a^{-1}$ is an exact Borel subalgebra of $B$. Since $B$ is directed, Lemma \ref{lemma_strong_of_directed} now implies  $aB'a^{-1}=B$.
\end{proof}
\section{G-equivariance}\label{section_skew}
In this section, we consider a finite group $G$ acting on a quasi-hereditary algebra $(A, \leq_A)$ via algebra automorphisms, such that the characteristic of $\field$ does not divide the order of $G$. We want to investigate the existence of $G$-invariant regular exact Borel subalgebras. To this end, we consider the induced $G$-actions on the categories involved in the diagram \eqref{diagram_1} and examine the functors with respect to their $G$-equivariance. \\
In particular, we consider $G$-actions on A-infinity algebras and categories, which are defined in analogy to the $G$-actions on dg-algebras and categories considered in \cite{Lemeur}.
\subsection{Preliminaries of Group Actions}
For an introduction to group actions on categories see for example \cite{ReitenRiedtmann, Poon}.\\
\begin{definition}
	Let $\mathcal{C}$ be a category. We say that $G$ acts on $\mathcal{C}$ if there is a group homomorphism from $G$ into the autoequivalences of $\mathcal{C}$.
	In this setting we write $gX$ and $g\varphi$ for the image of an object $X$ and a morphism $\varphi$ in $\mathcal{C}$ under the autoequivalence given by a group element $g\in G$.\\  
	We call a functor $F:\mathcal{C}\rightarrow \mathcal{C}'$ between two categories with $G$-actions (strongly) $G$-equivariant, if 
	\begin{align*}
		F(gX)=gF(X)
	\end{align*}
	for all $X\in\mathcal{C}$ and 
	\begin{align*}
		F(g\varphi)=gF(\varphi)
	\end{align*}
	for all $\varphi\in \mathcal{C}(X, Y)$, i.e. if $F\circ g=g\circ F$ for all $g\in G$. We call $F$ weakly $G$-equivariant if there is a collection of natural isomorphisms
	\begin{align*}
		\alpha_g:F\circ g\rightarrow g\circ F
	\end{align*} 
	such that $g((\alpha_h)_{X})\circ (\alpha_g)_{hX}=\alpha_{ghX}$ for all $h,g\in G$, $X\in \mathcal{C}$; and $\alpha_e=\id$ is the identity.\\
	Similarly, let $\mathcal{E}$ be a strictly unital A-infinity category. Then we say that $G$ acts on $\mathcal{E}$ if $G$ acts on $\mathcal{E}$ via strictly unital A-infinity autoequivalences, and we call an A-infinity functor  $F=(F_n)_n:\mathcal{E}\rightarrow \mathcal{E}'$ between two A-infinity categories with $G$-actions (strongly) $G$-equivariant if $F\circ g=g\circ F$ for all $g\in G$.
	\end{definition}
 In this article, we don't need the notion of weakly $G$-equivariant A-infinity functors.\\
 Note that being $G$-equivariant is stable under composition, but not under taking inverses and under natural isomorphism of functors, while begin weakly $G$-equivariant is stable under all of these.
 \begin{lemma}
     If $\mathcal{E}$ is a strictly unital A-infinity category with a $G$-action, then there is an induced $G$-action on the category $H^0(\mathcal{E})$. 
     Moreover, if $$F=(F_n)_n:\mathcal{E}\rightarrow \mathcal{E}'$$ is a strictly unital, $G$-equivariant A-infinity functor between two strictly unital A-infinity categories $\mathcal{E}$ and $\mathcal{E}'$ with $G$-actions, then 
     \begin{align*}
         H^0(F):\mathcal{E}\rightarrow \mathcal{E}'
     \end{align*}
     is a strongly $G$-equivariant functor.
 \end{lemma}
 \begin{proof}
     Suppose $G$ acts on $\mathcal{E}$ via the A-infinity autoequivalences given by $g=(g_n)_n$ for every $g\in G$. Then since $H^0$ is a functor by Definition \ref{definition_homology_ainfty_cat}, we have induced equivalences
     \begin{align*}
         H^0(g):H^0(\mathcal{E})\rightarrow H^0(\mathcal{E}),\\ X\mapsto g_0(X),
         a+\im(m_1)\mapsto g_1(a)+\im(m_1),
     \end{align*}
     and
     \begin{align*}
         H^0(g\circ h)=H^0(g)\circ H^0(h).
     \end{align*}
     Moreover, if $$F=(F_n)_n:\mathcal{E}\rightarrow \mathcal{E}'$$ is a strictly unital, $G$-equivariant A-infinity functor, then we have for every $g\in G$
     \begin{align*}
         F\circ g=g\circ F
     \end{align*}
     and thus
     \begin{equation*}
         H^0(F)\circ H^0(g)=H^0(F\circ g)=H^0(g\circ F)=H^0(g)\circ H^0(F).\qedhere
     \end{equation*}
 \end{proof}
The following is a classical example of a group action on a category, and was extensively studied in  \cite{ReitenRiedtmann}.
\begin{example}\label{example_A*G}
	Let $A$ be a finite-dimensional algebra and suppose $G$ acts on $A$ via algebra automorphisms. Then $G$ acts on $\modu A$ via defining for an $A$-module $M$ the $A$-module $gM$ as the vector space given by formal products $\{gm: m\in M\}$ with addition and scalar multiplication $gm+gm'=g(m+m')$ and $\lambda gm=g\lambda m$ and with multiplication $$a gm=g (g^{-1}(a)m),$$ and for a module homomorphism $\varphi:M\rightarrow N$ setting $$g\varphi(gm)=g\varphi(m).$$
    In this setting, a $G$-action on an $A$-module $M$ is a collection of $A$-module isomorphisms
    \begin{align*}
        \tr_g^M:gM\rightarrow M
    \end{align*}
    for $g\in G$
    such that $\tr_e^M=\id_M$ and for all $g,h\in G$ 
    \begin{align*}
        \tr_g^M\circ g(\tr_h^M)=\tr_{gh}^M.
    \end{align*}
    The category of $A$-modules with a $G$-action is the category whose objects are pairs $(M, (\tr_g^M)_{g\in G})$ consisting of an $A$-module $M$ and a $G$-action $(\tr_g^M)_{g\in G}$ on $M$, and whose morphisms $$f:(M, (\tr_g^M)_{g\in G})\rightarrow (N, (\tr_g^N)_{g\in G})$$ are $A$-linear maps
    $f:M\rightarrow N$ such that
    \begin{align*}
        \tr_g^N\circ g(f)=f\circ \tr_g^M
    \end{align*}
    for all $g\in G$.
    By \cite[Proposition 4.8]{Poon} there is an equivalence of categories between the category of $A$-modules with a $G$-action and the category of $A*G$-modules, where $A*G$ is the skew group algebra, which is given by
$A*G:=A\otimes_{\field}\field G$ as $\field$-vector spaces together with the multiplication
\begin{align*}
    (a\otimes g)\cdot (a'\otimes g'):=ag(a')\otimes gg'.
\end{align*}
\end{example}
\begin{remark}
    If $\mathcal{C}$ is a category with a $G$-action, then we can generalize the definition of an $A$-module with a $G$-action to an object in $\mathcal{C}$ with a $G$-action, see for example \cite[Definition 4.7]{Poon}.
\end{remark}
\begin{lemma}\label{lemma_Gstructuretransport}\cite[Lemma 4]{MartinezVilla}
    Let $\mathcal{C}$ be a category with a $G$-action and suppose $M\in \mathcal{C}$ is an object with a $G$-action $(\tr_g^M)_{g\in G}$. Then $G$ acts on $\mathcal{C}(M, M)$ via the algebra automorphisms
    \begin{align*}
        f\mapsto \tr_g^M\circ gf\circ (\tr_g^M)^{-1}.
    \end{align*}
    Moreover, if $F:\mathcal{C}\rightarrow \mathcal{C}'$ is a weakly $G$-equivariant functor with corresponding natural isomorphisms
    \begin{align*}
		\alpha_g:F\circ g\rightarrow g\circ F,
	\end{align*} 
    then there is an induced $G$-action on $F(M)$ given by
    \begin{align*}
        \tr_g^{F(M)}=F(\tr_g^M)\circ \alpha_{g, M}^{-1}: gF(M)\rightarrow F(M),
    \end{align*}
    and 
    \begin{align*}
        F_{M, M}:\mathcal{C}(M, M)\rightarrow \mathcal{C}'(F(M), F(M))
    \end{align*}
is $G$-equivariant with respect to the induced actions on $\mathcal{C}(M, M)$ and $\mathcal{C}'(F(M), F(M))$.
\end{lemma}
\begin{lemma}
    If $A$ is a finite-dimensional algebra with a $G$-action, and we consider $\modu A$ with the induced $G$-action, then $A\in \modu A$ is an object with a $G$-action given by
    \begin{align*}
        \tr_g^A: gA\rightarrow A, ga\mapsto g(a).
    \end{align*}
    Moreover, the induced $G$-action on $\End_A(A)$ is given by
    \begin{align*}
        (g\cdot r_a)(x)=r_{g(a)}(x),
    \end{align*}
    for $a,x\in A$, $g\in G$, where $r_a$ denotes right multiplication with $a$.
    In particular, the isomorphism $A\rightarrow \End_A(A)^{\op}, a\mapsto r_a$ is $G$-equivariant. 
\end{lemma}
\begin{proof}
    The first part can be found in \cite[Theorem 10]{MartinezVilla}. Moreover, the $G$-action on $\End_A(A)$ can be explicitly calculated as follows:
     \begin{align*}
        (g\cdot r_a)(x)=\tr_g^A\circ gr_a\circ (\tr_g^A)^{-1}(x)=\tr_g^A\circ (gr_a)(g \cdot (g^{-1}(x)))=\tr_g^A(g\cdot (g^{-1}(x)a))=g(g^{-1}(x)a)=xg(a)=r_{g(a)}(x).
    \end{align*}
    Now the $G$-equivariance of the canonical isomorphism  $A\rightarrow \End_A(A)^{\op}, a\mapsto r_a$ is obvious.
\end{proof}
For A-infinity algebras, as usual, we want everything to be strictly unital over $L=\field^n$. Let us thus fix a $G$-action on $L$.
\begin{definition}
	Let $\mathcal{E}$ be an A-infinity algebra strictly unital over $L$. Then we say that $G$ acts on $\mathcal{E}$ if $G$ acts on $\mathcal{E}$ via A-infinity automorphisms which restrict to the fixed strict action on $L$. We say that a strictly unital A-infinity homomorphism $\varphi=(\varphi_n)_n:\mathcal{E}\rightarrow \mathcal{E}'$ is $G$-equivariant if $\varphi\circ g=g\circ \varphi$ for all $g\in G$.\\
\end{definition}
\begin{remark}
	Consider the category of $L$-$L$-bimodules with a $G$-action, where the $G$-action on $L\otimes L^{\op}$ is given by $g(x\otimes y)=g(x)\otimes g^{-1}(y)$. Then since $(L\otimes L^{\op})*G$ is semisimple by \cite[Theorem 1.3]{ReitenRiedtmann}, this is a semisimple category, and endowing it with the tensor product over $L$ turns it into a monoidal category. Then, similarly to Example \ref{example_A*G}, A-infinity algebras strictly unital over $L$ with a $G$-action are in one-to-one correspondence with strictly unital A-infinity algebras in the category of $L$-$L$-bimodules with a $G$-action. Moreover, an A-infinity homomorphism is $G$-equivariant if and only if it corresponds to an A-infinity homomorphism in  the category of $L$-$L$-bimodules with a $G$-action.
\end{remark}
The following lemma shows that a $G$-action on an A-infinity algebra induces a $G$-action on the twisted module category.
\begin{lemma}\label{lemma_Ginvariant_twmod}
	Let $\mathcal{E}$ be an A-infinity algebra strictly unital over $L$ with a $G$-action. Then there is an induced $G$-action on $\add^{\mathbb{Z}}(\mathcal{E})$ via A-infinity autoequivalences 
\begin{align*}
	g^{\mathbb{Z}}:\add^{\mathbb{Z}}(\mathcal{E})&\rightarrow \add^{\mathbb{Z}}(\mathcal{E}),\\
	X&\mapsto gX\\
	g^{\mathbb{Z}}_n((a_1\otimes \varphi_1)\otimes \dots \otimes (a_n\otimes \varphi_n))&=\sum_{i=1}^n (-1)^{(n-1)\sum_{i=1}^n|\varphi_i|+\sum_{i< j}|a_i||\varphi_j|}g_n(a_1\otimes \dots \otimes a_n)\otimes (g\varphi_1\circ \dots\circ g\varphi_n),
\end{align*}
for every $g\in G$,
where $X\in \modu^{\mathbb{Z}}(L)$, $a_1,\dots, a_n\in \mathcal{E}$ and $\varphi_i\in \Hom_{\field}(X_{i}, X_{i-1})$ for some $X_0, \dots, X_n\in \modu^{\mathbb{Z}}(L)$;
as well as on $\tw(\mathcal{E})$ via A-infinity autoequivalences
	\begin{align*}
		g^{\tw}:\tw(\mathcal{E})&\rightarrow \tw(\mathcal{E}),\\
		(X, w_X)&\mapsto \left(gX, \sum_{n=1}^\infty g^{\mathbb{Z}}_n(w_X^{\otimes n})\right),\\
		g^{\tw}_n(x_1\otimes \dots \otimes x_n)&=\sum_{k=n}^\infty\sum_{j_0+\dots+j_n=k-n}(-1)^{\sum_{l=1}^n lj_l}g^{\mathbb{Z}}_k(w_X^{\otimes j_0}\otimes x_1\otimes w_X^{\otimes j_1}\dots \otimes w_X^{\otimes j_{n-1}}\otimes x_n\otimes  w_X^{\otimes j_{n}}),
	\end{align*}
    where $X\in \modu^{\mathbb{Z}}(L)$, $w_X\in \add^{\mathbb{Z}}(\mathcal{E})(X,X)_1$ is triangular, and $x_i\in \tw(\mathcal{E})(X_i, X_{i-1})=\add^{\mathbb{Z}}(\mathcal{E})(X_{i}, X_{i-1})$ for some $X_0, \dots, X_n\in \modu^{\mathbb{Z}}(L)$;
	and on $\twmod(\mathcal{E})$ via the restriction $g^{\twmod}$ of the functors above.\\
	Moreover, if $\mathcal{E}'$ is another A-infinity algebra strictly unital over $L$ with a $G$-action, and $f=(f_n)_n:\mathcal{E}\rightarrow \mathcal{E}'$ is a strictly unital A-infinity homomorphism commuting with the $G$-action, then the induced functors $\add^{\mathbb{Z}}(f)$, $\tw(f)$ and $\twmod(f)$ are $G$-equivariant.
\end{lemma}
\begin{proof}
	By definition, every $g\in G$ gives rise to an A-infinity homomorphism 
	\begin{align*}
		g: \mathcal{E}\rightarrow \mathcal{E}.
	\end{align*}
	This becomes strictly unital over $L$, if on the left side we consider the usual embedding $\iota:L\rightarrow \mathcal{E}$ and on the right side we precompose it with the action of $g$ on $L$, i.e. we consider $\iota\circ g:L\rightarrow \mathcal{E}$. 
 Denote for every $g\in G$ by  $\add^{\mathbb{Z}}_g(\mathcal{E})$, $\tw_g(\mathcal{E})$ and $\twmod_g(\mathcal{E})$ the categories constructed with respect to the strictly unital structure given by the embedding $g\circ \iota$,
 and denote by $\add^{\mathbb{Z}}(\mathcal{E})$, $\tw(\mathcal{E})$ and $\twmod(\mathcal{E})$ categories constructed with respect to the strictly unital structure given by the embedding $\iota$.
 Let $X, Y\in \modu L$ and $\varphi\in \Hom_{\field}(X,Y)$. Then we write
 \begin{align*}
     g\varphi:gX\rightarrow gY, gx\mapsto g(\varphi(x)).
 \end{align*}
 Then we have for every $l\in L$, $x\in X$
 \begin{align*}
     (g(l\cdot\varphi))(gx)=g(l\varphi(x))=g(l)g(\varphi(x)),
 \end{align*}
 so that $g(l\cdot\varphi)=g(l)g\varphi$, and
  \begin{align*}
     (g(\varphi\cdot l))(gx)= g(\varphi(lx))=(g\varphi)(glx)=(g\varphi)(g(l)\cdot gx)=((g\varphi)\cdot g(l))(gx),
 \end{align*}
 so that $g(\varphi\cdot l)=(g\varphi)\cdot g(l)$.
 We claim that we have strict A-infinity equivalences
	\begin{align*}
		\tr_g^{\mathbb{Z}}:\add^{\mathbb{Z}}_g(\mathcal{E})\rightarrow \add^{\mathbb{Z}}(\mathcal{E}), X\mapsto gX, a\otimes \varphi\mapsto a\otimes g\varphi,
	\end{align*}
as well as
\begin{align*}
	\tr^{\tw}_g:\tw_g(\mathcal{E})&\rightarrow \tw(\mathcal{E}),\\
	 \left(X, w_X\right)&\mapsto \left(gX,\tr_g^{\mathbb{Z}}(w_X) \right),\\
		x &\mapsto \tr_g^{\mathbb{Z}}(x)
\end{align*}
and
	\begin{align*}
		\tr_g^{\twmod}:\twmod_g(\mathcal{E})&\rightarrow \twmod(\mathcal{E})
	\end{align*}
	given by the restriction of the functor $\tr_g^{\tw}$.
 We show this for $\add^{\mathbb{Z}}$.
 First, let us show that $\tr_g^{\mathbb{Z}}$ is well defined. Let
 $a\in \mathcal{A}$, $X, Y\in \modu^{\mathbb{Z}} L$, $\varphi\in \Hom_{\field}(X,Y)$ and $l\in L$.
 Then we have
 \begin{align*}
    a\otimes g(l\cdot \varphi)=a\otimes g(l) g\varphi=g(l)a\otimes g\varphi
 \end{align*}
 and
  \begin{align*}
    a\otimes g( \varphi\cdot l)=a\otimes  g\varphi g(l)=ag(l)\otimes g\varphi.
 \end{align*}
 Thus $\tr_g^{\mathbb{Z}}$ is well-defined.\\
 Let $a_1, \dots , a_n\in \mathcal{E}$, $X_0, \dots, X_n\in \modu^{\mathbb{Z}} L$ and for $1\leq i\leq n$ let $\varphi_i\in \Hom_{\field}(X_{n-i+1}, X_{i-n})$.
 Then 
 \begin{align*}
     m_n^{\add^{\mathbb{Z}}(\mathcal{E})}(\tr_g^{\mathbb{Z}}(a_1\otimes \varphi_1)\otimes \dots \otimes \tr_g^{\mathbb{Z}}(a_n\otimes\varphi_n))
     &=m_n^{\add^{\mathbb{Z}}(\mathcal{E})}((a_1\otimes g\varphi_1)\otimes \dots \otimes (a_n\otimes g\varphi_n))\\
     &=(-1)^{n\sum_{i=1}^n |\varphi_i|+\sum_{i<j}|a_i||\varphi_j|}m_n^{\mathcal{E}}(a_1\otimes\dots \otimes a_n)\otimes (g\varphi_1\circ \dots \circ g\varphi_n)\\
     &=(-1)^{n\sum_{i=1}^n |\varphi_i|+\sum_{i<j}|a_i||\varphi_j|}m_n^{\mathcal{E}}(a_1\otimes\dots \otimes a_n)\otimes g(\varphi_1\circ\dots\circ \varphi_n)\\
     &=\tr_g^{\mathbb{Z}}((-1)^{n\sum_{i=1}^n |\varphi_i|+\sum_{i<j}|a_i||\varphi_j|}m_n^{\mathcal{E}}(a_1\otimes\dots \otimes a_n)\otimes (\varphi_1\circ\dots\circ \varphi_n))\\
     &=\tr_g^{\mathbb{Z}}(m_n^{\add_g^{\mathbb{Z}}(\mathcal{E})}((a_1\otimes \varphi_1)\otimes \dots \otimes (a_n\otimes\varphi_n))).
 \end{align*}
The calculation for $\tw_g$ is similar, and the case $\twmod_g$ follows from $\tw_g$  by restriction.\\
	Thus we obtain A-infinity autoequivalences
	\begin{align*}
		\tr_g^{\mathbb{Z}}\circ \add^{\mathbb{Z}}(g)&:\add^{\mathbb{Z}}(\mathcal{E})\rightarrow\add^{\mathbb{Z}}(\mathcal{E})\\
		\tr_g^{\tw}\circ \tw(g)&:\tw(\mathcal{E})\rightarrow\tw(\mathcal{E})\\
		\tr_g^{\twmod}\circ \twmod(g)&:\twmod(\mathcal{E})\rightarrow\twmod(\mathcal{E}).
	\end{align*}
	The explicit formulas follow directly from the definition and from the explicit formulas given in Lemma \ref{lemma_twfunctor}. The fact that this yields a $G$-action and that $G$-equivariant A-infinity homomorphisms give rise to $G$-equivariant A-infinity functors, on the other hand, is a consequence of the functoriality of $\add^{\mathbb{Z}}, \tw$ and $\twmod$.
\end{proof}
The case where $\mathcal{E}$ is a dg-algebra unital over $L$ is of particular importance in Proposition \ref{proposition_ginvariant_End_F(M)} to show that the functor $C_M$ from Theorem \ref{theorem_equivalence_End_F(M)} is weakly $G$-equivariant. We thus explicitly calculate the induced $G$-action on $\twmod(\mathcal{E})$ for a dg-algebra $\mathcal{E}$.
\begin{example}\label{example_gaction_twdg}
	Let $\mathcal{E}$ be a dg-algebra unital over $L$ and suppose $G$ acts via dg-algebra automorphisms restricting to the usual action on $L$. By the previous lemma, there is an induced action on $\twmod(\mathcal{E})$. In this case, $\twmod(g)$ is the functor
	\begin{align*}
		\twmod(g):\twmod(\mathcal{E})&\rightarrow \twmod_g(\mathcal{E})\\
		 \left(X, \sum_{i=1}^m a_i\otimes \varphi_i\right)&\mapsto \left(X, \sum_{i=1}^m g(a_i)\otimes \varphi_i\right),
	\end{align*}
 where $X\in \modu L$, $a_1, \dots, a_n\in \mathcal{E}_1, \varphi_1,\dots ,\varphi_n\in \End_{\field}(X)$ such that $\sum_{i=1}^m a_i\otimes \varphi_i\in \add(\mathcal{E})(X,X)=\mathcal{E}_1\otimes_{L\otimes L^{\op}}\End_{\field}(X)$ is triangular.
	 Therefore, the $G$-action is given by the composition 
	\begin{align*}
		\tr_g\circ \twmod(g):\twmod(\mathcal{E})&\rightarrow \twmod(\mathcal{E})\\
		 \left(X, \sum_{i=1}^m a_i\otimes \varphi_i\right)&\mapsto \left(gX, \sum_{i=1}^m g(a_i)\otimes g\varphi_i\right),\\
		a\otimes \varphi &\mapsto g(a)\otimes g\varphi.
	\end{align*}
\end{example}
\begin{remark}\label{remark_Gstructure}
Suppose $G$ is a group acting on an algebra $A$ via algebra automorphisms and let $M=\bigoplus_{i=1}^n M_i$ be an $A*G$-module. Let $P^\cdot(M)$ be a projective resolution of $M$ as an $A*G$-module. Then, by Lemma \ref{lemma_Gstructuretransport}, $G$ acts on $\End_A^*(P^\cdot(M))$ via $g\cdot \varphi= g\varphi(g^{-1}\cdot)$.\\
If, with respect to this action and the fixed action on $L$, the algebra homomorphism $L\rightarrow  \End_A^*(P^{\cdot}(M))$ from (2.) in Remark \ref{remark_long} becomes $G$-equivariant, then $G$ acts on $\End_A^*(P^\cdot(M))$ via unital dg-algebra automorphisms, that is, $G$ acts on the unital dg-algebra $\End_A^*(P^\cdot(M))$.\\
Kadeishvili's theorem over the semisimple monoidal category of $L$-$L$-bimodules with a $G$-action now yields that there is an A-infinity structure on $\Ext^*_A(M, M)$ strictly unital over $L$ with a $G$-action restricting to the fixed action on $L$, such that we have a strictly unital $G$-equivariant A-infinity quasi-isomorphism $$f=(f_n)_n:\Ext_A^*(M, M)\rightarrow \End_A^*(P^\cdot(M)),$$ which in every component is a homomorphism of $L$-$L$-bimodules with a $G$-action, and which is strictly unital over $L$.\\
The proposition above tells us that in this setting, we have $G$-actions on $\twmod(\End_A^*(P^\cdot(M)))$ as well as on $\twmod(\Ext_A^*(M, M))$ and a $G$-equivariant A-infinity quasi-equivalence
\begin{align*}
	\twmod(f):\twmod(\Ext_A^*(M, M))\rightarrow \twmod(\End_A^*(P^\cdot(M))).
\end{align*}
Taking homology in degree zero gives us $G$-actions on the categories $H^0(\twmod(\End_A^*(P^\cdot(M))))$ and\\ $H^0(\twmod(\Ext_A^*(M, M)))$ and a $G$-equivariant equivalence
\begin{align*}
	H^0(\twmod(f)):H^0(\twmod(\Ext_A^*(M, M)))\rightarrow H^0(\twmod(\End_A^*(P^\cdot(M)))).
\end{align*}
\end{remark}
The following proposition shows that in this setting the equivalence $H^0(\twmod(\End_A^*(P^\cdot(M))))\rightarrow \filt(M)$ from Theorem \ref{theorem_equivalence_End_F(M)} is also $G$-equivariant.
\begin{proposition}\label{proposition_ginvariant_End_F(M)}
	Suppose $G$ is a group acting on $A$ via algebra automorphisms and $M=\bigoplus_{i=1}^n M_i$ is an $A*G$-module. Let $P^\cdot(M)$ be a projective resolution of $M$ as an $A*G$-module, and suppose the unit map $L\rightarrow \End_A(P^\cdot(M))$ from (2.) in Remark \ref{remark_long} is $G$-equivariant. Then the equivalence
	\begin{align*}
		C=C_M:\twmod_L(\End_A(P^\cdot(M)))\rightarrow \filt(M)
	\end{align*}
	from Theorem \ref{theorem_equivalence_End_F(M)} is weakly $G$-equivariant.
\end{proposition}
\begin{proof}
	Using Example \ref{example_gaction_twdg}, on objects we have 
	\begin{align*}
		C\left(g\left(X, \sum_{i=1}^m a_i\otimes \varphi_i\right)\right)=C\left(gX, \sum_{i=1}^n g(a_i)\otimes g\varphi_i\right)=H^0\left(P^\cdot(M)\otimes_L gX, d_{P^\cdot(M)}\otimes \id_{gX}+\sum_{i=1}^m g(a_i)\otimes g\varphi_i\right)\\
	\end{align*}
	and 
	\begin{align*}
		g\left(C\left(X, \sum_{i=1}^m a_i\otimes \varphi_i\right)\right)&=g\left(P^\cdot(M)\otimes_L X, d_{P^\cdot(M)}\otimes \id_{X}+\sum_{i=1}^m a_i\otimes \varphi_i\right)\\
  &=\left(gP^\cdot(M)\otimes_L X, g(d_{P^\cdot(M)}\otimes \id_{X})+g\sum_{i=1}^m a_i\otimes \varphi_i\right).\\
	\end{align*}
	Note we have natural isomorphisms $\alpha_g:C\circ g\rightarrow g\circ C$ given by 
	\begin{align*}
		(\alpha_g)_{\left(X,  \sum_{i=1}^m a_i\otimes \varphi_i\right)}: C\circ g\left(X,  \sum_{i=1}^m a_i\otimes \varphi_i\right)&=\left(P^\cdot(M)\otimes_L gX, d_{P^\cdot(M)}\otimes \id_{gX}+\sum_{i=1}^n g(a_i)\otimes g\varphi_i \right)\\
    \rightarrow &g\circ C\left(X,  \sum_{i=1}^m a_i\otimes \varphi_i\right)=\left(gP^\cdot(M)\otimes_L X, g(d_{P^\cdot(M)}\otimes \id_{X})+g\sum_{i=1}^m a_i\otimes \varphi_i\right),\\
		 &a\otimes gx\mapsto g(g^{-1}(a)\otimes x).
	\end{align*}
	Clearly, $\alpha_e$ is the identity. Moreover, for any $(X, w_X)=\left(X,  \sum_{i=1}^m a_i\otimes \varphi_i\right)\in H^0(\twmod(\Ext^*_B(L^B, L^B))$, and for all $g,h\in G$, $a\in P^\cdot(M)$ and $x\in X$ we have
 \begin{align*}
     g((\alpha_h)_{(X, w_X)})\circ (\alpha_g)_{h(X, w_X))}(a\otimes ghx)=g((\alpha_h)_ {(X, w_X)})( g(g^{-1}(a)\otimes hx))\\
     =g((\alpha_h)_{(X, w_X)}(g^{-1}(a)\otimes hx))=g(h(h^{-1}(g^{-1}(a))\otimes x))=(\alpha_{gh})_{(X, w_X)}(a\otimes x).
 \end{align*}
 Hence $C$ is weakly $G$-equivariant.
\end{proof}
To construct a $G$-action on our regular exact Borel subalgebra $B$, we also need to take a closer look at Keller's reconstruction theorem.
\begin{definition}
    Let $\mathcal{A}$ be a a degree-wise finite-dimensional A-infinity algebra strictly unital over $L$.  Then the dual bar construction (or A-infinity Koszul dual) of $\mathcal{A}$ is the dg-algebra, whose underlying algebra is the tensor algebra
    \begin{align*}
        \dualbar s\mathcal{A}:=\textup{T}_L(\dual s\mathcal{A})
    \end{align*}
    together with the differential
    \begin{align*}
        d: \dualbar s\mathcal{A}\rightarrow  \dualbar s\mathcal{A},\\
        d(\varphi)=\sum_{n=1}^\infty \varphi\circ s\circ m_n\circ (s^{-1})^{\otimes n}\textup{ for }\varphi\in \dual s \mathcal{A}.
    \end{align*}
\end{definition}
\begin{theorem}(Keller's Reconstruction Theorem \cite[7.7]{Keller}, see also \cite[Proposition 6.3]{KKO}\label{thm_Keller}).\\
	Let $B$ be a finite-dimensional algebra and let $\{L_1^B, \dots, L_n^B\}=\Sim(B)$ be a set of representatives of the simple $B$-modules. Let $L^B:=\bigoplus_{L\in \Sim(B)}L$ and let $\mathcal{C}=\dualbar s\Ext^*_B(L^B, L^B)$ be the dg-algebra given by the dual bar construction of the A-infinity algebra $\Ext^*_B(L^B, L^B)$ and $I\subseteq \mathcal{C}$ be the dg-ideal generated by the negative degree part. Let
	\begin{align*}
		B':=(\mathcal{C}/I)_0.
	\end{align*}
Then $B'$ is a basic finite-dimensional algebra and we have an equivalence
\begin{align*}
	F':H^0(\twmod(\Ext^*_B(L^B, L^B)))\rightarrow \modu B'
\end{align*}
where $F'(X, w_X=\sum_{i=1}^m a_i\otimes \varphi_i)$ is the module $X$ with multiplication given by
\begin{align*}
	[\eta]\cdot x:=\sum_{i=1}^m \eta(sa_i)\varphi_i(x)
\end{align*}
for $x\in X$, $\eta\in \dual s\Ext^*_B(L^B, L^B)$, and an element $[\varphi]\in H^0(\twmod(\Ext^*_B(L^B, L^B))((X, w_X), (Y, w_Y)))$ gives rise to a map from $X$ to $Y$ by
\begin{align*}
	\varphi\in &(\twmod(\Ext^*_B(L^B, L^B))((X, w_X), (Y, w_Y)))_0=\Ext^*_B(L^B, L^B)_0\otimes_{L\otimes L^{\op}}\Hom_{\field}(X, Y)\\
	&=\Hom_B(L^B, L^B)\otimes_{L\otimes L^{\op}}\Hom_{\field}(X, Y)\cong L\otimes_{L\otimes L^{\op}}\Hom_{\field}(X, Y)\cong \Hom_L(X, Y).
\end{align*}
In particular, $B'$ is Morita-equivalent to $B$.
\end{theorem}
\begin{proposition}\label{proposition_Keller_equivariant}
	In the above setting, suppose $G$ acts on $\Ext^*_B(L^B, L^B)$ via strictly unital A-infinity automorphisms. Then $G$ acts on $B'$ via algebra automorphisms and the Morita equivalence
	\begin{align*}
		F':H^0(\twmod(\Ext^*_B(L^B, L^B)))\rightarrow \modu B'
	\end{align*}
	is strongly $G$-equivariant.
\end{proposition}
\begin{proof}
	Since $G$ acts on $\Ext^*_B(L^B, L^B)$ via A-infinity automorphisms $g=(g_n)_n$, the functoriality of the dual bar construction \cite[pp. 29-30]{LefHas} implies that $G$ acts on $\mathcal{C}=\dualbar s\Ext^*_B(L^B, L^B)$ via dg-algebra automorphisms given by 
 \begin{align*}
		g(\eta)= \sum_{n=1}^\infty g(\eta)_n,
	\end{align*}
 where $ g(\eta)_n\in \dual (s\Ext^*_B(L^B, L^B))^{\otimes n}$ is given by 
 \begin{align*}
     g(\eta)_n=\eta\circ s\circ g^{-1}_n\circ (s^{-1})^{\otimes n}.
 \end{align*}
 for $\eta\in \Dual s\Ext^n_B(L^B, L^B)), a_1, \dots , a_n\in \Ext^*_B(L^B, L^B)$.
	In particular, any $g\in G$ gives a map of degree zero, and thus maps the negative degree part of $\mathcal{C}$ to the negative degree part, so that $g(I)=I$. Hence there is an induced action on 
	\begin{align*}
		B':=(\mathcal{C}/I)_0.
	\end{align*}
	via algebra automorphisms, given by 
	\begin{align*}
		g([\eta])= \sum_{n=1}^\infty [g(\eta)_n].
	\end{align*}
	Let $(X, w_X)\in H^0(\twmod(\Ext_B^*(L^B, L^B)))$. Write $w_X= \sum_{i=1}^m a_i\otimes \varphi_i$ for $a_1, \dots , a_n\in \Ext^1_B(L^B, L^B)$ and $\varphi_1, \dots \varphi_m\in \End_{\field}(X)$.
	Then $g(X, w_X)=(gX,  w_{gX})$ where
	\begin{align*}
		 w_{gX}=\sum_{n=1}^\infty g_n(w_X^{\otimes n}) =\sum_{n=1}^\infty \sum_{i_1,\dots, i_n\in \{1,\dots m\}} g_n(a_{i_1}\otimes\dots \otimes a_{i_n})\otimes (g\varphi_{i_1}\circ\dots\circ g\varphi_{i_n})
	\end{align*}
  and so
	$F'(g(X, w_X))$ is given by the $L$-module $gX$ together with the multiplication
 \begin{align*}
     [\eta]\cdot gx&=\eta(w_{gX})(gx)\\
     &=\sum_{n=1}^\infty \sum_{i_1,\dots, i_n\in \{1,\dots m\}} \eta(s g_n(a_{i_1}\otimes\dots \otimes a_{i_n}))\otimes (g\varphi_{i_1}\circ\dots\circ g\varphi_{i_n})(gx)\\
      &=\sum_{n=1}^\infty \sum_{i_1,\dots, i_n\in \{1,\dots m\}} \eta(s g_n(a_{i_1}\otimes\dots \otimes a_{i_n}))\otimes g(\varphi_{i_1}\circ\dots\circ \varphi_{i_n}(x))\\
      &=g\sum_{n=1}^\infty \sum_{i_1,\dots, i_n\in \{1,\dots m\}} \eta(s g_n(a_{i_1}\otimes\dots \otimes a_{i_n}))\otimes (\varphi_{i_1}\circ\dots\circ \varphi_{i_n}(x))\\
      &=g\sum_{n=1}^\infty \sum_{i_1,\dots, i_n\in \{1,\dots m\}} (-1)^{\sum_{i=1}^n (n-i)|sa_i|} \eta\circ s\circ  g_n\circ (s^{-1})^{\otimes n} (sa_{i_1}\otimes\dots \otimes sa_{i_n}))\otimes (\varphi_{i_1}\circ\dots\circ \varphi_{i_n}(x))\\
      &=g\sum_{n=1}^\infty \sum_{i_1,\dots, i_n\in \{1,\dots m\}} \eta\circ s\circ  g_n\circ (s^{-1})^{\otimes n} (sa_{i_1}\otimes\dots \otimes sa_{i_n}))\otimes (\varphi_{i_1}\circ\dots\circ \varphi_{i_n}(x))\\
    &=g \sum_{n=1}^\infty \sum_{i_1,\dots, i_n\in \{1,\dots m\}}g^{-1}(\eta)_n(a_{i_1}\otimes\dots \otimes a_{i_n})\otimes (\varphi_{i_1}\circ\dots\circ \varphi_{i_n}(x))\\
    &=g \sum_{n=1}^\infty [g^{-1}(\eta)_n](x)\\
    &=g  ((g^{-1}([\eta]))(x)),
 \end{align*}
	where we have used that $|sa_i|=0$ for $1\leq i\leq n$.
	Thus, $F'(g(X, w_X))$ is the $B'$-module $gX$.\\
	Moreover, if $[x\otimes \varphi]\in H^0(\twmod(\Ext_B^*(L^B, L^B))((X, w_X), (Y, w_Y)))$ then $x\in L\cong \Ext_B^0(L^B, L^B)\subseteq \Ext^*_B(L^B, L^B)$, so that $x\otimes \varphi=1\otimes x\varphi$, since the tensor product is the tensor product over $L\otimes L^{\op}$. Thus we can assume $x=1$ without loss of generality. Since the $G$-action on $L\cong \Ext_B^0(L^B, L^B)\subseteq \Ext^*_B(L^B, L^B)$ is just the $G$-action on $L$ by assumption, we moreover have
	 $$g([1\otimes \varphi])=[g(1)\otimes g\varphi]=[1\otimes g\varphi].$$ 
	 By definition of $F'$, $F'([1\otimes \varphi])=\varphi$ and $F'([1\otimes g\varphi])=g\varphi$. Thus
	\begin{equation*}
		F'(g[1\otimes \varphi])=g\varphi=gF'([1\otimes \varphi]).\qedhere
	\end{equation*}
\end{proof}
\subsection{The proof of the main theorem of the section}
Let $(A, \leq_A)$ be a quasi-hereditary algebra with a $G$-action.
Recall from \cite[Definition 3.1]{mypaper} the following definition:
\begin{definition}\label{definition_Ginvariant}
    The partial order $\leq_A$ is called $G$-invariant if for any two simple $A$-modules $L, L'\in \Sim(A)$ and for every $h,g\in G$ we have
\begin{align*}
    L<_A L'\Leftrightarrow gL<_AhL'.
\end{align*}
\end{definition}
Suppose that $\leq_A$ is $G$-invariant.
In the proof of our main theorem, we would like to assume that  $A$ is basic with maximal semisimple subalgebra $L'\cong L$ and a $G$-action that restricts to a $G$-action on $L'$, and view the induced action on $L$ as the fixed $G$-action on $L$ from the previous subsection.\\ 
The following two results serve to show that we may assume that $A$ is basic.
\begin{proposition}\label{proposition_GequivariantMorita}
    Let $A$ be a finite-dimensional algebra and $G$ be a group acting on $A$. Suppose $P$ is an $A*G$-module such that its restriction ${}_{A|}P$ is a projective generator in $\modu A$.
    Then the Morita equivalence given by 
    \begin{align*}
        F=\Hom_A(P, -):\modu A\rightarrow \modu A'
    \end{align*}
     is weakly $G$-equivariant, where the $G$-action on $A':=\End_A(P)^{\op}$ is given as in Lemma \ref{lemma_Gstructuretransport}.
\end{proposition}
\begin{proof}
Consider the natural isomorphisms
    \begin{align*}
	\alpha_g: \Hom_A(P, -)\circ g \rightarrow  g\circ \Hom_A(P, -)
\end{align*}
given by 
\begin{align*}
		(\alpha_g)_M: \Hom_A(P, gM)\rightarrow g\Hom_A(P, M), \varphi\mapsto g[g^{-1}(\varphi)\circ (\tr^{P}_{g^{-1}})^{-1}],
\end{align*}
where $\tr^P_g:gP\rightarrow P$ is just multiplication by $g$, and $g[\dots]$ denotes formal multiplication with $g$ and not application of the functor $g$.\\
Then for the the unit $e\in G$ we have $\alpha_e=\id$ by definition. Moreover, for $h,g\in G$, $M\in \modu A$ and $\varphi\in  \Hom_A(P, hgM)$ we have
\begin{align*}
    g((\alpha_h)_M)\circ (\alpha_g)_{hM}(\varphi)&=g((\alpha_h)_M)( g[g^{-1}(\varphi)\circ (\tr^{P}_{g^{-1}})^{-1}])\\
    =g[(\alpha_h)_M( g^{-1}(\varphi)\circ (\tr^{P}_{g^{-1}})^{-1})]
    &=gh[h^{-1}g^{-1}(\varphi)\circ h^{-1}((\tr^{P}_{g^{-1}})^{-1})\circ (\tr^{P}_{h^{-1}})^{-1}]\\
    =gh[h^{-1}g^{-1}(\varphi)\circ (\tr^{P}_{h^{-1}}\circ h^{-1}(\tr^{P}_{g^{-1}}))^{-1}]
    &=gh[(gh)^{-1}(\varphi)\circ (\tr^{P}_{(gh)^{-1}})^{-1}]
    =\alpha_{ghM}(\varphi)
\end{align*}
by Example \ref{example_A*G}.
\end{proof}
\begin{theorem}\label{theorem_Gactionbasic}
    Let $A$ be a finite-dimensional algebra and $G$ be a group acting on $A$. Then there is a basic algebra $A'$ with a $G$-action such that $A'$ is Morita equivalent to $A$ and there is a Morita equivalence
    \begin{align*}
        F:\modu A\rightarrow \modu A'
    \end{align*}
    which is weakly $G$-equivariant.
\end{theorem}
\begin{proof}
    Let $L\in \Sim(A)$ and denote by $H_L$ the stabilizer of $L$ in $G$. Then by \cite[Proposition 1.16]{mypaper}, $L$ is an $A*H_L$-module, and $\field G\otimes_{\field H_L}L$ is an $A*G$-module.
    Hence 
    \begin{align*}
       L^A:= \bigoplus_{GL\in \Sim(A)/G}\field G\otimes_{\field H_L}L
    \end{align*}
    is an $A*G$-module. Moreover, as an $A$-module, we have
    \begin{align*}
        {}_{A|}L^A\cong  \bigoplus_{GL\in \Sim(A)/G}{}_{A|}\field G\otimes_{\field H_L}L\cong\bigoplus_{GL\in \Sim(A)/G}\bigoplus_{gH_L\in G/H_L}gL\cong \bigoplus_{GL\in \Sim(A)/G}\bigoplus_{L'\in GL}L'\cong \bigoplus_{L\in \Sim(A)}L,
    \end{align*} so that $[L^A:L]=1$ for all $L\in \Sim(A)$.
    Let $P$ be a projective cover of  $L^A$ in $\modu A*G$. Then its restriction to $\modu A$, ${}_{A|}P$, is a projective cover of $L^A$ in $\modu A$, and thus a basic projective generator. Hence $A':=\End_A(P)^{\op}$ is basic. By Proposition \ref{proposition_GequivariantMorita}, we moreover have a $G$-equivariant Morita equivalence
    \begin{equation*}
         F=\Hom_A(P,-):\modu A\rightarrow \modu A'.\qedhere
    \end{equation*}
\end{proof}
Now, if we assume $A$ is basic, then \cite[Proposition 2.1]{ReitenRiedtmann} tells us that  $A$ contains a maximal semisimple subalgebra $L'$ such that the $G$-action on $A$ restricts to a $G$-action on $L'$.\\
Hence in the following we again fix an action on $L$ and assume that $A$ is basic, with maximal semisimple subalgebra $L'\cong L$ and a $G$-action that restricts to the fixed $G$-action on $L'\cong L$.\\
We want to apply Remark \ref{remark_Gstructure}, which describes how to endow the A-infinity algebra $\Ext^*_A(M, M)$ of an $A*G$-module $M$ with a $G$-action, in the case where $M=\Delta^A=\bigoplus_{i=1}^n \Delta_i$ is the direct sum of standard modules. For this, we need to show the following:
\begin{lemma}\label{lemma_Gactiondelta}
    Suppose $(A, \leq_A)$ is a basic quasi-hereditary algebra with maximal semisimple subalgebra $L$ and a $G$-action that restricts to the fixed $G$-action on $L$, such that $\leq_A$ is $G$-invariant. Let
    $\Delta^A=\bigoplus_{i=1}^n \Delta_i$ be the direct sum of all standard modules of $A$. Then $\Delta^A$ has the structure of an $A*G$-module such that the unit map $L\rightarrow \End_A(P^\cdot(\Delta^A))$ becomes $G$-equivariant, where the $G$-action on $\End_A(P^\cdot(\Delta^A))$ is given as in Remark \ref{remark_Gstructure}.
\end{lemma}
\begin{proof}
    Denote by $\varepsilon_i$ the $i$-th unit vector in $L=\field^n$ and by $e_i$ its image in $A$. Then, since $G$ acts via algebra automorphisms and thus maps orthogonal principle indecomposable idempotents to orthogonal principle indecomposable idempotents, $G$ acts on $\{\varepsilon_1, \dots , \varepsilon_n\} $ and thus on $\{e_1, \dots, e_n\}$ via permutations. In other words, there is a $G$-action on the set $\{1, \dots n\}$ such that $g(\varepsilon_i)=\varepsilon_{g(i)}$ and $g(e_i)=e_{g(i)}$.\\
    For $1\leq i\leq n$ let $L_i:=\field e_i$ be the simple $A$-module corresponding to $i$, $P_i:=Ae_i$  its projective cover and $\Delta_i:=\Delta(L_i)=Ae_i/A\left(\sum_{j>i}e_j\right)Ae_i$ be the associated standard module.
    Then since $e_j>_Ae_i$ if and only if $g(e_j)>_Ag(e_i)$, $G$ acts on $\Delta^A:=\bigoplus_{i=1}^n \Delta_i$ via $$g\left(xe_i+A\left(\sum_{e_j>_Ae_i}e_j\right)Ae_i\right)=g(x)g(e_i)+A\left(\sum_{j>g(i)}e_j\right)Ae_{g(i)}\in \Delta_{g(i)}$$
    for $xe_i+A\left(\sum_{e_j>_A e_i}e_j\right)Ae_i\in \Delta_i$.\\
    Let $u_i=\id_{\Delta_i}\in \End_A(\Delta^A)$ be the idempotent corresponding to the identity on $\Delta_i$.
    Then 
    \begin{align*}
        g(u_i)(x)=g(u_i(g^{-1}(x)))=0 \textup{ if }x\notin g(\Delta_i)=\Delta_{g(i)}\\
        g(u_i)(x)=g(u_i(g^{-1}(x)))=g(g^{-1}(x))=x \textup{ if }x\in g(\Delta_i)=\Delta_{g(i)}.
    \end{align*}
    Hence $g(u_i)=u_{g(i)}$. In other words, the embedding, 
    \begin{align*}
        L\rightarrow \End_A(\Delta^A), \varepsilon_i\mapsto \id_{\Delta_i}
    \end{align*}
    from (2.) in Remark \ref{remark_long} is $G$-equivariant.
\end{proof}
Now we are in a situation to prove our main result.
\begin{theorem}\label{thm_skew}
	Let $(A, \leq_A)$ be a quasi-hereditary algebra with a $G$-action, such that $\leq_A$ is $G$-invariant. Then there is a quasi-hereditary algebra $(A', \leq_{A'})$ with a $G$-action and a weakly $G$-equivariant Morita equivalence $\modu A\rightarrow \modu A'$ of quasi-hereditary algebras such that $A'$ has a $G$-invariant basic regular exact Borel subalgebra $B'$.
\end{theorem}
\begin{proof}
    By Theorem \ref{theorem_Gactionbasic}, we can assume without loss of generality that $A$ is basic. Moreover, by \cite[Proposition 2.1]{ReitenRiedtmann} $A$ has a $G$-invariant maximal semisimple subalgebra $L'\cong L=\field^n$. We fix the $G$-action on $L$ to be the $G$-action induced by the restriction of the $G$-action on $A$.\\
	By \cite[Corollary 1.3]{KKO}, there is a quasi-hereditary algebra $(R, \leq_{R})$ and a Morita equivalence $F_R:\modu R\rightarrow \modu A$ of quasi-hereditary algebras such that $R$ has a regular exact Borel subalgebra $B$.\\
	Since $F_R$ is a Morita equivalence of quasi-hereditary algebras $F_R(\Delta^R)\cong \Delta^A$, so that $F_R$ induces a strictly unital A-infinity isomorphism
	\begin{align*}
		h=(h_n)_n:\Ext^*_A(\Delta^A, \Delta^A)\rightarrow \Ext^*_R(\Delta^R, \Delta^R), [\varphi]\mapsto [F_R(\varphi)].
	\end{align*}
By Lemma \ref{lemma_Gactiondelta}, $\Delta^A$ has the structure of an $A*G$-module such that $G$ acts on $\Ext^*_A(\Delta^A, \Delta^A)$ and this action restricts to the usual action on $L$. This induces an action on $\Ext^*_R(\Delta^R, \Delta^R)$ such that $h$ is $G$-equivariant. Since $h$ is strictly unital, the $G$-action on  $\Ext^*_R(\Delta^R, \Delta^R)$ also restricts to the usual action on $L$. By Example \ref{example_Ext}, the induction functor $R\otimes_B -$ induces a strictly unital A-infinity homomorphism 
\begin{align*}
	f=(f_n)_n:\Ext_B^*(L^B, L^B)\rightarrow \Ext_R^*(\Delta^R, \Delta^R), [\varphi]\mapsto [\id_R\otimes \varphi],
\end{align*}
where $f_1$ is an isomorphism in degree greater than zero, and $\Ext_B^0(L^B, L^B)\cong L$. Since the $G$-action on $\Ext^*_R(\Delta^R, \Delta^R)$ restricts to the usual action on $L$, and is thus strictly unital if we view $\Ext^*_R(\Delta^R, \Delta^R)$ as strictly unital over $L$ via $\iota$ and via $\iota\circ g$, Proposition \ref{proposition_ainftyhoms} implies that we can restrict it to a $G$-action on $\Ext_B^*(L^B, L^B)$ such that $f$ is $G$-equivariant.\\
Moreover, by Theorem \ref{theorem_equivalence_End_F(M)} and the functoriality of $\twmod$ we have a commutative diagram 
\[\begin{tikzcd}[ampersand replacement=\&, column sep = huge]
	{H^0(\twmod(\Ext^*_B(L^B, L^B)))} \& {H^0(\twmod(\Ext^*_R(\Delta^R, \Delta^R)))} \& {H^0(\twmod(\Ext^*_A(\Delta^A, \Delta^A)))} \\
	{\modu B} \& {\filt(\Delta^R)} \& {\filt(\Delta^A)}
	\arrow["{H^0(\twmod(f))}", from=1-1, to=1-2]
	\arrow["{H^0(\twmod(h))}", from=1-2, to=1-3]
	\arrow["{R\otimes_B -}", from=2-1, to=2-2]
	\arrow["{F_R}", from=2-2, to=2-3]
	\arrow["{C_{L^B}}"', from=1-1, to=2-1]
	\arrow["{C_{\Delta^R}}"', from=1-2, to=2-2]
	\arrow["{C_{\Delta^A}}"', from=1-3, to=2-3]
\end{tikzcd}\]
where the vertical arrows are equivalences by Theorem \ref{theorem_equivalence_End_F(M)}, and Proposition \ref{proposition_twmod_equiv}, and $C_{\Delta^A}$ is weakly $G$-equivariant by Lemma \ref{lemma_Ginvariant_twmod} and Proposition \ref{proposition_ginvariant_End_F(M)}.\\
Keller's reconstruction theorem \ref{proposition_Keller_equivariant} now yields a basic algebra $B''$ with a $G$-action and a $G$-equivariant equivalence 
\begin{align*}
	S:H^0(\twmod(\Ext^*_B(L^B, L^B)))\rightarrow	\modu B''.
\end{align*}
Choosing quasi-inverses and setting $H:=C_{\Delta^A}\circ H^0(\twmod(h\circ f))\circ S^{-1}$  and $T:=S\circ C_{L^B}^{-1}$ yields a diagram
\[\begin{tikzcd}[ampersand replacement=\&]
	{\modu B''} \& {\modu A} \\
	{\modu B} \& {\modu R}
	\arrow["H", from=1-1, to=1-2]
	\arrow["{F_R}"', from=2-2, to=1-2]
	\arrow["{R\otimes_B -}"', from=2-1, to=2-2]
	\arrow["T", from=2-1, to=1-1]
\end{tikzcd}\]
which commutes up to some natural isomorphism $\alpha: F_R\circ (R\otimes_B -)\rightarrow H\circ T$ and where the upper horizontal arrow is weakly $G$-equivariant and the vertical arrows are equivalences.
This gives rise to a commutative diagram 
\[\begin{tikzcd}[ampersand replacement=\&]
	{\End_{B''}(T(B))} \& {\End_A(H\circ T(B))} \\
	\& {\End_A(F_R(R\otimes_B B))} \\
	{\End_B(B)} \& {\End_R(R\otimes_B B)}
	\arrow["{\varphi\mapsto \id_R\otimes \varphi}", from=3-1, to=3-2]
	\arrow["{\varphi\mapsto T(\varphi)}"', from=3-1, to=1-1]
	\arrow["{\varphi\mapsto H(\varphi)}"', from=1-1, to=1-2]
	\arrow["{\varphi\mapsto F_R(\varphi)}"', from=3-2, to=2-2]
	\arrow["{\rho_{\alpha_B}}"', from=2-2, to=1-2]
\end{tikzcd}\]
where $\alpha_B:F_R(R\otimes_B B)\rightarrow H\circ T(B)$ is the isomorphism obtained from the natural isomorphism $\alpha$ and $\rho_{\alpha_B}$ is conjugation by $\alpha_B$. 
As in Theorem \ref{theorem_conjugation}, we have that the composition
\[\begin{tikzcd}[ampersand replacement=\&]
	{B^{\op}} \& {\End_B(B)} \& {\End_R(R\otimes_B B)} \& {R^{\op}}
	\arrow["{b\mapsto r_b}", from=1-1, to=1-2]
	\arrow["{\id_R\otimes -}", from=1-2, to=1-3]
	\arrow["{r_x\mapsto  x}", from=1-3, to=1-4]
\end{tikzcd}\]
is just the canonical embedding $B^{\op}\rightarrow R^{\op}$, so that $\End_B(B)^{\op}$ is a basic regular exact Borel subalgebra of $\End_R(R\otimes_B B)^{\op}$. 
Thus the diagram above shows that $\End_{B''}(T(B))^{\op}$ is a basic regular exact Borel subalgebra $\End_A(H\circ T(B))^{\op}$.\\
Note that since $B$ and $B''$ are both basic projective generators, and $\psi$ is an equivalence, we have an isomorphism $\beta: T(B)\rightarrow B''$. This gives rise to a commutative diagram
\[\begin{tikzcd}[ampersand replacement=\&]
	{\End_{B''}(B'')} \& {\End_A(H(B''))} \\
	{\End_{B''}(T(B))} \& {\End_A(H\circ T(B))}
	\arrow["{\varphi\mapsto H(\varphi)}", from=1-1, to=1-2]
	\arrow["{\rho_\beta}", from=2-1, to=1-1]
	\arrow["{\rho_{H(\beta)}}", from=2-2, to=1-2]
	\arrow["{\varphi\mapsto H(\varphi)}", from=2-1, to=2-2]
\end{tikzcd}\]
so that $B':=\End_{B''}(B'')^{\op}$ is a basic regular exact Borel subalgebra of $A':=\End_A(H(B''))^{\op}$. Moreover, since $B''$ has a $G$-action as a left $B''$-module, and $H$ is weakly $G$-equivariant, there is an induced $G$-action on $H(B'')$ such that the map
\begin{align*}
	\End_{B''}(B'')\rightarrow \End_A(H(B''))
\end{align*}
is $G$-equivariant. In particular, $B'=\End_{B''}(B'')^{\op}$ is a $G$-invariant basic regular Borel subalgebra of $A'=\End_A(H(B''))^{\op}$.
Finally, since $F_R$ is an equivalence and $R$ is a projective generator in $\modu R$, $H(B'')\cong H(T(B))\cong F_R(R\otimes_B B)\cong F_R(R)$ is a projective generator in $\modu A$.
Thus the functor
\begin{align*}
	\Hom_A(H(B''), -):\modu A\rightarrow \modu A'
\end{align*}
is a Morita equivalence, and by \ref{proposition_GequivariantMorita} this Morita equivalence is weakly $G$-equivariant.
\end{proof}
Note that the $G$-action on $A'$ depends on its construction. Given a $G$-action on $A$, the fact that there is a weakly $G$-equivariant Morita equivalence $F:\modu A\rightarrow \modu A'$ is in general not enough to determine the $G$-action on $A'$. In fact, we later give an example showing that in general, $A'$ may admit other $G$-actions which make $F$ weakly $G$-equivariant, but which do not allow for a basic regular exact Borel subalgebra $B$ in $A'$ such that $g(B)=B$ for all $g\in G$.\\
In other words, if we are given a quasi-hereditary algebra $R$ with a basic regular exact Borel subalgebra $B$ and a $G$-action compatible with the partial order $\leq_R$, then there is not necessarily a different basic exact Borel subalgebra $B'$ in $R$ such that $g(B')=B'$ for all $g\in G$. However, there is a different $G$-action on $R$ making the identity functor $\modu R\rightarrow \modu R$ weakly $G$-equivariant, such that $B$ is invariant under this $G$-action.
\begin{corollary}
    Let $R$ be a quasi-hereditary algebra with a basic regular exact Borel subalgebra $B$ and a $G$-action compatible with the partial order $\leq_R$. Then, there is a $G$-action
    \begin{align*}
        G\times R\rightarrow R, (g, r)\mapsto g\star r
    \end{align*} on $R$ such that the identity functor $\modu R\rightarrow (\modu R, \star)$ is weakly $G$-equivariant, and such that $g\star B=B$ for all $g\in G$.
\end{corollary}
\begin{proof}
    By Theorem \ref{thm_skew}, there is a quasi-hereditary algebra $R'$ with a $G$-action and a $G$-invariant basic regular exact Borel subalgebra $B'$ and a weakly $G$-equivariant Morita equivalence
    \begin{align*}
        \modu R\rightarrow \modu R'
    \end{align*}
    By \cite[Theorem A (4)]{Conde}, $R$ and $R'$ have the same multiplicites of indecomposable projectives, so that there is an isomorphism $\varphi: F(R)\rightarrow R'$. By Lemma \ref{lemma_Gstructuretransport}, structure transport along $F^{-1}$ induces a $G$-action $\star'$ on $R$ such that 
    \begin{align*}
        R'\rightarrow R, a\mapsto F^{-1}(\varphi^{-1} \circ r_a\circ \varphi)(1_R)
    \end{align*}
    is $G$-equivariant. Hence $R$ has a basic regular exact Borel subalgebra $B'$ such that $g\star' B'=B'$ for all $g\in G$. Now by Theorem \ref{theorem_conjugation}, $B'=aBa^{-1}$ for some $a\in R$. Thus, setting
    \begin{align*}
        g\star x:=a^{-1} g\star'(axa^{-1})a
    \end{align*}
    defines a $G$-action on $R$ such that $g\star B=B$ for all $g\in G$.
    Clearly, the identity functor $(\modu R, \star')\rightarrow (\modu R, \star)$ is weakly $G$-equivariant. Moreover, by definition $F^{-1}:\modu R'\rightarrow (\modu R, \star')$ is weakly $G$-equivariant, and by assumption $F: \modu R\rightarrow \modu R'$ is weakly   $G$-equivariant. Hence the composition $\id:\modu R\rightarrow (\modu R, \star)$ is weakly $G$-equivariant.
\end{proof}
The following lemma describes how the $G$-action on an algebra may vary under the condition that the identity functor is weakly $G$-equivariant.
\begin{lemma}\label{lemma_2Gactions}
    Let $A'$ be an associative algebra and assume $G$ acts on $A'$ via automorphisms
    \begin{align*}
        G\times A'\rightarrow A', (g, a)\mapsto g(a).
    \end{align*}
    Now suppose that there is another $G$-action on $A'$ given by
    \begin{align*}
        G\times A'\rightarrow A', (g, a)\mapsto g\star a,
    \end{align*}
    such that the identity functor
    \begin{align*}
        (\modu A', \star)\rightarrow \modu A'.
    \end{align*}
    is weakly $G$-equivariant, where the $G$-action on the codomain is induced by the first $G$-action on $A'$, and the $G$-action on the domain is induced by the second $G$-action on $A'$. 
    Then there is a map
    \begin{align*}
       \rho: G\rightarrow (A')^{\times}
    \end{align*}
    such that for all $a\in A'$ and $g\in G$
    \begin{align*}
      g\star a=\rho(g)g(a)\rho(g)^{-1}.
    \end{align*}
    Moreover, $\rho(e_G)=1_{A'}$ and $\rho(gh)=\rho(g)g(\rho(h))$ for all $g,h\in G$.
    On the other hand, if 
    \begin{align*}
       \rho: G\rightarrow (A')^{\times}
    \end{align*}
    is a map such that $\rho(e_G)=1_{A'}$ and $\rho(gh)=\rho(g)g(\rho(h))$ for all $g,h\in G$, then
    \begin{align*}
      g\star a:=\rho(g)g(a)\rho(g)^{-1}.
    \end{align*}
    defines a $G$-action on $A$ such that the identity functor $\modu A'\rightarrow  (\modu A', \star)$ is weakly $G$-equivariant.
\end{lemma}
\begin{proof}
\begin{enumerate}
    \item Recall that in this setting, there are induced $G$-actions 
    \begin{align*}
         \tr_g:gA\rightarrow A, ga\mapsto g(a).
    \end{align*}
    on $A'\in \modu A'$ with respect to the first $G$-action, and
    \begin{align*}
         \tr_g':g\star A'\rightarrow A', ga\mapsto g\star a.
    \end{align*}
    on $A'\in (\modu A', \star)$ with respect to the second $G$-action.
    Since the identity functor is weakly $G$-equivariant, Lemma \ref{lemma_Gstructuretransport} endows $A'$
    with the structure 
    \begin{align*}
        \tr_g'':gA\rightarrow A, ga\mapsto g*a
    \end{align*}
    of an object with a $G$-action in $\modu A'$, where $\modu A'$ is equipped with the first $G$-action,
    such that the identity map
    \begin{align*}
        \End_A(A')^{\op}\rightarrow \End_A(A')^{\op}
    \end{align*}
    is $G$-equivariant, where the first $G$-action is induced by $\star$ and the second by $*$.
    In other words, we obtain that
    \begin{align*}
        r_{g\star a}=g*(r_a(g^{-1}*-))
    \end{align*}
    for all $a\in A$, $g\in G$. Thus
    \begin{align*}
        g\star a=r_{g\star a}(1_A)=g*(r_a(g^{-1}*1_A))=g*((g^{-1}*1_A)\cdot a)
    \end{align*}
    since $\tr_g''$ is $A$-linear, we have for every $g\in G$ and $a, b\in A$
    \begin{align*}
        g*(ab)=g(a)\cdot (g*b),
    \end{align*}
    so that
     \begin{align*}
        g\star a=g*((g^{-1}*1_A)\cdot a)=g(g^{-1}*1_A)\cdot (g*a)=g(g^{-1}*1_A)\cdot g(a)\cdot (g*1_A).
    \end{align*}
    Therefore, setting $\rho(g)=g(g^{-1}*1_A)$, we obtain
    \begin{align*}
      g\star a=\rho(g)g(a)\rho(g)^{-1}.
    \end{align*}
    Moreover, we clearly have $\rho(e_G)=1_A$, and
    \begin{align*}
        \rho(gh)=gh((gh)^{-1}*1_A)=g(h(h^{-1}*(g^{-1}*1_A)))=g(h(h^{-1}(g^{-1}*1_A)\cdot h^{-1}*1_A))\\
        =g(g^{-1}*1_A)\cdot g(h(h^{-1}*1_A))=\rho(g) g(\rho(h)).
    \end{align*}
    \item Suppose that 
    \begin{align*}
       \rho: G\rightarrow (A')^{\times}
    \end{align*}
    is a map such that $\rho(e_G)=1_{A'}$ and $\rho(gh)=\rho(g)g(\rho(h))$ for all $g,h\in G$. Then for
    \begin{align*}
      g\star a:=\rho(g)g(a)\rho(g)^{-1}.
    \end{align*}
    we have
    \begin{align*}
        e_G\star a=\rho(e_G)e_G(a)\rho(e_G)^{-1}=1_A\cdot a\cdot 1_A^{-1}=a
    \end{align*}
    for every $a\in A$ and
    \begin{align*}
        (gh)\star a=\rho(gh)(a)\rho(gh)^{-1}=\rho(g)g(\rho(h)) g(h(a))g(\rho(h)^{-1})\rho(g)^{-1}\\
        =\rho(g)g(\rho(h)h(x)\rho(h)^{-1})\rho(g)^{-1}=g\star(h\star a)
    \end{align*}
    for all $g,h\in G$, $a\in A$. Hence this defines a $G$-action on $A$.
    Moreover, for every $g\in G$ we have a natural isomorphism $\alpha_g: (g\star -)\rightarrow g(-)$ given by the components
    \begin{align*}
       \alpha_{g, M}: g\star M\rightarrow gM, g\star m\mapsto \rho(g)\cdot gm
    \end{align*}
    such that $\alpha_{e_G}$ is the identity and 
    \begin{align*}
        g(\alpha_{h, M})\circ \alpha_{g, h\star M}(g\star h\star m)&=  g(\alpha_{h, M})(\rho(g)\cdot g\cdot(h\star m))\\
        = g(\alpha_{h, M})(g\cdot (h\star ((h^{-1}\star (g^{-1}(\rho(g))))\cdot m)))&=g\cdot (\rho(h)\cdot h\cdot( (h^{-1}\star (g^{-1}(\rho(g))))\cdot m))\\
        =g(\rho(h))\cdot (gh)\cdot ((h^{-1}\star (g^{-1}(\rho(g))))\cdot m)&=g(\rho(h))\cdot (gh)\cdot((\rho(h^{-1})h^{-1}(g^{-1}(\rho(g)))\rho(h^{-1})^{-1})\cdot m)\\
        =g(\rho(h))\cdot (gh)\cdot(h^{-1}(\rho(h)^{-1}g^{-1}(\rho(g))\rho(h))\cdot m)
        &=\rho(g)g(\rho(h))\cdot (gh)\cdot m\\
        =\rho(gh)\cdot ghm&=\alpha_{gh, M}(m).\qedhere
    \end{align*}
    \end{enumerate}
\end{proof}
We thus obtain the following corollary, describing in which way the $G$-action giving rise to an invariant Borel may vary from the originial $G$-action:
\begin{corollary}
    Let $R$ be a quasi-hereditary algebra with a basic regular exact Borel subalgebra $B$ and a $G$-action compatible with the partial order $\leq_R$. Then, there is a map
    \begin{align*}
        \rho: G\rightarrow R^{\times}
    \end{align*} such that $\rho(e_G)=1_R$ and $\rho(gh)=\rho(g)g(\rho(h))$ for all $g,h\in G$,
    and such that
    \begin{align*}
        g\star a:=\rho(g)g(a)\rho(g)^{-1}
    \end{align*}
    defines a $G$-action on $R$ with $g\star B=B$ for all $g\in G$.
\end{corollary}

\subsection{Example}
In this subsection, we consider an example. Let $D= \field[x]/x^n$ and let $A$ be its Auslander algebra. Note that since $A$ is an Auslander algebra, it is quasi-hereditary.
Moreover, $A$ is given by the quiver
\[\begin{tikzcd}[ampersand replacement=\&]
	1 \& 2 \& \dots \& n
	\arrow["{x_1}", curve={height=-6pt}, from=1-1, to=1-2]
	\arrow["{y_2}", curve={height=-6pt}, from=1-2, to=1-1]
	\arrow["{y_3}", curve={height=-6pt}, from=1-3, to=1-2]
	\arrow["{x_2}", curve={height=-6pt}, from=1-2, to=1-3]
	\arrow["{x_{n-1}}", curve={height=-6pt}, from=1-3, to=1-4]
	\arrow["{y_n}", curve={height=-6pt}, from=1-4, to=1-3]
\end{tikzcd}\]
with relations $y_{i+1}x_i=x_{i-1}y_i$ for all $2\leq i\leq n$ and $x_{n-1}y_n=0$.
To ease notation, we often drop the indices for $x$ and $y$ and write $x^k$ for a composition $x_{i+k}x_{i+k-1}\dots x_{i+1}x_i$ and $y^k$ for a composition $y_{i}y_{i+1}\dots y_{i+k-1}y_{i+k}$, where $1\leq i\leq n$.
For $i, j\in \{1, \dots , n\}$, $0\leq a, b\leq n$ with $b-a=j-i$ let
\begin{align*}
    r_{i, j, a, b}:P_i\rightarrow P_j, m\mapsto my^ax^b.
\end{align*}
Then if $i\geq j$, $$(r_{i, j, i-j, 0}, r_{i, j, i-j+1, 1}, r_{i, j, i-j+2, 2},\dots, r_{i, j, n-j, n-i})$$
is a basis for $\Hom_A(P_i, P_j)$, and if $i\leq j$ then 
$$(r_{i, j, 0, j-i}, r_{i, j, 1, j-i+1}, r_{i, j, 2, j-i+2},\dots, r_{i, j, n-j, n-i})$$
is a basis for $\Hom_A(P_i, P_j)$.
Moreover, if $b>n-i$, then $r_{i,j, a, b}=0$.\\
Additionally, for all $1\leq i, j\leq n$, $0\leq a,b\leq n$, $b-a=j-i$ we have
\begin{align*}
    r_{i,j, a, b}\circ r_{i', j', a', b'}=\delta_{j', i}r_{i, j', a+a', b+b'}
\end{align*}
where $\delta$ is the Kronecker delta.
First, we explicitly describe a Morita equivalent algebra $R$ with a regular exact Borel subalgebra $B$. This was constructed according to the procedure given in \cite{KKO}, although we do not refer to this procedure for the verification that $B$ is indeed a regular exact Borel subalgebra of $R$ and that $R$ is Morita equivalent to $A$. Instead, we proceed by defining $B$, $R$ as well as the embedding $\iota:B\rightarrow R$, giving in either case only a somewhat informal reasoning why we define them as we do, and then explicitly check that they are as claimed.
\subsubsection{Defining $B$}
By \cite[Section 4]{mypaper}, the indecomposable projective modules $P_i^A$ of $A$ are of the form 
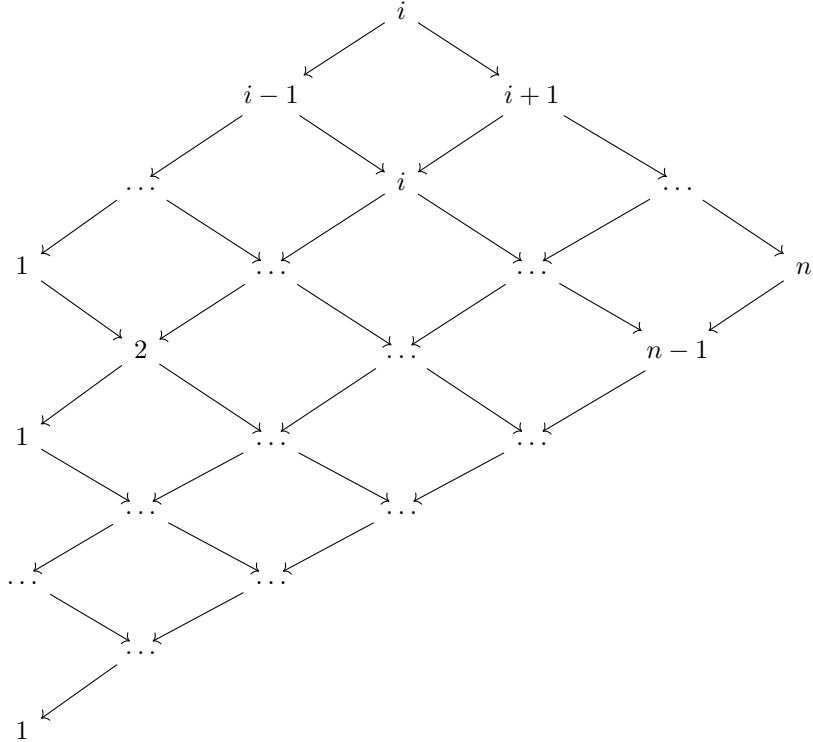
\[\begin{tikzcd}[ampersand replacement=\&]
	\&\&\& i \\
	\&\& {i-1} \&\& {i+1} \\
	\& \dots \&\& i \&\& \dots \\
	1 \&\& \dots \&\& \dots \&\& n \\
	\& 2 \&\& \dots \&\& {n-1} \\
	1 \&\& \dots \&\& \dots \\
	\& \dots \&\& \dots \\
	\dots \&\& \dots \\
	\& \dots \\
	1
	\arrow[from=1-4, to=2-3]
	\arrow[from=2-3, to=3-2]
	\arrow[from=2-5, to=3-4]
	\arrow[from=3-4, to=4-3]
	\arrow[from=1-4, to=2-5]
	\arrow[from=2-5, to=3-6]
	\arrow[from=3-6, to=4-7]
	\arrow[from=4-7, to=5-6]
	\arrow[from=5-6, to=6-5]
	\arrow[from=6-5, to=7-4]
	\arrow[from=6-3, to=7-4]
	\arrow[from=6-3, to=7-2]
	\arrow[from=3-4, to=4-5]
	\arrow[from=4-5, to=5-6]
	\arrow[from=2-3, to=3-4]
	\arrow[from=5-2, to=6-1]
	\arrow[from=4-1, to=5-2]
	\arrow[from=3-2, to=4-1]
	\arrow[from=4-3, to=5-2]
	\arrow[from=5-2, to=6-3]
	\arrow[from=7-2, to=8-1]
	\arrow[from=9-2, to=10-1]
	\arrow[from=8-3, to=9-2]
	\arrow[from=8-1, to=9-2]
	\arrow[from=7-2, to=8-3]
	\arrow[from=7-4, to=8-3]
	\arrow[from=6-1, to=7-2]
	\arrow[from=3-2, to=4-3]
	\arrow[from=4-3, to=5-4]
	\arrow[from=5-4, to=6-3]
	\arrow[from=5-4, to=6-5]
	\arrow[from=4-5, to=5-4]
	\arrow[from=3-6, to=4-5]
\end{tikzcd}\]
 the standard modules $\Delta_i^A$ are of the form 
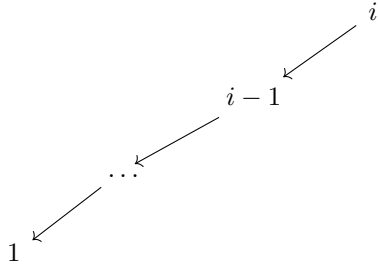
\[\begin{tikzcd}[ampersand replacement=\&]
	\&\&\& i \\
	\&\& {i-1} \\
	\& \dots \\
	1
	\arrow[from=1-4, to=2-3]
	\arrow[from=2-3, to=3-2]
	\arrow[from=3-2, to=4-1]
\end{tikzcd}\] and the projective modules $P_i$ have unique $\Delta$-filtrations 
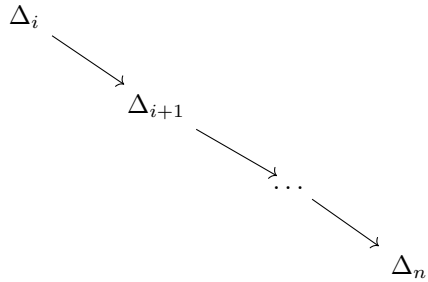
\[\begin{tikzcd}[ampersand replacement=\&]
	{\Delta_i} \\
	\& {\Delta_{i+1}} \\
	\&\& \dots \\
	\&\&\& {\Delta_n}
	\arrow[from=2-2, to=3-3]
	\arrow[from=3-3, to=4-4]
	\arrow[from=1-1, to=2-2]
\end{tikzcd}\]
Here, we have angled the graphic representation of filtration of the standard modules by simple modules and the filtration of the projective modules such that, inserting the former picture into the latter gives exactly the graphic representation of the filtration of the projective modules by simple modules given above.
Additionally, $\Ext_A^k(\Delta_i, \Delta_j)$ is one dimensional if $i<j$ and $k=1$, and given by the extension
$E_{ij}=\begin{pmatrix}
    \Delta_i\\ \Delta_{j}
\end{pmatrix}$ which is the factor module of the image of
\begin{align*}
    (r_{i, j-1, 0, j-i-1}, r_{j, j-1, 1, 0}): P_i\oplus P_j \rightarrow P_{j-1}
\end{align*}
by the image of $r_{j+1, j-1, 2, 0}$.
Moreover, $\Ext_A^k(\Delta_i, \Delta_j)=(0)$ if $i\geq j$ or $k>1$.\\
In particular, by the general construction of $B$ in \cite{KKO}, since there is no $\Ext_A^k(\Delta^A, \Delta^A)$ for $k\geq 2$, there are no relations in $B$, so that $B$ is freely generated over $L$ by $\Ext_A^1(\Delta_i, \Delta_j)^{\op}$.
We can thus define $B$ as follows.
\begin{definition}
    Let $B=kQ'$ where $Q'$ is the quiver with $n$ vertices and an edge $(i,j)$ whenever $i<j$.
\end{definition}
\begin{remark}
    Note that a $\field$-basis of $B$ is given by the $e_i$ and the paths $$(i_{k-1}, i_k)(i_{k-2}, i_{k-1})\dots (i_2, i_3)(i_1, i_2)$$ for $1\leq i_1<\dots <i_k\leq i_n$, $1\leq k\leq n$. For ease of notation, we denote these paths by $(i_1<\dots<i_k)$.
\end{remark}
\subsubsection{Defining $R$}
To describe $R$, we use that $R\cong \End_R(R\otimes_B B)^{\op}$. We thus would first like to describe $R\otimes_B B$, in other words, we would like to find out how 
$Q_i:=R\otimes_B P_i^B$ decomposes into a direct sum of projective modules, using the fact that $R\otimes_B -$ is exact and $R\otimes_B L_i^B\cong \Delta_i^R$ .\\
Therefore, we begin by having a look at the composition series of the projective $B$-modules $P_i^B$.\\
Since $B$ is the path algebra of the full directed graph on $n$ vertices, the Loewy diagram of $P_i^B$ is given by a  labelled Fenwick tree $T_{n, i}$.\\
Recall that the Fenwick tree, sometimes also called a binary indexed tree, $F_k$ is inductively defined by
\begin{enumerate}
    \item $F_1$ consists of a single vertex.
    \item $F_{k}$ is constructed by attaching to the root all Fenwick trees $F_j$ for $j<k$.
\end{enumerate}
Fenwick trees were originally used as a way to compute efficiently sums over subsets of an array $a=(a_k)_{k=1}^K$ of numbers \cite{Ryabko, Fenwick}. In this case, one would associate Fenwick trees $F_1(a), \dots, F_K(a)$ with labelled vertices, where the labels are inductively given by
\begin{enumerate}
    \item The unique vertex of $F_1$ is labelled by $a_1$.
    \item The root of $F_k$ has label $a_k$ and before attaching the subtree $F_k$ to this root, we increase every label therein by $a_k$.
\end{enumerate}
The labelled tree $T_{n, i}$ is isomorphic to $F_{n-i+1}$  as a tree, but its labelling arises in a different way.
More precisely, $(T_{n, i})_{1\leq i\leq n}$ is the family of rooted trees with labeled vertices inductively defined by
\begin{enumerate}
    \item $T_{n, n}=\{n\}$ consists of a single vertex labeled $n$.\\
    \item $T_{n, i-1}$ is given by
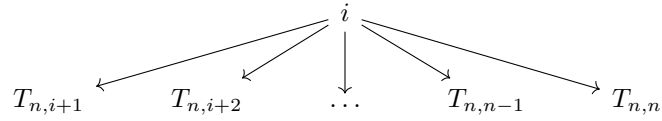
\[\begin{tikzcd}[ampersand replacement=\&]
	\&\& i \\
	{T_{n, i+1}} \& {T_{n, i+2}} \& \dots \& {T_{n, n-1}} \& {T_{n, n}}
	\arrow[from=1-3, to=2-1]
	\arrow[from=1-3, to=2-2]
	\arrow[from=1-3, to=2-3]
	\arrow[from=1-3, to=2-5]
	\arrow[from=1-3, to=2-4]
\end{tikzcd}\]
i.e. by attaching to a root labelled $i$ every tree $T_{n, j}$ for $j>1$.
\end{enumerate}
For example, $T_{n, n-3}$ is given by
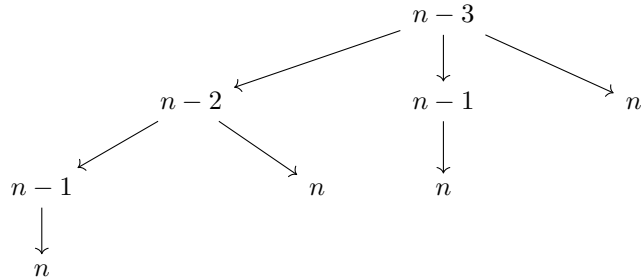
\[\begin{tikzcd}[ampersand replacement=\&]
	\&\&\& {n-3} \\
	\& {n-2} \&\& {n-1} \&\& n \\
	{n-1} \&\& n \& n \\
	n
	\arrow[from=1-4, to=2-2]
	\arrow[from=2-2, to=3-1]
	\arrow[from=3-1, to=4-1]
	\arrow[from=1-4, to=2-4]
	\arrow[from=2-2, to=3-3]
	\arrow[from=2-4, to=3-4]
	\arrow[from=1-4, to=2-6]
\end{tikzcd}\]
Since binary trees are combinatorially easier to describe, we use a bijection from $(T_{n, i})_i$ to certain binary trees $(B_{n, i})_i$, which is given by adding vertices when necessary. More formally, for any planar rooted tree $T$ with labeled vertices we inductively define a corresponding planar rooted binary tree $\Binary(T)$ with labeled vertices by 
\begin{itemize}
    \item $\Binary(T)=T$ if $T$ is empty or consists of one vertex.
    \item If $T$ is a tree of the form 
\[\begin{tikzcd}[ampersand replacement=\&]
	r \\
	{T_1}
	\arrow[from=1-1, to=2-1]
\end{tikzcd}\]
with root $r$ and a single subtree $T_1$ attached to the root, then we define $\Binary(T)$ as the tree 
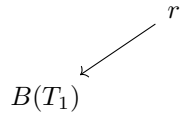
\[\begin{tikzcd}[ampersand replacement=\&]
	\& r \\
	{B(T_1)}
	\arrow[from=1-2, to=2-1]
\end{tikzcd}\]
    \item If $T$ is a tree of the form 
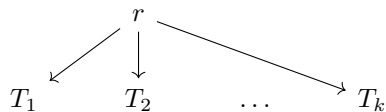
\[\begin{tikzcd}[ampersand replacement=\&]
	\& r \\
	{T_1} \& {T_2} \& \dots \& {T_k}
	\arrow[from=1-2, to=2-1]
	\arrow[from=1-2, to=2-2]
	\arrow[from=1-2, to=2-4]
\end{tikzcd}\]
with root $r$ and subtrees $T_1, \dots, T_k$, $k\geq 2$, attached to the root, then we define $T'$ as the planar rooted tree with labeled vertices given by 
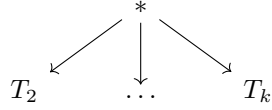
\[\begin{tikzcd}[ampersand replacement=\&]
	\& {*} \\
	{T_2} \& \dots \& {T_k}
	\arrow[from=1-2, to=2-1]
	\arrow[from=1-2, to=2-3]
	\arrow[from=1-2, to=2-2]
\end{tikzcd}\]
and set $\Binary(T)$  to be the tree 
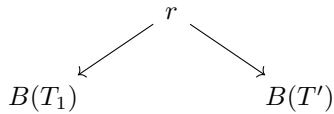
\[\begin{tikzcd}[ampersand replacement=\&]
	\& r \\
	{B(T_1)} \&\& {B(T')}
	\arrow[from=1-2, to=2-1]
	\arrow[from=1-2, to=2-3]
\end{tikzcd}\]
\end{itemize}
Let us view $T_{n, i}$ as planar labeled trees, with the planar representation given as in the definition. Then $\Binary(T_{n, i})$ are the binary trees $B_{n, i}$ such that
\begin{enumerate}
    \item $B_{n, i}$ has $n-i+1$ generations.
    \item Every vertex in the second to last generation has a left child, but no right child.
    \item Every vertex that is not in the last or second to last generation has exactly two children.
    \item The root has label $i$.
    \item Every vertex in the $j$-th generation that is a left child of its predecessor has label $j$.
    \item Every vertex that is a right child of its predecessor has label $*$.
\end{enumerate}
In the original interpretation of Fenwick trees as calculacting iteratively all sums over an array of natural numbers of a given set, the associated binary tree instead corresponds to, in every layer, choosing to include or not include the $n$-th element in the sum.\\
In our case, we have that, for example, $B_{n, n-3}$ is the binary tree
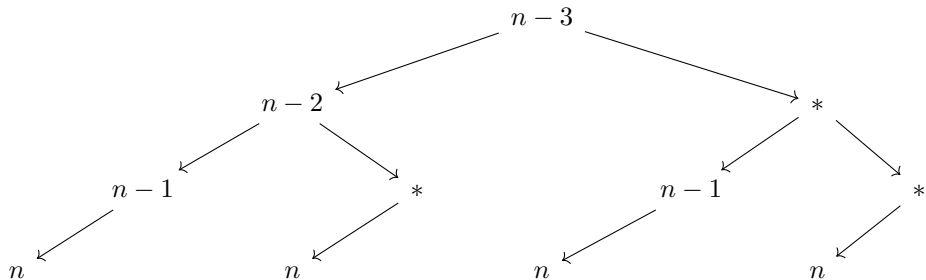
\[\begin{tikzcd}[ampersand replacement=\&]
	\&\&\&\& {n-3} \\
	\&\& {n-2} \&\&\&\& {*} \\
	\& {n-1} \&\& {*} \&\& {n-1} \&\& {*} \\
	n \&\& n \&\& n \&\& n
	\arrow[from=1-5, to=2-3]
	\arrow[from=2-3, to=3-2]
	\arrow[from=3-2, to=4-1]
	\arrow[from=3-6, to=4-5]
	\arrow[from=3-8, to=4-7]
	\arrow[from=1-5, to=2-7]
	\arrow[from=2-7, to=3-6]
	\arrow[from=2-7, to=3-8]
	\arrow[from=2-3, to=3-4]
	\arrow[from=3-4, to=4-3]
\end{tikzcd}\]
While in the Loewy diagram every vertex corresponds to a simple composition factor of $P_i^B$, in $B_{i, n}$ the vertices marked $*$ correspond to  $(0)$ instead, while the vertices marked with a label $j$ still correspond to composition factors $L_j^B$. We can visualize this by replacing the labelled vertices with the composition factors respectively zeros that they contribute. For example, for $P_{n-3}^B$ we obtain
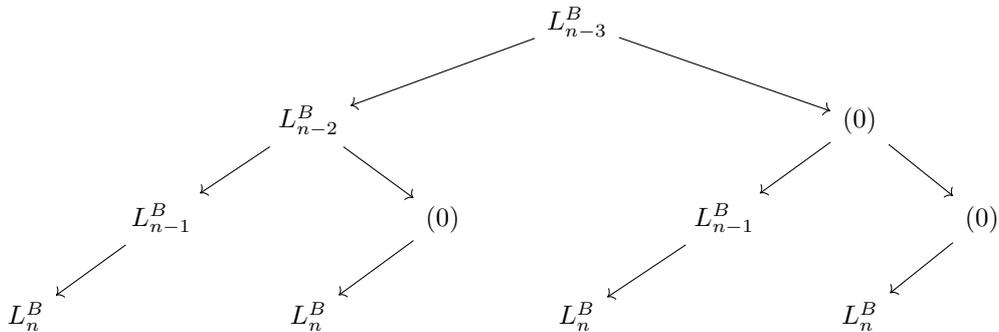
\[\begin{tikzcd}[ampersand replacement=\&]
	\&\&\&\& {L^B_{n-3}} \\
	\&\& {L^B_{n-2}} \&\&\&\& {(0)} \\
	\& {L^B_{n-1}} \&\& {(0)} \&\& {L^B_{n-1}} \&\& {(0)} \\
	{L^B_n} \&\& {L^B_n} \&\& {L^B_n} \&\& {L^B_n}
	\arrow[from=1-5, to=2-3]
	\arrow[from=2-3, to=3-2]
	\arrow[from=3-2, to=4-1]
	\arrow[from=3-6, to=4-5]
	\arrow[from=3-8, to=4-7]
	\arrow[from=1-5, to=2-7]
	\arrow[from=2-7, to=3-6]
	\arrow[from=2-7, to=3-8]
	\arrow[from=2-3, to=3-4]
	\arrow[from=3-4, to=4-3]
\end{tikzcd}\]
We would like to label the paths from the root to a leaf in $B_{n, 1}$ by vectors of length $n$, consisting of numbers $0$ and $1$, such that all entries above $i$ are zero, the $i$-th entry is one and every subsequent entry is one whenever we move into a left subtree and zero whenever we move into a right subtree. Hence the possible set of indices for a tree of the form $B_{n, i}$, $1\leq i\leq n$, is given by 
$I:=\{\alpha\in  \{0,1\}^n: \alpha_n=1\}$. Moreover, if  $s(\alpha):=\min\{1\leq k\leq n: \alpha_k=1\}$ is the first-non-zero entry of $\alpha$, then  $I_i:=\{\alpha\in I: s(\alpha)=i\}$ are the indices corresponding to paths of $B_{n, i}$ for a fixed $i$.\\
Let us consider $Q_i=R\otimes_B P_i^B$. This is a projective $R$-module, and every composition series of $P_i^B$ by simple modules gives rise to a composition series of $R$ by standard modules. We can thus replace the labels in the tree $B_{i, n}$ by the corresponding standard modules to visualize $Q_i$:
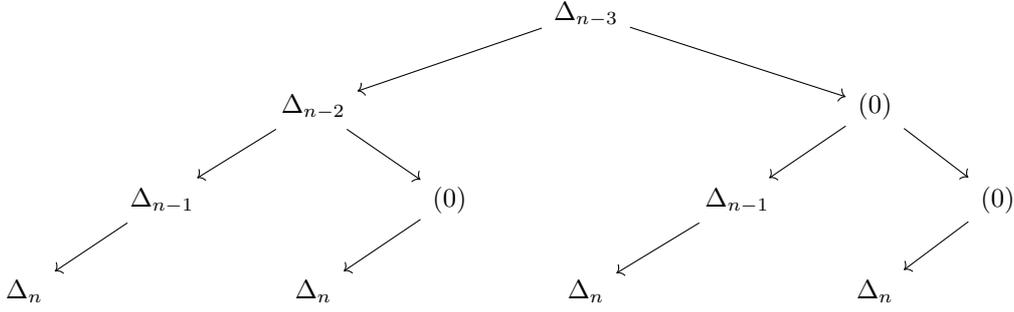
\[\begin{tikzcd}[ampersand replacement=\&]
	\&\&\&\& {\Delta_{n-3}} \\
	\&\& {\Delta_{n-2}} \&\&\&\& {(0)} \\
	\& {\Delta_{n-1}} \&\& {(0)} \&\& {\Delta_{n-1}} \&\& {(0)} \\
	{\Delta_n} \&\& {\Delta_n} \&\& {\Delta_n} \&\& {\Delta_n}
	\arrow[from=1-5, to=2-7]
	\arrow[from=1-5, to=2-3]
	\arrow[from=2-3, to=3-2]
	\arrow[from=3-2, to=4-1]
	\arrow[from=2-3, to=3-4]
	\arrow[from=3-4, to=4-3]
	\arrow[from=2-7, to=3-6]
	\arrow[from=3-6, to=4-5]
	\arrow[from=2-7, to=3-8]
	\arrow[from=3-8, to=4-7]
\end{tikzcd}\]
Note, however, that here some of the extensions shown in the tree might split. In fact, since $Q_i$ is projective and any projective $P_i^R$ of $R$ has a unique composition series by standard modules with composition factors $\Delta_i, \Delta_{i+1}, \dots, \Delta_n$, we see that any extension $\Delta_i\rightarrow  \Delta_j$ for $j\neq i+1$ in $Q_i$ does split. We thus obtain the picture 
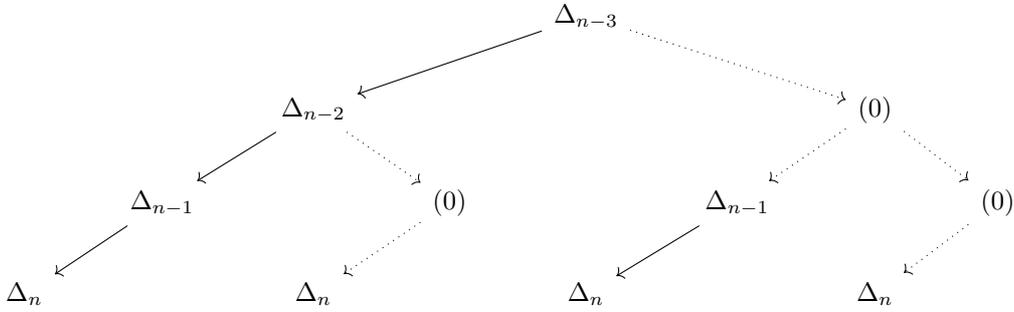
\[\begin{tikzcd}[ampersand replacement=\&]
	\&\&\&\& {\Delta_{n-3}} \\
	\&\& {\Delta_{n-2}} \&\&\&\& {(0)} \\
	\& {\Delta_{n-1}} \&\& {(0)} \&\& {\Delta_{n-1}} \&\& {(0)} \\
	{\Delta_n} \&\& {\Delta_n} \&\& {\Delta_n} \&\& {\Delta_n}
	\arrow[dotted, from=1-5, to=2-7]
	\arrow[from=1-5, to=2-3]
	\arrow[from=2-3, to=3-2]
	\arrow[from=3-2, to=4-1]
	\arrow[dotted, from=2-3, to=3-4]
	\arrow[dotted, from=3-4, to=4-3]
	\arrow[dotted, from=2-7, to=3-6]
	\arrow[from=3-6, to=4-5]
	\arrow[dotted, from=2-7, to=3-8]
	\arrow[dotted, from=3-8, to=4-7]
\end{tikzcd}\]
Here we can see that any path $\alpha$ from the root to a leaf in $B_{n, i}$ gives rise to a projective summand $P_{j(\alpha)}^R$  in $Q_i$ where $j(\alpha):=\max(\{1\leq k\leq n: \alpha_k=0\}\cup\{0\})+1$, consisting of the composition factors corresponding to the part of the path which moves up and to the right from the leaf for as long as possible.\\
We thus define  $P_\alpha:=P_{j(\alpha)}^A\in \modu A$, $Q_i:=\bigoplus_{\alpha\in I_i}P_\alpha$ and $P:=\bigoplus_{i=1}^n Q_i$, and let $R:=\End_A(P)^{\op}$.\\
\subsubsection{The embedding}
    We define an embedding $\iota:B\rightarrow R$ as follows:
    First, denote by  $\varepsilon_i$ is the $i$-th unit vector, and let $\beta_i:=\sum_{j\geq i}\varepsilon_j$.
    Additionally, for $\alpha, \beta\in I$,  $0\leq a, b\leq n$ with $b-a=j(\beta)-j(\alpha)$ let 
    \begin{align*}
        r_{\alpha, \beta, a, b}=r_{j(\alpha), j(\beta), a, b}:P_{\alpha}\rightarrow P_{\beta}.
    \end{align*}
    Then we define
    \begin{align*}
        \iota(e_i)&:=\id_{Q_i}=\sum_{\alpha\in I_i}r_{\alpha, \alpha, 0, 0} \textup{ for }1\leq i\leq n,\\
        \iota((i, j))&:=\sum_{\alpha\in I_j}r_{\alpha, \alpha+\varepsilon_i, 0, 0} \textup{ for }i+1<j\\
        \iota((i, i+1))&:=r_{\beta_{i+1}, \beta_i, 0, 1}+ \sum_{k=i+2}^n r_{\beta_{i+1}, \beta_k+\varepsilon_i, k-i-1, 0}+\sum_{\alpha\in I_{i+1}\setminus{\beta_{i+1}}}r_{\alpha, \alpha+\varepsilon_i, 0, 0}.
    \end{align*}
    In other words,
    \begin{itemize}
        \item for $1\leq i\leq n$, $\iota(e_i)$ is the idempotent corresponding to $Q_i$;
        \item for $1\leq i<i+1<j\leq n$, $\iota((i, j))$ is the embedding of modules $Q_j\rightarrow Q_i$ given, on the corresponding graphs, by mapping $T_{n, j}$ to the subtree $T_{n, j}$ of the root $i$ in $T_{n, i}$ by the identity; and
        \item for $1\leq i<i+1\leq n$,  $\iota((i, i+1))$ consists of a sum of maps from $T_{n, i+1}$ into any subtree  $T_{n, j}$ of the root in $T_{n, i}$ for $i+1\leq j\leq n$.
    \end{itemize}  
    This gives rise to a well-defined algebra homomorphism, since $(\iota_{e_i})_{1\leq i\leq n}$ clearly form a complete set of principal orthogonal idempotents and
    \begin{align*}
        \iota((i, j)) \iota(e_k)=\delta_{i,k}\iota((i,j))=\iota((i,j)e_k)\\
        \iota(e_j)\iota((i, j))=\delta_{j,k}\iota((i,j))=\iota(e_k(i,j)),
    \end{align*}
    where $\delta$ denotes the Kronecker delta.
\subsubsection{The verification}
    Now that we have defined $B$, $R$ and $\iota$, it remains to show that $\iota$ makes $B$ a regular exact Borel subalgebra of $R$. That $R$ is Morita equivalent to $A$ is clear by definition.\\
    Let us begin by introducing a grading on $R$:\\
    \begin{definition}
        For every $\varphi:P_\alpha\rightarrow P_\beta$ in $R$, we set $|\varphi|:=s(\beta)-j(\beta)$.
    This turns $R$ into a $\mathbb{Z}$-graded vector space (however, not into a graded algebra).
    For $r=\sum_{i=1}^k \varphi_i\in R$, $d:=\max\{|\varphi_i|:i=1, \dots k\}$, we call
    \begin{align*}
        \sum_{i \textup{ s.t.} |\varphi_i|=d}\varphi_i
    \end{align*}
    the leading coefficient of $r$.
    \end{definition}
    For every $\alpha\in I$, $j(\alpha)\leq j\leq n$ let 
    \begin{align*}
        f_{\alpha, j}=r_{\alpha, \beta_j, j-j(\alpha), 0}:P_\alpha\rightarrow P_{\beta_j}, m\mapsto my^{j-j(\alpha)}.
    \end{align*}
    Then  for $1\leq k\leq n$,  $1\leq i_1<\dots i_{k-1}<i_k=j$ the leading coefficient of 
    \begin{align*}
        f_{\alpha, j}\cdot \iota((i_1<\dots < i_k))
    \end{align*}    
    is 
    \begin{align*}
        r_{\alpha, \alpha+\sum_{l=1}^{k-1}\varepsilon_{i_l}, j-j(\alpha), j-j\left( \alpha+\sum_{l=1}^{k-1}\varepsilon_{i_l}\right)}:P_\alpha\rightarrow P_{ \alpha+\sum_{l=1}^{k-1}\varepsilon_{i_l}}.
    \end{align*}
    In particular, the leading coefficients of the elements $f_{\alpha, j}\cdot \iota((i_1<\dots < i_k))$, $\alpha\in I$, $j(\alpha)\leq j\leq n$, $1\leq k\leq n$, $1\leq i_1<\dots i_{k-1}<i_k=j $ are linearly independent, so that the elements  $f_{\alpha, j}\cdot \iota((i_1<\dots < i_k))$, $\alpha\in I$, $j(\alpha)\leq j\leq n$, $1\leq k\leq n$, $1\leq i_1<\dots i_{k-1}<i_k=j$ are linearly independent.
    Since $ f_{\alpha, j}\cdot \iota((i_1<\dots < i_k))=0$ for $i_k\neq j$, and $\{(i_1<\dots < i_k)$ for $1\leq k\leq n$, $1\leq i_1<\dots i_{k-1}<i_k=j $ are a $\field$-basis for $B$, we thus obtain that
    \begin{align*}
        R=\bigoplus_{\alpha\in I, j(\alpha)\leq j\leq n}f_{\alpha, j}\iota(B)
    \end{align*}
    and that $f_{\alpha, j}\cong e_j B$ for every $\alpha\in I$, $j(\alpha)\leq j\leq n$.
    In particular, $R$ is projective as a right $B$-module.\\
    Moreover, there is a well-defined right $B$-module homomorphism
    \begin{align*}
       \eta: R\rightarrow B, f_{\beta_j, j}\mapsto e_j, f_{\alpha, j}\mapsto 0 \textup{ if } \alpha\neq \beta_j.
    \end{align*}
    This is clearly a splitting of $\iota$, so that, in particular, $\iota$ is injective and $B$ is a subalgebra of $R$. Moreover,
    \begin{align*}
        \ker(\eta)=\bigoplus_{\alpha\in I,  j(\alpha)\leq j\leq n, \alpha\neq \beta_j }f_{\alpha, j}\iota(B)=\spann_{\field}\{r_{\alpha, \beta, a, b}: a\geq 1\}.
    \end{align*}
    This is a right ideal in $R$, so that $B$ is a normal exact subalgebra of $R$.\\
We want to show that the simple modules of $B$ induce the standard modules of $R$, and that for $i<j$ the unique non-split extension $E_{ij}^B$ between two simples $L_i^B$ and $L_j^B$ of $B$ induces the unique non-split extension  $E_{ij}^R$ between two standard modules  $\Delta_i^R$ and $\Delta_j^R$. These both lead to a similar calculation, which we can generalize to an arbitrary $B$-factor module of $P_i^B$.
Hence suppose $M=BS$ is a $B$-submodule of $P_i^B$ with some generating set $S\subseteq P_i^B$ and consider $R\otimes_B P_i^B/M$.
Then we have
\begin{align*}
    R\otimes_B P_i^B/M&=\End_A(P)\otimes_{\iota} Be_i/BMe_i\\
    &=\{\iota(e_i)\circ f:f\in \End_A(P)\}/\iota(e_i)\circ \iota(M)\\
    &\cong\{\id_{Q_i}\circ f:f\in \End_A(P)\}/(\id_{Q_i}\circ (\field\iota(S))\circ \End_A(P))\\
    &\cong \Hom_A(P, Q_i)/\spann_{\field}\{ \iota(s)\circ r: P\rightarrow Q_i: s\in S, r\in \End_A(P)\}.
\end{align*}
Moreover, since $P$ is projective, every map $f:P\rightarrow Q_i$  whose image is contained in $\sum_{s\in S}\im(s)$ factors through $\bigoplus_{s\in S}s: P^{S}\rightarrow Q_i$, and is thus contained in $\spann_{\field} (\iota(s)\circ r: P\rightarrow Q_i: s\in S, r\in \End_A(P))$.
Hence
\begin{align*}
    R\otimes_B P_i^B/M &\cong \Hom_A(P, Q_i)/\spann_{\field}\{ \iota(s)\circ r: P\rightarrow Q_i: s\in S, r\in \End_A(P)\}\\
    &\cong \Hom_A\left(P, Q_i/\left(\sum_{s\in S}\im(s)\right)\right).
\end{align*}
In particular, for $M=\rad(B)e_i$, we have
\begin{align*}
    R\otimes_B L_i^B=R\otimes_B B/\rad(B)e_i
    \cong \Hom_A(P, Q_i/\left(\sum_{i<j}\im(\iota([i,j]))\right))\\
    \cong \Hom_A(P, P_{\beta_i}/\im(r_{\beta_{i+1}, \beta_i, 0, 1}))
    \cong \Hom_A(P, P_i^A/Ax_i)
    \cong \Hom_A(P, \Delta_i^A)
    \cong \Delta_i^R.
    \end{align*}
    Thus $B$ is a normal exact Borel subalgebra of $R$.\\
    Moreover, consider the extension
\begin{align*}
    E_{ij}^B:= \begin{pmatrix}
        L_i^B\\ L_j^B
    \end{pmatrix}\cong P_i^B/\sum_{(l,k)\neq (i,j)}B(l,k)Be_i
\end{align*}
corresponding to $(i,j)\in B$.
Then, $\sum_{(l,k)\neq (i,j)}B(l,k)Be_i$ is generated as a $B$-module by $(i,k)$, $i<k\neq j$ and $(j, k)(i,j)$, $i<j<k$, so that we have
\begin{align*}
   R\otimes_B E_{ij}\cong  R\otimes_B Be_i/\sum_{(l,k)\neq (i,j)}B(i,k)Be_i
   \cong   \Hom_A\left(P, Q_i/\left(\sum_{i<k\neq j}\im(\iota((i,k)))+\sum_{i<j< k}\im(\iota((j,k)(i,j)))\right)\right).
\end{align*}
Moreover, for $j\neq i+1,$
\begin{align*}
     Q_i/&\left(\sum_{i<k\neq j}\im(\iota((i,k)))+\sum_{i<j< k}\im(\iota((j,k)(i,j)))\right)\\
     &\cong (P_{\beta_i}\oplus P_{\beta_j+\varepsilon_i})/(\im(r_{\beta_{i+1}, \beta_i, 0, 1}+r_{\beta_{i+1}, \beta_j+\varepsilon_i, j-(i+1), 0})+\im(r_{\beta_{j+1}, \beta_j+\varepsilon_i, 0, 1}))\\
     &\cong \begin{pmatrix}\Delta_i^A\\ \Delta_j^A\end{pmatrix}
\end{align*}
and for $j=i+1$,
\begin{align*}
     Q_i/&\left(\sum_{i<k\neq j}\im(\iota((i,k)))+\sum_{i<j< k}\im(\iota((j,k)(i,j)))\right)\\
     &\cong (P_{\beta_i}\oplus P_{\beta_j+\varepsilon_i})/\im(r_{\beta_{i+1}, \beta_i, 0, 1})\\
     &\cong \begin{pmatrix}\Delta_i^A\\ \Delta_{i+1}^A.\end{pmatrix}
\end{align*}
In either case, this is the (up to isomorphism unique) non-split extension $E_{ij}^A$ between $\Delta_i^A$ and $\Delta_j^A$ in $\modu A$.
We thus have an isomorphism
\begin{align*}
    R\otimes_B Be_i/\sum_{k\neq j}B(i,k)Be_i\cong   \Hom_A(P, E_{ij}^A)\cong E_{ij}^R
\end{align*}
where, since $\Hom_A(R, -)$ is an equivalence of quasi-hereditary algebras, $E_{ij}^R$ is the up isomorphism unique extension between $\Delta_i^R$ and $\Delta_j^R$ in $\modu R$.
As mentioned before, there are no higher extensions, neither between simples in $\modu B$ nor between standard modules in $\modu R$. Thus, $B$ is a regular exact Borel subalgebra of $R$.\\
\subsubsection{The $G$-actions}
Let $G=\{e_G, g, \dots, g^{N-1}\}\cong \mathbb{Z}/N\mathbb{Z}$ be a cyclic group such that $\chara(\field)$ does not divide $N$, and let $\xi$ be a primitive $N$-th root of unity in $\field$. Let us consider the $G$-action on $R$ given by $g(y)=\xi y$, $g(x)=x$ and $g(e_i)=e_i$ for $1\leq i\leq n$. \\
To find the $G$-action on $B$, consider the extension between $\Delta_i^A$ and $\Delta^A_j$ for $1\leq i< j$ corresponding to the arrow $(i,j)$ in $B$. For $1\leq k\leq n$ let $P^\cdot(\Delta_k^A)$ be the projective resolution 
\[\begin{tikzcd}[ampersand replacement=\&]
	{(0)} \& {P_{k+1}^A} \& {P_k^A} \& {(0)}
	\arrow["{r_{k+1, k, 1, 0} }", from=1-2, to=1-3]
	\arrow[from=1-1, to=1-2]
	\arrow[from=1-3, to=1-4]
\end{tikzcd}\]
respectively
\[\begin{tikzcd}[ampersand replacement=\&]
	{(0)} \& {P_n^A} \& {(0)}
	\arrow[from=1-2, to=1-3]
	\arrow[from=1-1, to=1-2]
\end{tikzcd}\]
if $k=n$, of $\Delta_k^A$. Then this extension between $i$ and $j$ corresponds in $\Hom_A^1(P^{\cdot}(\Delta_i^A), P^{\cdot}(\Delta_j^A))\cong \Hom_A(P_{i+1}^A, P_j^A)$ to the homomorphism $r_{i+1, j, j-i-1, 0}$. Since $g(r_{i+1, j, j-i-1, 0})=\xi^{j-i-1}r_{i+1, j,j-i-1, 0}$, the induced $G$-action on $B$ is given by
\begin{align*}
    g(e_i)=e_i\textup{ and }\\
    g((i,j))=\xi^{-j+i+1}(i,j)
\end{align*}
for $1\leq i\leq n$ respectively $1\leq i<j\leq n$,
where the change of sign in the exponent is due to the dualizing.\\
Since we have defined $R$ as the opposite of an endomorphism ring in $\modu A$ and not according to its presentation in \cite{KKO}, it is less obvious how the $G$-action on $R$ should be defined. However, if we want it to restrict to the given $G$-action on $B$, there is a natural candidate for it given as follows: \\
Since $G$ fixes the principle orthogonal idempotents $e_i^A$ in $A$, $P_i^A:=Ae_i^A$ naturally obtains the structure of an $A*G$-module.\\
For $\alpha \in I$ let $n(\alpha):=|\{s(\alpha)\leq j\leq n: \alpha_j=0\}|$ be the number of zeros between the first and the last $1$ in the vector $\alpha$. We endow $P_\alpha$ with the structure of an $A*G$-module via 
\begin{align*}
    \tr_g^{P_{\alpha}}:=\xi^{-n(\alpha)}\tr_g^{P_{j(\alpha)}}.
\end{align*}
By Lemma \ref{lemma_Gstructuretransport}, $R=\End_A(P)^{\op}$ obtains thus the structure of an algebra with a $G$-action via
\begin{align*}
    g(\varphi)(x)=g(\varphi(g^{-1}(x)))
\end{align*}
and $\Hom_A(P, -)$ becomes weakly $G$-equivariant.
For $\alpha, \beta\in I$, $0\leq a, b\leq n$ with $b-a=j(\beta)-j(\alpha)$ we have
\begin{align*}
    g(r_{\alpha, \beta, a, b})=\xi^{b-n(\beta)+n(\alpha)}r_{\alpha, \beta, a, b}.
\end{align*}
Hence, for $1\leq i<i+1<j\leq n$
\begin{align*}
    g(\iota((i,j)))&=g\left(\sum_{\alpha\in I_j}r_{\alpha, \alpha+\varepsilon_i, 0, 0}\right)\\
    =\sum_{\alpha\in I_j}\xi^{-n(\alpha+\varepsilon_i)+n(\alpha)}r_{\alpha, \alpha+\varepsilon_i, 0, 0}&=\sum_{\alpha\in I_j}\xi^{-(n(\alpha)+j-i-1)+n(\alpha)}r_{\alpha, \alpha+\varepsilon_i, 0, 0}\\
    =\xi^{-j+i+1}\sum_{\alpha\in I_j}r_{\alpha, \alpha+\varepsilon_i, 0, 0}&=\xi^{-j+i+1}\iota((i,j))=\iota(g((i,j))),
\end{align*}
and for $1\leq i<n$
\begin{align*}
    &g(\iota((i,i+1)))\\
    &=g\left(r_{\beta_{i+1}, \beta_i, 1, 0}+\sum_{j>i+1}r_{\beta_{i+1}, \beta_j+\varepsilon_i, 0, j-i-1}+\sum_{\alpha\in I_{i+1}\setminus \{\beta_{i+1}\}}r_{\alpha, \alpha+\varepsilon_i, 0, 0}\right)\\
   &=g(r_{\beta_{i+1}, \beta_i, 1, 0})+\sum_{j>i+1}g(r_{\beta_{i+1}, \beta_j+\varepsilon_i, 0, j-i-1})+\sum_{\alpha\in I_{i+1}\setminus \{\beta_{i+1}\}}g(r_{\alpha, \alpha+\varepsilon_i, 0, 0})\\
    &=\xi^{-n(\beta_{i+1})+n(\beta_i)}r_{\beta_{i+1}, \beta_i, 1, 0}+\sum_{j>i+1}\xi^{j-i-1-n(\beta_{i+1})+n(\beta_j+\varepsilon_i)}r_{\beta_{i+1}, \beta_j+\varepsilon_i, 0, j-i-1}+\sum_{\alpha\in I_{i+1}\setminus \{\beta_{i+1}\}}\xi^{-n(\alpha+\varepsilon_i)+n(\alpha)}r_{\alpha, \alpha+\varepsilon_i, 0, 0}\\
     &=r_{\beta_{i+1}, \beta_i, 1, 0}+\sum_{j>i+1}r_{\beta_{i+1}, \beta_j+\varepsilon_i, 0, j-i-1}+\sum_{\alpha\in I_{i+1}\setminus \{\beta_{i+1}\}}r_{\alpha, \alpha+\varepsilon_i, 0, 0}\\
     &=\iota((i,i+1))=\iota(g((i,i+1))).
\end{align*}
Finally, for $1\leq i\leq n$
\begin{align*}
    g(\iota(e_i))=\sum_{\alpha\in I_i}g(\id_{\alpha})=\sum_{\alpha\in I_i}g(r_{\alpha, \alpha, 0, 0})\\=\sum_{\alpha\in I_i}r_{\alpha, \alpha, 0, 0}=\iota(e_i)=\iota(g(e_i)).
\end{align*}
Hence $\iota$ is $G$-equivariant, and in particular, $B$ is a $G$-invariant regular exact Borel subalgebra of $R$.\\
Note that the $G$-action on $R$ does not coincide with the perhaps more canonical $G$-action where one instead endows $P_\alpha$ simply with the same $G$-action as the one on $P_{j(\alpha)}$, as the latter would give
\begin{align*}
     g(r_{\alpha, \beta, a, b})=\xi^{b}r_{\alpha, \beta, a, b}.
\end{align*}
The next example was added to illuminate this phenomenon.
\subsection{A Second Example}
In the previous section, the $G$-action on $R$ was chosen carefully to restrict to the $G$-action on $B$. The aim of this subsection is to show that this was indeed necessary. In other words, we give an example of a quasihereditary algebra $R$ with a regular exact Borel subalgebra $B$ and a $G$-action on $R$ compatible with the partial order $\leq_R$ such that $R$ does not have a regular exact Borel subalgebra $B'$ with $g(B')=B'$ for all $g\in G$.\\
Throughout, let $\chara(\field)\neq 2$.
Consider $A=\field Q$, where $Q$ is the quiver
\[\begin{tikzcd}[ampersand replacement=\&]
	1 \& 2 \& 3
	\arrow["\alpha", from=1-1, to=1-2]
	\arrow["\beta"', from=1-3, to=1-2]
\end{tikzcd}\]
with the natural order $1<_A 2<_A 3$. Let $G=\{1_G, g\}\cong \mathbb{Z}/2\mathbb{Z}$ and define
\begin{align*}
    g(e_i)=e_i \textup{ for }1\leq i\leq 3, \quad
    g(\alpha)=-\alpha, \quad
    g(\beta)=\beta.
\end{align*}
Since this is hereditary, it is in particular quasi-hereditary with the natural order $1<_A 2<_A 3$, and since $G$ fixes all idempotents, $G$ is compatible with $\leq_A$.
Let $P_i=Ae_i$ for $1\leq 2\leq 3$ and let $P_3'\cong P_3$. Let $P:=P_1\oplus P_3'\oplus P_2\oplus P_3$ and $R:=\End_A(P)^{\op}$. Then $R$ has a $\field$-basis consisting of the elements $\id_{P_1}, \id_{P_2}, \id_{P_3}, \id_{P_3}, \id_{P_3'}$ as well as $f_{33'}:P_3\rightarrow P_3'$ an isomorphism, $f_{3'3}:=f_{33'}^{-1}$,
\begin{align*}
    f_{21}:P_2\rightarrow P_1, a\mapsto a\alpha\\
    f_{23}:P_2\rightarrow P_3, a\mapsto a\beta
\end{align*}
and $f_{23'}:=f_{33'}\circ f_{23}$.\\
Let $B=\field Q_B$, where $Q_B$ is the quiver
\[\begin{tikzcd}[ampersand replacement=\&]
	1 \& 2 \& 3
	\arrow["{\alpha'}"', from=1-1, to=1-2]
	\arrow["{\beta'}", curve={height=-12pt}, from=1-1, to=1-3]
\end{tikzcd}\]
and define $\iota:B\rightarrow R$ via 
\begin{equation*}
    \iota(e_1)=\id_{P_1}+\id_{P_3'},\quad
    \iota(e_i)=\id_{P_i}\textup{ for }i=2, 3,\quad
    \iota(\alpha')=f_{21}+f_{23'},\quad
    \iota(\beta')=f_{33'}.
\end{equation*}
It is easy to see that $\iota$ is an algebra homomorphism and that it turns $B$ into a regular exact Borel subalgebra of $R$, see also \cite[Example 52]{markusthuresson}.\\
Since the $G$-action on $A$ fixes $e_i$ for $1\leq i\leq 3$, $P_i$ obtains the structure of an $A*G$-module for $1\leq i\leq 3$. Via structure transport along $f_{33'}$, this induces an $A*G$-module structure on $P_3'$ and thus on $P=P_1\oplus P_{3}'\oplus P_2\oplus P_3$. Thus, $R=\End_A(P)^{\op}$ obtains a $G$-action, such that $\Hom_A(P, -)$ is weakly $G$-equivariant. Explicitly, this $G$-action is given by
\begin{align*}
    g(\id_X)&=\id_X \textup{ for }X\in \{P_1, P_2, P_3, P_3'\}\\
    g(f_{21})&=-f_{21}\\
    g(f_{ij})&=f_{ij} \textup{ for }(i,j)\neq (2,1).
\end{align*}
We want to show that with this $G$-action, $R$ does not have a regular exact Borel subalgebrs $B'$ such that $h(B')=B'$ for all $h\in G$. By Theorem \ref{theorem_conjugation}, any such $B'$ would be of the form $aBa^{-1}$ for some $a\in R^\times$. Hence, the $G$-action
\begin{align*}
    G\times A\rightarrow A, (h,a')\mapsto h*a':=ah(a^{-1}a'a)a^{-1}
\end{align*}
would fix $B$. Moreover, we would have natural isomorphisms
\begin{align*}
    \alpha_h:h*-\rightarrow h, h*M\rightarrow hM, h*m\mapsto h\cdot(h^{-1}(a)a^{-1}m)
\end{align*}
that would make the identity $\modu A\rightarrow (\modu A, *)$ weakly $G$-equivariant.\\
Therefore, we first classify the $G$-actions $(h,x)\mapsto h*x$ on $R$ making  the identity functor weakly $G$-equivariant and fulfilling $h*B=B$ for all $h\in G$.\\
Suppose we have such a $G$-action
\begin{align*}
    G\times A\rightarrow A, (h,a')\mapsto h*a'.
\end{align*}
By Lemma \ref{lemma_2Gactions}, we have a map
\begin{align*}
    \rho: G\rightarrow R^{\times}
\end{align*}
such that $\rho(e_G)=1_R$ and $\rho(hh')=\rho(h)h(\rho(h'))$ for all $h,h'\in G$, and such that
\begin{align*}
    h*a'=\rho(h)h(a')\rho(h)^{-1}
\end{align*}
for all $h\in G$, $a'\in R$.
Since $G$ only has one generator, $\rho$ is determined by $\rho(g)\in A^{\times}$, and 
\begin{align*}
    1_A=\rho(e_G)=\rho(g^2)=\rho(g)g(\rho(g)).
\end{align*}
Hence $\rho(g)^{-1}=g(\rho(g))$.
Write $\rho(g)=r_0+r_1$ for $r_0\in \spann_{\field}( \{\id_{P_i}: (1\leq i\leq 3)\}\cup \{\id_{P_3'}, f_{33'}, f_{3'3}\})$ and $r_1\in \rad(R)=\spann(f_{21}, f_{23}, f_{23'})$. Then we have that $r_0$ is invertible and, since $\rad(R)^2=(0)$,
\begin{align*}
    \rho(g)^{-1}=r_0^{-1}-r_0^{-2}r_1=g(r_0+r_1)=r_0+g(r_1)
\end{align*}
we obtain that $r_0=r_0^{-1}$ and $g(r_1)=-r_0^{-2}r_1=-r_1$. In particular, $r_1=\lambda_{21}f_{21}$ for some $\lambda_{21}\in \field$. Moreover, writing $r_0=\varepsilon_1 \id_{P_1}+\varepsilon_2\id_{P_2}+m$ for $m=m_1\id_{P_3}+m_2f_{3'3}+m_3f_{33'}+m_4\id_{P_3'}\in \spann_{\field}(  \id_{P_3} , \id_{P_3'}, f_{33'}, f_{3'3}) $ gives $\varepsilon_i^2=1$, i.e. $\varepsilon_i\in \{-1, 1\}$, and $m^2=\id_{P_3}+\id_{P_3'}$.\\
Recall that
\begin{align*}
B=\spann_{\field}( \id_{P_1}+\id_{P_3'}, \id_{P_2}, \id_{P_3}, f_{21}+f_{23'}, f_{33'})  \\
g(B)=\spann_{\field}( \id_{P_1}+\id_{P_3'}, \id_{P_2}, \id_{P_3}, -f_{21}+f_{23'}, f_{33'}).
\end{align*}
In particular, by assumption,
\begin{align*}
     \rho(g)\id_{P_2}\rho(g)^{-1}=\rho(g)^{-1}\circ \id_{P_2}\circ \rho(g)=(\varepsilon_2 \id_{P_2}-\lambda_{21}f_{21})\circ \varepsilon_2\id_{P_2}\\=\id_{P_2}-\varepsilon_2\lambda_{21}f_{21}\in B.
\end{align*}
 Thus $\lambda_{21}=0$. Moreover,
\begin{align*}
     \rho(g)f_{33'}\rho(g)^{-1}=\rho(g)^{-1}\circ f_{33'}\circ \rho(g)
     =(m_4 \id_{P_3'}+m_2f_{3'3})\circ f_{33'}\circ (m_1\id_{P_3}+m_2f_{3'3})\\=m_1m_4f_{33'}+m_1m_2\id_{P_3}+m_2m_4\id_{P_3'}+m_2^2f_{3'3}\in B
\end{align*}
so that $m_2m_4=m_2^2=0$. Since $m_2\in \field$, this implies $m_2=0$; and since $m^2=\id_{P_3}+\id_{P_3'}$, we obtain $m_1^2=m_4^2=1$ and $m_1m_3+m_3m_4=0$.
Additionally
\begin{align*}
     \rho(g)(-f_{21}+f_{23'})\rho(g)^{-1}&=\rho(g)^{-1}\circ(-f_{21}+f_{23'})\circ \rho(g)\\
     &=-\rho(g)^{-1}\circ f_{21}\circ \rho(g)+\rho(g)^{-1}\circ f_{23'}\circ \rho(g)\\
     &=-\varepsilon_1\varepsilon_2f_{21}+\varepsilon_2 m_4f_{23'}\\
     &=-\varepsilon_2(\varepsilon_1f_{21}-m_4f_{23'}).
\end{align*}
Hence $m_4=-\varepsilon_1$. Since $m_1^2=m_4^2=1$ and $m_1m_3+m_3m_4=0$ we thus obtain that either $m_1=m_4=-\varepsilon_1$ and $m_3=m_2=0$, or $m_1=\varepsilon_1$, $m_4=-\varepsilon_1$, $m_2=0$ and $m_3\in \field$. Therefore, we have
\begin{align*}
    \rho(g)=\varepsilon_1\id_{P_1}+\varepsilon_2\id_{P_2}-\varepsilon_1 \id_{P_3}-\varepsilon_1\id_{P_3'}
\end{align*}
or
\begin{align*}
    \rho(g)=\varepsilon_1\id_{P_1}+\varepsilon_2\id_{P_2}+\varepsilon_1 \id_{P_3}+m_3f_{33'}-\varepsilon_1\id_{P_3'}
\end{align*}
for some $\varepsilon_1, \varepsilon_2\in \{-1, 1\}$ and $m_3\in \field$.
In both cases, we can multiply $\rho(s)$ with $\varepsilon_1$ without changing the corresponding $G$-action, and thus obtain the cases
\begin{align*}
    \rho(g)=\id_{P_1}+\varepsilon\id_{P_2}-\id_{P_3}-\id_{P_3'}
\end{align*}
and
\begin{align*}
    \rho(g)=\id_{P_1}+\varepsilon\id_{P_2}+\id_{P_3}+\lambda f_{33'}-\id_{P_3'}
\end{align*}
for some $\varepsilon\in \{-1, 1\}$ and $\lambda\in \field$.
We can calculate the induced $G$-action on $R$:
In the first case, we obtain
\begin{align*}
    g*\id_X&=\id_{X} \textup{ for all }X\in \{P_1, P_2, P_3, P_3'\}\\
    g*f_{ij}&=f_{ij} \textup{ for }(i,j)\in \{(3,3'), (3',3)\}\\
    g*f_{ij}&=-\varepsilon f_{ij} \textup{ for }(i,j)\in \{(2,1), (2,3), (2,3')\}
\end{align*}
and in the second case we obtain
\begin{align*}
    g*\id_{P_i}=\id_{P_i} \textup{ for }i=1,2,\quad
    g*\id_{P_3}&=\id_{P_3}+\lambda f_{33'},\quad
    g*\id_{P_3'}=\id_{P_3'}-\lambda f_{33'},\\
    g*f_{21}=-\varepsilon f_{21},\quad
    g*f_{33'}&=-f_{33'},\quad
    g*f_{3'3}=-f_{3'3}-\lambda\id_{P_3'}+\lambda\id_{P_3},\\
    g*f_{23}=\varepsilon f_{23}+\varepsilon \lambda f_{23'}&,\quad
    g*f_{23'}=-\varepsilon f_{23'}.
\end{align*}
It is easy to see that these $G$-actions indeed fulfill $g*B=B$.
Hence there is a regular exact Borel subalgebra $B'$ of $R$ such that $h(B')=B'$ for all $h\in G$, if and only if there is an $a\in R^\times$ such 
\begin{align*}
    g*x=a g(a^{-1}xa)a^{-1}
\end{align*}
for all $x\in R$, and $*$ one of the above $G$-actions.
Therefore, suppose there is such an $a\in R^\times$ and consider the linear maps
\begin{align*}
    \rho_a:&R\rightarrow R, x\mapsto axa^{-1}\\
    g*-:&R\rightarrow R, x\mapsto g*x\textup{ and }\\
    g(-):&R\rightarrow R, x\mapsto g(x).
\end{align*}
Then we have, by assumption, $g*-= \rho_a\circ g(-)\circ \rho_a^{-1}$. Since
the characteristic polynomial is invariant under base change, this implies that the characteristic polynomial of $g*-$ is the same as the characteristic polynomial of $g(-)$.
However, the characteristic polynomial of $g(-)$ is $(t-1)^8(t+1)$, while the characteristic polynomial of $g*-$ is in the first case given by $(t-1)^6(t+\varepsilon)^3$ and in the second case given by $(t-1)^4(t+1)^2(t+\varepsilon)^2(t-\varepsilon)$, and these cannot be identical for any $\varepsilon\in \{-1,1\}$.
\section{Acknowledgements}

First and foremost, the author would like to thank her advisor Julian Külshammer for many helpful discussions. Moreover, the author would like to thank Teresa Conde and Steffen König for their valuable feedback and for sharing their results prior to their publication.
\bibliographystyle{plain}
\bibliography{uniqueness}
\end{document}